\documentclass[11pt,a4paper,review]{article}

\newif\ifsiamart
\makeatletter%
\@ifclassloaded{siamart171218}%
  {\siamarttrue}%
  {\siamartfalse}%
\makeatother%

\ifsiamart\else
\usepackage{authblk}
\fi
\usepackage{silence}
\usepackage[T1]{fontenc}
\usepackage[utf8]{inputenc}
\usepackage{csquotes}
\usepackage[UKenglish]{babel}
\usepackage{paralist}
\usepackage[shortlabels]{enumitem}
\usepackage[margin=.6in]{geometry}
\usepackage[dvipsnames]{xcolor}
\usepackage{graphicx}
\usepackage{bbm}
\usepackage{amsmath}
\usepackage{amssymb}
\usepackage{amsfonts}
\usepackage{stmaryrd}
\usepackage{mathrsfs}
\usepackage{mathtools}
\usepackage{epstopdf}
\usepackage{comment}
\usepackage{booktabs}

\usepackage{xspace}
\usepackage{float}
\usepackage{stackengine}

\ifsiamart\else
\usepackage{amsthm}
\usepackage{hyperref}
\hypersetup{%
    linktocpage=true,
    colorlinks=true,
    citecolor=red,
    linkcolor=blue,
    urlcolor=blue,
}
\usepackage[nameinlink,capitalise]{cleveref}
\newcommand{\email}[1]{\href{mailto:#1}{#1}}
\newenvironment{keywords}{\small \noindent\begin{quote}\textbf{Keywords.}}{\end{quote}}
\newenvironment{AMS}{\vspace{.2cm}\small \noindent \begin{quote}\textbf{MSC codes.}}{\end{quote}}
\fi

\usepackage{setspace}
\DeclarePairedDelimiter\myabs{\lvert}{\rvert}
\newcommand{\norm}[1]{\left\lVert#1\right\rVert}

\DeclareMathOperator{\I}{I}

\DeclareMathOperator{\diag}{diag}

\DeclareMathOperator{\Law}{Law}
\DeclareMathOperator{\e}{e}

\DeclareMathOperator*{\mytrace}{tr}

\newcommand{\hessian}{\operatorname{D}^2}
\newcommand{\placeholder}{\mathord{\color{black!33}\bullet}}%
\newcommand{\range}[2]{\llbracket #1, #2 \rrbracket}
\newcommand{\normal}{\mathcal{N}}
\renewcommand{\d}{\mathrm d}

\newcommand{\wm}{\mathcal M_{\beta}}
\newcommand{\wc}{\mathcal C_{\beta}}

\newcommand{\N}{\mathbf{N}}
\newcommand{\R}{\mathbf R}
\newcommand{\real}{\R}
\newcommand{\nat}{\N}
\newcommand{\proba}{\mathbf P}
\newcommand{\expect}{\mathbf{E}}
\renewcommand{\t}{\mathsf T}

\renewcommand{\leq}{\leqslant}
\renewcommand{\geq}{\geqslant}
\renewcommand{\le}{\leqslant}
\renewcommand{\ge}{\geqslant}

\newcommand{\particleNumber}[1]{(#1)}
\newcommand{\xn}[1]{X^{\particleNumber{#1}}}
\newcommand{\yn}[1]{Y^{\particleNumber{#1}}}
\newcommand{\xl}{\overline X}
\newcommand{\xnl}[1]{\overline X^{\particleNumber{#1}}}
\newcommand{\mfldis}{\overline{\rho}}
\newcommand{\wien}{W} %
\newcommand{\wienn}[1]{\wien^{\particleNumber{#1}}}

\ifsiamart
\newtheorem{remark}[theorem]{Remark}
\newtheorem{assumption}{Assumption}

\else
\theoremstyle{plain}

\newtheorem{theorem}{Theorem}
\newtheorem{lemma}[theorem]{Lemma}
\newtheorem{corollary}[theorem]{Corollary}
\newtheorem{proposition}[theorem]{Proposition}

\newtheorem{remark}[theorem]{Remark}

\newtheorem{assumption}{Assumption}

\fi

\crefname{lemma}{Lemma}{Lemmas}
\crefname{remark}{Remark}{Remarks}
\crefname{assumption}{Assumption}{Assumptions}
\crefname{proposition}{Proposition}{Propositions}
\crefname{section}{Section}{Sections}
\crefname{subsection}{Subsection}{Subsections}
\crefname{equation}{}{}
\Crefname{equation}{Equation}{Equations}
\newlist{lemmaenum}{enumerate}{3}
\setlist[lemmaenum]{label=(\alph*),ref=\,(\alph*)}
\crefname{lemmaenum}{Lemma}{Lemmas}
\newlist{assumpenum}{enumerate}{5}
\setlist[assumpenum]{label=(\alph*), font={\bfseries}}
\newlist{auxenum}{enumerate}{2}
\setlist[auxenum]{label=(\alph*),ref=(\alph*)}
\crefname{auxenumi}{Item}{Items}
\crefname{enumi}{}{}
\crefname{equation}{}{}
\crefname{assumpenumi}{}{}
\crefname{assumpenumii}{}{}
\Crefname{assumpenumi}{Assumption}{Assumptions}
\Crefname{assumpenumii}{Assumption}{Assumptions}
\Crefname{assumpenumii}{Assumption}{Assumptions}
\crefrangeformat{assumpenumi}{#3#1#4--#5#2#6}
\Crefname{lemmaenumi}{Part}{Parts}
\Crefname{figure}{Figure}{Figures}
\numberwithin{equation}{section}
\numberwithin{theorem}{section}

\AddToHook{env/lemma/begin}{\crefalias{theorem}{lemma}}
\AddToHook{env/proposition/begin}{\crefalias{theorem}{proposition}}

\usepackage{tikz}
\usetikzlibrary{shapes.misc}
\usetikzlibrary{positioning}

\usepackage[url=true,backend=biber,style=trad-abbrv,url=false,doi=false,isbn=false]{biblatex}
\DeclareFieldFormat{volume}{volume \textbf{#1}}
\DeclareFieldFormat[article]{volume}{\textbf{#1}}

\ifsiamart\else
\let\oldparagraph=\paragraph
\renewcommand\paragraph[1]{\oldparagraph{#1.}}
\fi

\newcommand{\empLetter}{\mu}
\newcommand{\emp}[1]{\empLetter^{#1}} %

\newcommand{\cbodifffunc}{\ensuremath{S}}

\definecolor{darkred}{rgb}{.7,0,0}
\definecolor{darkgreen}{rgb}{.1,.7,0}

\newcommand{\wasserstein}{\mathcal W}

\begin{document}
\title{Mean-field limits for Consensus-Based Optimization and Sampling}

\ifsiamart
    \author{%
        Nicolai Jurek Gerber\thanks{%
            Hausdorff Center for Mathematics, Rheinische Friedrich-Wilhelms-Universität, Bonn, Germany (before September 2023),
            Institute of Applied Analysis, Ulm University, Germany (since October 2024)
            (\email{nicolai.gerber@uni-bonn.de}).
        } \and
        Franca Hoffmann\thanks{%
            Department of Computing and Mathematical Sciences, Caltech, USA (\email{franca.hoffmann@caltech.edu}).
        } \and
        Urbain Vaes\thanks{%
            MATHERIALS, Inria Paris \& CERMICS, \'Ecole des Ponts, France (\email{urbain.vaes@inria.fr})
        }
    }
\else
    \author[1,2]{Nicolai Jurek Gerber$^{a,}$}
    \author[3]{Franca Hoffmann$^{b,}$}
    \author[4,5]{Urbain Vaes$^{c,}$}
    \affil[ ]{\footnotesize $^a$\email{nicolai.gerber@uni-bonn.de},
        $^b$\email{franca.hoffmann@caltech.edu},
        $^c$\email{urbain.vaes@inria.fr}}
    \affil[1]{\footnotesize Hausdorff Center for Mathematics, Rheinische Friedrich-Wilhelms-Universität, Bonn, Germany (before September 2023)}
    \affil[2]{\footnotesize Institute of Applied Analysis, Ulm University, Germany (since October 2024)}
    \affil[3]{\footnotesize Department of Computing and Mathematical Sciences, Caltech, USA}
    \affil[4]{\footnotesize MATHERIALS project-team, Inria Paris}
    \affil[5]{\footnotesize CERMICS, \'Ecole des Ponts, France}
    \date{\today}
\fi

\maketitle

\begin{abstract}
    For algorithms based on interacting particle systems that admit a mean-field description, convergence analysis is often more accessible at the mean-field level. In order to transfer convergence results obtained at the mean-field level to the finite ensemble size setting, it is desirable to show that the particle dynamics converge in an appropriate sense to the corresponding mean-field dynamics. In this paper, we prove quantitative mean-field limit results for two related interacting particle systems: Consensus-Based Optimization and Consensus-Based Sampling. Our approach requires a generalization of Sznitman’s classical argument: in order to circumvent issues related to the lack of global Lipschitz continuity of the coefficients, we discard an event of small probability, the contribution of which is controlled using moment estimates for the particle systems. In addition, we present new results on the well-posedness of the particle systems and their mean-field limit, and provide novel stability estimates for the weighted mean and the weighted covariance.
\end{abstract}

\begin{keywords}
    Mean-field limits,
    Interacting particle systems,
    Consensus-Based Optimization,
    Consensus-Based Sampling,
    Coupling methods,
    Wasserstein stability estimates.
\end{keywords}

\begin{AMS}
    35Q93, %
    65C35, %
    70F45, %
    35K55. %
\end{AMS}

\section{Introduction}
Particle methods provide a powerful paradigm for solving complex global optimization problems leading to highly parallelizable algorithms.
Despite widespread and growing adoption, theory underpinning their behavior has been mainly based on meta-heuristics. In application settings involving black-box procedures, or where gradients are too costly to obtain, one relies on derivative-free approaches instead.

This work focuses on the rigorous derivation of quantitative mean-field limits for the recently proposed derivative-free algorithms Consensus-Based Optimization (CBO) and Consensus-Based Sampling (CBS), as well as their variants.
The former, initially proposed in~\cite{CBO}, is an algorithm for global optimization tasks,
while the latter, proposed in~\cite{CBS-Carrillo2021},
can be employed for both sampling and optimization purposes.
Both methods are based on an interacting particle system with only locally,
not globally, Lipschitz coefficients and interaction kernels,
which makes it challenging to rigorously prove propagation of chaos.

CBO and CBS methods are a novel family of algorithms providing an alternative to the popular particle swarm optimization (PSO) algorithm \cite{Kennedy_Eberhart, Kennedy1997ThePS} that is found in many optimization toolboxes.
Recent studies have shown that CBO is competitive with PSO for a number of benchmark functions~\cite{totzeck2018numerical},
and able to solve high-dimensional optimization problems arising from machine learning~\cite{CBO-high-dim,MR4329816}.
For a detailed comparison between CBO and PSO,
and additional context on the connections between the two methods,
we refer the reader to the articles~\cite{CBO-PSO-zero-inertia-1,grassi2023mean,huang2023global}.

By now, CBO and CBS have also been adapted for various other tasks and application settings,
including finance \cite{MR4337731}, rare event simulation~\cite{2023arXiv230409077A} and federated learning \cite{CBO-federal-carrillo2023},
to mention just a few.
Compared to~PSO however, this family of algorithms is more amenable to theoretical analysis.
Since current theoretical guarantees~for CBO and CBS are only available at the level of the mean-field description,
providing quantitative estimates for the distance between the evolution of the interacting particles and the corresponding mean-field probability distribution is essential.

The rest of this introductory section is organized as follows: we first introduce Consensus-Based Optimization (\cref{sub:CBO-intro}) and Consensus-Based Sampling (\cref{sub:CBS-intro}),
giving pointers to the relevant literature. In~\cref{sub:lit_rev}
we present a brief review of the literature on the use of coupling methods to prove propagation of chaos
for models with non-globally Lipschitz coefficients.
Finally,
we summarize the main contributions of this work in~\cref{sub:contrib},
and introduce the notation used throughout the paper in~\cref{sub:notation}.

Since Consensus-Based Sampling can be viewed as a variant on Consensus-Based Optimization,
we focus predominantly on the former and then briefly discuss the latter.

\subsection{Consensus-Based Optimization (CBO)}\label{sub:CBO-intro}
Given an objective function $f\colon \R^d \rightarrow \R$~to be minimized,
we consider the following version of Consensus-Based Optimization~(CBO):
\begin{align}
    \label{eq:cbo-particles}
    \d \xn{j}_t & = -\left(\xn{j}_t - \wm\left( \emp{J}_t\right)\right) \, \d t +  \cbodifffunc\left(\xn{j}_t - \wm\left( \emp{J}_t\right) \right) \, \d \wienn{j}_t,
    \qquad t\ge0, \qquad j \in \range{1}{J},
\end{align}
where $J>1$ is a natural number, $\range{1}{J} := \{1, \dots, J\}$ and $\emp{J}_t := \frac{1}{J} \sum_{i=1}^J \delta_{\xn{i}_t}$ denotes the empirical measure associated with the particle system.
For a probability measure~$\mu \in\mathcal{P}(\R^d)$,
the notation $\wm(\mu)$ denotes the mean of $\mu$ after reweighting by~$\e^{-\beta f}$.
Its precise definition, together with its expression when $\mu = \emp{J}_t$,
is given by
\begin{align}
    \label{eq:wmean-emp-intro}
    \wm(\mu):= \frac{\int_{\R^d} x \e^{-\beta f(x)} \, \mu(\d x)}
    {\int_{\R^d} \e^{-\beta f(x)} \, \mu(\d x)},
    \qquad
    \wm (\emp{J}_t) :=\frac{\sum_{j=1}^J \xn{j} \e^{-\beta f\bigl(\xn{j}_t\bigr)}} {\sum_{k=1}^J  \e^{-\beta f\bigl(\xn{k}_t\bigr)}}.
\end{align}
The weighted mean $\wm(\emp{J}_t)$ may be viewed as an approximation of the particle position~$\xn{j}_t$ such that~$f(\xn{j}_t)$ is minimal.
In~\eqref{eq:cbo-particles}, the processes~$\bigl(\wienn{j}_t\bigr)_{j \in \range{1}{J}}$ are independent $\R^d$-valued Brownian motions and
$\cbodifffunc\colon \R^d\to\R^{d\times d}$ is a globally Lipschitz continuous function.
The two diffusion types commonly studied in the literature are
isotropic diffusion~\cite{CBO,carrillo2018analytical,fornasier2021consensusbased}, with~$\cbodifffunc(u)=\sqrt{2 \theta}\lvert u \rvert \I_d$,
and anisotropic or component-wise~\cite{CBO-high-dim,FornasierKlockRiedl2022-anisotropic} diffusion
with~$\cbodifffunc(u)=\sqrt{2 \theta}\diag(u)$.
Here~$\I_d\in \R^{d\times d}$ is the identity matrix and $\diag\colon \R^d\to \R^{d \times d}$ the operator mapping a vector onto a diagonal matrix with the vector as its diagonal.
It was shown in~\cite{CBO-high-dim} that CBO with anisotropic diffusion is well-suited for high-dimensional global optimization problems,
for example those arising in machine learning.
Commonly a constant coefficient is included in front of the drift term in~\eqref{eq:cbo-particles},
but since this coefficient can be removed by rescaling time upon redefining the diffusion parameter~$\theta$,
we choose to omit it.

By design, CBO allows passage to the corresponding mean-field description,
which is not the case for the~PSO algorithm.
Indeed, since the particles in~\eqref{eq:cbo-particles} evolve towards the weighted mean~\eqref{eq:wmean-emp-intro}
instead of the current best particle, as is the case for Particle Swarm Optimization (PSO)~\cite{PSO},
it is possible to formally write down the corresponding mean-field equation:
\begin{equation}
    \label{eq:cbo-mf}
    \left\{
    \begin{aligned}
        \d\xl_t   & = -\left(\xl_t - \wm( \mfldis_t)\right) \, \d t + \cbodifffunc\left(\xl_t - \wm( \mfldis_t) \right) \, \d \wien_t, \\
        \mfldis_t & = \Law(\xl_t)\,,
    \end{aligned}
    \right.
\end{equation}
where the weighted mean $\wm(\mu)$ for general measures $\mu \in\mathcal{P}(\R^d)$ can be interpreted as a smooth approximation of the `current best' that appears in PSO \cite{CBO-PSO-zero-inertia-1}.
The connection between the particle system~\eqref{eq:cbo-particles} and the dynamics~\eqref{eq:cbo-mf},
which is of McKean--Vlasov type and comparatively simpler to analyze mathematically,
stimulated theoretical research using tools from stochastic analysis and partial differential equations \cite{CBO-mfl-Huang2021, carrillo2018analytical, fornasier2021consensusbased}.
In particular, consensus formation and convergence of the CBO mean-field dynamics \eqref{eq:cbo-mf} to a consensus point as $t\to\infty$ has been proven in \cite{carrillo2018analytical}.
While the consensus point does not necessarily coincide with the global minimizer of the cost function $f$,
\cite[Theorem~4.2]{carrillo2018analytical} shows that the distance between the consensus point and the true global minimizer can be made arbitrarily small by choosing a large parameter~$\beta$.
The proof in \cite{carrillo2018analytical} covers also non-convex cost functions $f$.

The convergence, in an appropriate sense, of the empirical measure associated with the CBO finite-size particle system~\eqref{eq:cbo-particles} to the solution of the mean-field equation~\eqref{eq:cbo-mf}
in the large particle limit $J \to \infty$ was first proved rigorously in a recent paper~\cite{CBO-mfl-Huang2021} using a compactness argument,
as we discuss further in~\cref{sub:lit_rev}.
However, the method of proof employed in \cite{CBO-mfl-Huang2021} does not provide an explicit rate of convergence with respect to the number of particles.
Obtaining such an explicit rate of convergence is highly desirable,
as it would enable to transfer convergence results concerning the mean-field equation~\eqref{eq:cbo-mf} to the finite ensemble size setting~\eqref{eq:cbo-particles} by means of a triangle inequality argument,
modulo appropriate interfacing of the metrics.
This is the approach proposed in~\cite[Theorem 13]{fornasier2021consensusbased},
where the authors obtain a partial mean-field result in a set of high probability,
enabling them to prove convergence with high probability to the global minimizer in the finite ensemble size setting.
Here, we provide a stronger version of the mean-field limit result given in~\cite[Proposition 15]{fornasier2021consensusbased}, extending to more general diffusion terms and
providing scope for global convergence results in a stronger metric than convergence in probability.

A number of extensions of CBO have been proposed in the literature.
As mentioned previously,
using an anisotropic diffusion term results in a scheme which is suitable for high-dimensional state spaces,
and has been applied successfully in the training of artificial neural networks \cite{CBO-high-dim}.
Using componentwise geometric Brownian motion in this way allows for proving consensus formation directly at the particle level,
which was achieved in~\cite{Ha2019,Ha2020ConvergenceAE,HA_Discrete_CBO_hetero_noise}.
(On this subject, we point out that the results in~\cite{Ha2020ConvergenceAE} on the proximity of the consensus point to the global minimizer concern only the best case scenario;
a more general convergence result in probability or a different metric is still lacking.)
We also mention adaptations with mini-batching~\cite{jin2018random},
memory effects~\cite{MR4160256}, and Lévy noise~\cite{MR4553241},
as well as extensions to multi-objective optimization~\cite{klamroth2022consensusbased,borghi2022adaptive}
and global optimization with constraints~\cite{MR4538901,carrillo2024interacting,CTV21,CBO-hypersurface-Fornasier_2020,beddrich2024constrained, MR4329816, CBO_sphere2,MR4337731}.
Moreover,
a variant tailored for objective functions with several minima was proposed in~\cite{bungert},
and results establishing a link between CBO and PSO in the so-called small inertia limit were obtained in~\cite{CBO-PSO-zero-inertia-1, CBO-PSO-zero-inertia-2}.
Many of the recent advances concerning the mathematical analysis of CBO and its extensions were recently overviewed in~\cite{CBO-survey-Totzeck2022}.
Furthermore, a free software implementation of CBO and some of its extensions was recently published~\cite{Bailo2024}.

\subsection{Consensus-Based Sampling (CBS)}\label{sub:CBS-intro}

Consensus-Based Sampling is a sampling method inspired by the Ensemble Kalman Sampler (EKS)~\cite{EKS-1-Hofmann-Stuart-Li-GI-2019} and by consensus-based approaches such as CBO.
It is based on an interacting particle system with a noise term preconditioned by the square root of the empirical covariance.
The method in the finite ensemble size and continuous time setting reads
\begin{equation}
    \label{eq:cbs-particles}
    \d \xn{j}_t = -\Bigl(\xn{j}_t-\wm\left(\emp{J}_t\right)\Bigr) \, \d t + \sqrt{2\lambda^{-1}\wc\left(\emp{J}_t\right)} \, \d \wienn{j}_t, \qquad j \in \range{1}{J},
\end{equation}
with parameters $\beta, \lambda >0$.
Here, for a probability measure $\mu \in\mathcal{P}(\R^d)$,
the notation $\wc(\mu)$ denotes the covariance of $\mu$ after reweighting by~$\e^{-\beta f}$;
see~\eqref{eq:weighted_mean} in \cref{sub:notation} for the precise definition.
It is possible to formally write down the mean-field equation corresponding to~\eqref{eq:cbs-particles}:
\begin{align}
    \label{eq:cbs-mf}
    \left\{
    \begin{aligned}
        \d \xl_t  & = -\bigl(\xl_t-\wm\left(\mfldis_t\right)\bigr) \, \d t + \sqrt{2\lambda^{-1}\wc(\mfldis_t)} \, \d \wien_t, \\
        \mfldis_t & = \Law(\xl_t).
    \end{aligned}
    \right.
\end{align}
When the parameters satisfy $\lambda = (1+\beta)^{-1}$,
the interacting particle system~\eqref{eq:cbs-particles} enables to generate approximate samples from the probability distribution with density proportional to~$\e^{-f}$,
that is from
\[
    \rho_f := \frac{\e^{-f}}{\mathcal Z}, \qquad \mathcal Z := \int_{\R^d} \e^{-f(x)} \, \d x.
\]
More precisely, the authors of \cite{CBS-Carrillo2021} show that for uniformly convex $f$ and Gaussian initial data~$\mfldis_0$,
the mean-field distribution~$\mfldis_t$ converges exponentially in Wasserstein distance as~$t \to \infty$
to a Gaussian $\normal(m_{\infty}, C_{\infty})$ that can be made arbitrarily close to the Laplace approximation of~$\rho_f$ by choosing a sufficiently large parameter~$\beta$.
In addition, it was noticed in~\cite{CBS-Carrillo2021} that,
with $\lambda :=1$, the dynamics~\eqref{eq:cbs-particles} behaves as a global optimization method for finding the global minimum of $f\colon \R^d \rightarrow \R$,
just like the Consensus-Based Optimization method proposed in~\cite{CBO}.
Since the convergence analysis in~\cite{CBS-Carrillo2021} is confined to the mean-field level,
it is of interest to prove rigorously that the particle dynamics~\eqref{eq:cbs-particles} converges in an appropriate sense to the corresponding mean-field dynamics~\eqref{eq:cbs-mf} as $J \to \infty$. This has so far been left as an open problem.
Recently, consensus-based sampling was extended to enable sampling of multi-modal distributions in~\cite{bungert},
and for rare event estimation in~\cite{2023arXiv230409077A}.     It has also been shown that CBS,
together with other sampling methods based on interacting particle systems,
can be metropolized to correct the sampling error~\cite{sprungk2023metropolisadjustedinteractingparticlesampling}.

\subsection{Propagation of chaos}
\label{sub:lit_rev}

The CBO~\eqref{eq:cbo-particles} and~CBS~\eqref{eq:cbs-particles} dynamics can both be recast as interacting particle systems of the form
\begin{equation}
    \label{eq:sde-ips}
    \d \xn{j}_t = b\left(\xn{j}_t, \mu^J_t\right) \, \d t + \sigma\left(\xn{j}_t, \mu^J_t\right) \, \d W^{(j)},
    \qquad j \in \llbracket 1, J \rrbracket,
\end{equation}
with drift coefficient~$b$ and diffusion coefficient~$\sigma$ depending on~$\mu$ through a convolution with appropriate kernels~$K_b$ and $K_{\sigma}$:
\begin{equation}
    \label{eq:form_coefficients}
    b(x, \mu) = \widetilde b(x, K_b \star \mu), \qquad
    \sigma(x, \mu) = \widetilde \sigma(x, K_{\sigma} \star \mu).
\end{equation}
Although the consensus-based dynamics can be written in the form~\eqref{eq:sde-ips},
it is not possible to do so in such a way that $\widetilde{b}, \widetilde{\sigma}, K_b$ and $K_{\sigma}$ are all globally Lipschitz continuous and
at the same time $K_b$ and $K_{\sigma}$ are globally bounded.
Therefore, the classical McKean's theorem~\cite[Theorem 3.1]{ReviewChaintronII} is not directly applicable,
and a different approach is needed.
We briefly discuss hereafter the existing mean-field results for CBO,
noting that no results have been obtained so far for CBS.
\begin{itemize}

    \item
          In~\cite{CBO-mfl-Huang2021},
          propagation of chaos for CBO is proved through a compactness argument based on Prokhorov's theorem.
          However, this method comes without convergence rates in terms of the number of particles~$J$.

    \item
          Later,
          propagation of chaos by means of synchronous coupling was shown with optimal rates in \cite{CBO-hypersurface-Fornasier_2020}
          for a variant of CBO constrained to hypersurfaces.

    \item
          Recent results in~\cite{fornasier2021consensusbased} also show a probabilistic mean-field approximation of the form
          \[
              \sup_{t\in[0,T]}\expect\left[\myabs*{\xn{j}_t-\xnl{j}_t}^2 \mathsf 1_{\Omega \setminus \Omega_M}\right] \le \frac{C}{J},
              \qquad
              \Omega_M:=\left\{\sup_{t\in[0,T]} \frac{1}{J}\sum_{j=1}^J\max\left\{\left\lvert \xn{j}_t \right\rvert^4, \left\lvert \xnl{j}_t \right\rvert^4\right\} \ge M\right\}.
          \]
          It is additionally proved in~\cite{fornasier2021consensusbased},
          by application of Markov's inequality,
          that the set of realizations $\Omega_M$ has probability close to~0 for large~$M$.
          Specifically, it holds that $\proba (\Omega_M) = \mathcal O(M^{-1})$.

    \item
          Shortly before completing this work,
          we found out that propagation of chaos for a variant of CBO had been proved using a coupling approach very recently in~\cite{MR4553241}.
          In that work, the lack of global Lipschitz continuity of the weighted mean is handled by using stopping times,
          similarly to the approach we present in~\cref{sub:stopping_times},
          which was developed independently.
          There are, however, several differences between our results and those presented in~\cite{MR4553241}.
          Apart from the assumptions, which are slightly less restrictive in our work,
          the main technical differences are that
          (i)~we improve existing Wasserstein stability estimates on the weighted mean~(\cref{lemma:stab:improved}),
          (ii)~we rely on a powerful result from the statistics literature~\cite{MR3696003,MR2597592} to bound one of the terms in Sznitman's argument~(\cref{lemma:convergence_weighted_mean_iid}) and
          (iii)~we use a concentration inequality to bound the probability that empirical measures exit a compact set~(\cref{sub:key_lemma}).
          Combined together, these ingredients enable to obtain mean-field result with the usual algebraic rate of convergence with respect to~$J$,
          whereas the error scales as $\ln(\ln(J))^{-1}$ in the mean-field result obtained in~\cite{MR4553241}.

    \item

          Finally,
          we mention the recent preprint~\cite{koß2024meanfieldlimitconsensus},
          which appeared a few months after this one and proves propagation of chaos for consensus-based sampling using a compactness argument,
          similar to the approach used in~\cite{CBO-mfl-Huang2021}.

\end{itemize}

\subsection{Our contributions}
\label{sub:contrib}

The main contributions of this paper are the following.
\begin{itemize}
    \item
          We obtain quantitative results on the mean-field limit for CBO and CBS.
          More precisely, we prove results with an explicit rate of convergence with respect to the number~$J$ of particles.
          Investigating the roles of the dimension and other problem parameters on the value of the constant prefactor is left for future work.

    \item
          We present two approaches for dealing with the weighted mean (in the case of CBO) and the weighted covariance (in the case of CBS).
          One of these approaches is based on the careful use of indicator sets,
          whereas the other is based on stopping times.
          The latter approach is inspired by a method initially proposed in~\cite{MR1949404}
          in order to handle local Lipschitz continuity issues
          in the context of numerical analysis for the Euler--Maruyama method.

    \item
          In order to prove mean-field results,
          we generalize existing theorems and prove new results on the well-posedness of the particle systems and their mean-field limit.
          We also present novel stability estimates for the weighted mean and the weighted covariance.
\end{itemize}

The rest of this paper is organized as follows.
In~\cref{sec:main},
we present our main results for the consensus-based dynamics~\eqref{eq:cbo-particles} and~\eqref{eq:cbs-particles}.
They are based on a number of auxiliary results,
which are stated precisely and proved in~\cref{sec:aux}.
These include Wasserstein stability estimates for the weighted mean and the weighted covariance,
as well as moment estimates for the particle systems.
In~\cref{sec:cbo-cbs-wellp},
we prove well-posedness results for the particle systems and their mean-field limits.
\Cref{sec:conclusion} is reserved for extensions and perspectives for future work.
Additional technical results are presented in~\cref{sec:aux-results}.
The interdependence of the sections is illustrated in~\cref{figure:organization}.

\begin{figure}[ht!]
    \centering
    \resizebox{.99\textwidth}{!}{%
        \begin{tikzpicture}
            \node (0) [draw, align=center, rounded rectangle, inner sep=.3cm, text width=6cm] {Main theorems on the mean-field limits for CBO and CBS. \\ \Cref{sec:main}};
            \node (1) [draw, align=center, rounded rectangle, inner sep=.2cm, text width=5cm, below left=1.5cm and 1cm of 0] {Wasserstein stability estimates for weighted moments. \\ \Cref{sub:stability_estimates}};
            \node (2) [draw, align=center, rounded rectangle, inner sep=.3cm, text width=4cm, right=of 1] {Moment estimates for empirical measures. \\ \Cref{sub:moment}};
            \node (3) [draw, align=center, rounded rectangle, inner sep=.3cm, text width=5cm, right=of 2] {Convergence of the weighted moments for i.i.d.\ samples. \\ \Cref{sub:iid}};
            \node (4) [draw, align=center, rounded rectangle, inner sep=.3cm, text width=5cm, below right=1.5cm and -1cm of 1] {Well-posedness result for the mean-field equations. \\ \Cref{sec:cbo-cbs-wellp}};
            \node (5) [draw, align=center, rounded rectangle, inner sep=.3cm, text width=4cm, below left=1.5cm and -1cm of 3] {Technical bound on weighted moments. \\ \Cref{sub:bound_weighted_moments}};
            \draw[->] (5) -- (1);
            \draw[->] (5) -- (2);
            \draw[->] (5) -- (4);
            \draw[->] (1) -- (0);
            \draw[->] (2) -- (0);
            \draw[->] (3) -- (0);
            \draw[->] (1) -- (4);
            \draw[->] (1) -- (0);
            \draw[->] (4) to [out=110, in=240] (0);
        \end{tikzpicture}
    }
    \caption{%
        Interdependence of the results contained in this paper.
    }%
    \label{figure:organization}
\end{figure}

\subsection{Notation}
\label{sub:notation}

We let $\N :=\{0,1,2,3,\dots\}$
and $\N^+ :=\{1,2,3,\dots\}$.
For symmetric matrices $X$ and $Y$,
the notation~$X \succcurlyeq Y$ means that $X - Y$ is positive semidefinite,
and the notation $X \succ Y$ means that $X -Y$ is strictly positive definite.
The notation $\I_d\in \R^{d\times d}$ refers to the identity matrix.
The Euclidean norm in~$\R^d$ is denoted by~$\myabs{\placeholder}$.
For a matrix $X \in \R^{d \times d}$,
the notation~$\norm{X}$ refers to the operator norm induced by the Euclidean vector norm,
and the notation~$\norm{X}_{\rm F}$ refers to the Frobenius norm.
The set of probability measures over a metric space~$E$ is denoted by~$\mathcal P(E)$,
and $\mathcal P_{p}(E) \subset \mathcal P(E)$ denotes the set of probability measures with finite moments up to order~$p$,
and~$\mathcal P_{p,R}(E) \subset \mathcal P_p(E)$ denotes the set of probability measures~$\mu$ such that
\[
    \left( \int_{E} \myabs{x}^p \, \mu(\d x) \right)^{\frac{1}{p}} \leq R.
\]
Note that $\mathcal P_{q,R}(E) \subset \mathcal P_{p,R}(E)$ for all $0 \leq p < q$.
If the metric space~$E$ is not specified,
then it is implicitly assumed that~$E = \R^d$.
For example, we sometimes write $\mathcal P_{p,R}$ as a shorthand notation for $\mathcal P_{p,R}(\R^d)$.
For a probability measure $\mu \in\mathcal{P}(\R^d)$
we use the following notation for the mean and covariance:
\begin{align}
    \mathcal{M}(\mu) = \int_{\R^d} x \, \mu(\d x),
    \qquad
    \mathcal{C}(\mu) = \int_{\R^d} \bigl(x - \mathcal M(\mu)\bigr) \otimes \bigl(x - \mathcal M(\mu)\bigr)\, \mu(\d x).
\end{align}
The notations $\wm(\mu)$ and $\wc(\mu)$ denote respectively
the mean and covariance after reweighting by~$\e^{-\beta f}$.
More precisely, we use the notation
\begin{equation}
    \label{eq:weighted_mean}
    \wm(\mu)
    = \mathcal M(L_{\beta} \mu), \qquad
    \wc(\mu) = \mathcal C(L_{\beta} \mu), \qquad
    L_\beta\mu := \frac{ \e^{-\beta f} \, \mu }{\int \e^{-\beta f} \, \d \mu}.
\end{equation}
For a probability measure $\mu \in \mathcal P(E)$ and a function $f\colon E \to \R$,
we use the short-hand notation
\[
    \mu [f] = \int_{E} f(x) \, \mu(\d x).
\]
By a slight abuse of notation,
we sometimes write $\mu[f(x)]$ instead of $\mu[f]$ for convenience.
For example, for a probability measure~$\mu \in \mathcal P_2(\R)$,
the notation $\mu\left[  x^2 \right]$ refers to the second raw moment of~$\mu$.
Furthermore, for~$J\in\N^+$ we denote by $\mu^J := \frac{1}{J} \sum_{j=1}^{J} \delta_{\xn{j}}$ and $\overline \mu^J := \frac{1}{J} \sum_{j=1}^{J} \delta_{\xnl{j}}$ the empirical measures of the particles~$\xn{1},\dots, \xn{J}$ and $\xnl{1}, \dots, \xnl{J}$, respectively.
Finally, throughout this paper,
the notation $\Omega$ refers to the sample space,
and the notation $C$
refers to a constant whose exact value is irrelevant in the context and may change from line to line.

\section{Main results}
\label{sec:main}

This section is organized as follows.
In~\cref{sec:setting}, we present the assumptions used throughout this work.
Then, after stating well-posedness results for CBO and CBS in~\cref{sub:well-posed},
we prove in~\cref{sub:key_lemma} a key preparatory result for the mean-field limits.
We then prove quantitative mean-field results for Consensus-Based Optimization and Sampling
in~\cref{sec:mfl_cbo,sec:mfl_cbs}, respectively.
Finally, a numerical experiment illustrating the convergence as~$J \to \infty$ for CBO is presented in~\cref{sub:experiment}.

\subsection{Main assumptions}
\label{sec:setting}

We assume throughout this paper that the diffusion operator $\cbodifffunc\colon \R^d\rightarrow \R^{d\times d}$ appearing in the CBO dynamics~\eqref{eq:cbo-particles} is globally Lipschitz continuous.
For the objective function~$f$,
we use the following assumption.

\begin{assumption}
    \label{assump:main}
    We assume the following properties for the target function $f$.
    \begin{assumpenum}[label=(H\arabic*)]
        \item
        \label{assump:f-local-lip-bound}
        The function $f\colon \R^d\rightarrow\R$ is bounded from below by $f_\star=\inf f$
        and satisfies
        \begin{align}
            \label{eq:assump-f:lip-growth-gradient}
            \forall x, y \in \R^d, \qquad
            |f(x)-f(y)| \le L_f \bigl(1+|x|+|y|\bigr)^s|x-y|,
        \end{align}
        for some constants $L_f>0$ and $s \ge 0$.

        \item \label{assump:f-at-infinity}
            There exist constants $c_{\ell}, C_{\ell}, c_u, C_{u} > 0$ and $u \geq \ell \geq 0$ such that
        \begin{subequations}
            \begin{align}
                \label{assump:f-upper-bound}
                \forall x \in \real^d, \qquad
                f(x) - f_{\star} \le c_u |x|^{u}+ C_u, \\
                \label{assump:f-lower-bound}
                \forall x \in \real^d, \qquad
                f(x) - f_\star \ge c_{\ell} |x|^{\ell} - C_\ell.
            \end{align}
        \end{subequations}
    \end{assumpenum}
\end{assumption}

In the rest of this paper,
we denote by $\mathcal A(s, \ell, u)$ the set of objective functions that satisfy~\cref{assump:main} with these parameters.
\begin{remark}
    \label{rmk:f-upper-bound}
    A few remarks are in order.
    \begin{itemize}
        \item The case $\ell=u=0$ corresponds to bounded $f$.
        \item If $f\in \mathcal A(s, \ell, u)$, then necessarily $\ell \le s+1$,
            because the set~$\mathcal A(s, \ell, u)$ is empty for $\ell > s + 1$.
        \item In the particular case where $s=1$ and either $u = \ell = 0$ or $u = \ell =2$,
        the assumptions in~\cref{assump:main} coincide with those in \cite{carrillo2018analytical}.
        The well-posedness and mean-field results proved in this paper generalize beyond this setting, as explained below.
    \end{itemize}
\end{remark} 
Our main results hold for $f \in \mathcal A(s, \ell, \ell)$,
i.e.\ under \cref{assump:main} with $\ell = u$.
The problem of proving mean-field limits for CBO/CBS in the case $\ell <u$ is discussed in \cref{sub:different_l_u}.

Many of the results we present hold only for values of $p$ in a range dictated by the interval of validity of the Wasserstein stability estimates presented in~\cref{sub:stability_estimates}.
In order to avoid repetitions,
we introduce the associated notation
\begin{equation}
    \label{eq:def_ps}
    p_{\mathcal M}(s, \ell) \coloneq \begin{cases}
        s+2,\, & \text{if $\ell = 0$}, \\
        1,\,   & \text{if $\ell > 0$},
    \end{cases}
    \qquad
    p_{\mathcal C}(s, \ell) \coloneq \begin{cases}
        s+3,\, & \text{if $\ell = 0$}, \\
        1,\,   & \text{if $\ell > 0$}.
    \end{cases}
\end{equation}

\subsection{Well-posedness for CBO and CBS}
\label{sub:well-posed}

In this section,
we state well-posedness results for CBO and CBS, referring for the proofs to \cref{sec:cbo-cbs-wellp}.
While for CBO, the following theorems are only generalizations of already existing results \cite{carrillo2018analytical},
the well-posedness proofs for the CBS particle and mean-field systems are new.

\begin{theorem}[Existence and uniqueness for the particle models]
    \label{thm:cbo-cbs-particles-well-posed}
    Assume that $f\colon \R^d \rightarrow \R$ is locally Lipschitz continuous and fix $J\in\N^+$.
    Then, the stochastic differential equations~\eqref{eq:cbo-particles} and~\eqref{eq:cbs-particles}
    have unique strong solutions $\bigl\{{\mathbf X}_t^{(J)}\,|\,t\ge 0\bigr\}$ for any initial condition ${\mathbf X}_0^{(J)}$ which is independent of the Brownian motions~$\mathbf{\wien}_t=\bigl(\wienn{1}_t, \dots, \wienn{J}_t\bigr)$.
    The solutions are almost surely continuous.
\end{theorem}

In the case of CBO, \cref{thm:cbo-cbs-particles-well-posed} was already proved in \cite[Theorem 2.1]{carrillo2018analytical}.
We now present well-posedness results for the mean-field SDEs, beginning with CBO.
For~$f \in \mathcal A(1, 0, 0) \cup \mathcal A(1, 2, 2)$,
the following result generalizes~\cite[Theorems 3.1 and 3.2]{carrillo2018analytical}.
\begin{theorem}[CBO: Existence and uniqueness for the mean-field SDE]
    \label{thm:cbo-mean-field-sde-well-posed}
    Suppose that~$f \in \mathcal A(s, \ell, \ell)$~with parameters $s,\ell\ge 0$ such that $\ell\le s+1$ and fix a final time~$T>0$.
    Assume also that $\mfldis_0\in\mathcal{P}_{p}(\R^d)$ for~$p \geq 2 \vee p_{\mathcal M}(s, \ell)$,
    and fix $x_0 \sim \mfldis_0$.
    Then,
        there exists a strong solution~$\xl \colon \Omega \to  C([0,T], \R^d)$ to \eqref{eq:cbo-mf} with initial condition~$\xl_0 = x_0$ such that
        $t \mapsto \wm(\mfldis_t)$ is continuous over $[0, T]$.
        Furthermore, the process  $\xl$ is unique within the class of strong solutions to \eqref{eq:cbo-mf} such that $t \mapsto \wm(\mfldis_t)$ is continuous over $[0, T]$, and it holds that
        \begin{align}
            \label{eq:cbo-moment-bound-mfl}
            \expect \left[ \sup_{t\in[0,T]}  \myabs*{\xl_t}^p \right]
            < \infty, &  & \text{where} \qquad
            \mfldis_t \coloneq\Law(\xl_t).
        \end{align}
    Finally, the function $t \mapsto \mfldis_{t}$ belongs to $C\left([0,T], \mathcal{P}_{p}(\R^d)\right)$
    and satisfies the following non-local Fokker--Planck equation in the weak sense:
    \begin{equation}
        \label{eq:CBO:FKP}
        \partial_t \mfldis =   \nabla \cdot \Bigl( \bigl(x - \wm(\mfldis) \bigr) \mfldis \Bigr)
        + \frac{1}{2} \nabla \cdot \nabla \cdot \Bigl(D (\mfldis, x) \mfldis \Bigr),
    \end{equation}
    where $D(\mfldis, x) \coloneq \cbodifffunc \bigl(x-\wm(\mfldis)\bigr) \cbodifffunc \bigl(x-\wm(\mfldis)\bigr)^\t$.
\end{theorem}

\begin{theorem}[CBS: Existence and uniqueness for the mean-field SDE]
    \label{thm:cbs-mean-field-sde-well-posed}
    Suppose that~$f \in \mathcal A(s, \ell, \ell)$~with parameters $s,\ell\ge 0$ such that $\ell\le s+1$ and fix a final time~$T>0$.
    Assume also that $\mfldis_0\in\mathcal{P}_{p}(\R^d)$ for~$p \geq 2 \vee p_{\mathcal C}(s, \ell)$, that $C_{\beta}(\mfldis_0) \succ 0$,
    and fix~$x_0 \sim \mfldis_0$.
    Then, there exists a strong solution~$\xl \colon \Omega \to  C([0,T], \R^d)$ to \eqref{eq:cbs-mf} with initial condition~$\xl_0 = x_0$ such that~$t \mapsto \wm(\mfldis_t)$ and~$t \mapsto \wc(\mfldis_t)$ are continuous over $[0, T]$,
        Furthermore, the process  $\xl$ is unique within the class of strong solutions to \eqref{eq:cbs-mf} such that~$t \mapsto \wm(\mfldis_t)$ and~$t \mapsto \wc(\mfldis_t)$ are continuous over $[0, T]$, and it holds that
        \begin{align}
            \label{eq:cbs-moment-bound-mfl}
            \expect \left[ \sup_{t\in[0,T]}  \myabs*{\xl_{t}}^p \right]
            < \infty, &  & \text{where}
            \qquad \mfldis_t \coloneq\Law(\xl_t).
        \end{align}
    Finally, the function $t \mapsto \mfldis_{t}$ belongs to~$C([0,T], \mathcal{P}_{p}(\R^d))$
    and satisfies the following non-local Fokker--Planck equation in the weak sense:
    \begin{align}
        \label{eq:cbs:FKP}
        \partial_t \mfldis = \nabla \cdot \Bigl( \bigl(x-\wm(\mfldis)\bigr)\mfldis + \lambda^{-1} \wc(\mfldis) \nabla \mfldis \Bigr).
    \end{align}
\end{theorem}

\subsection{Key preparatory lemma}
\label{sub:key_lemma}

We first show a key preparatory lemma which is needed for the mean-field results.
This result is a simple corollary of Markov's inequality.
\begin{lemma}
    [Bound on the probability of large excursions]
    \label{lemma:small_set}
    Let $(Z_j)_{j\in\N^+}$ be a family of i.i.d.\ $\real$-valued random variables such that $\expect  \left[\myabs*{Z_1}^{r}\right]  < \infty$
    for some $r \geq 2$.
    Then for all $R > \expect \bigl[\myabs*{Z_1}\bigr]$,
    there exists a constant $C>0$ such that
    \[
        \forall J \in \N^+, \qquad
        \proba \left[  \frac{1}{J} \sum_{j=1}^{J} Z_j \geq R \right]
        \leq C J^{-\frac{r}{2}}.
    \]
\end{lemma}
\begin{proof}
    This follows from a generalization of Chebychev's inequality~\cite[Exercise 3.21]{MR3674428}.
    Let
    \[
        X = \frac{1}{J} \sum_{j=1}^{J} Z_j.
    \]
    By the classical Marcinkiewicz--Zygmund inequality,
    it holds that
    \begin{align*}
        \expect \bigl\lvert X - \expect [X] \bigr\rvert^{r}
         & \leq J^{-r} C_{\rm MZ}(r) \expect \left[ \left( \sum_{j=1}^{J} \Bigl\lvert Z_j - \expect [Z_j] \Bigr\rvert^2 \right)^{\frac{r}{2}} \right] \\
         & \leq J^{-\frac{r}{2}} C_{\rm MZ}(r) \expect \left[ \frac{1}{J} \sum_{j=1}^{J} \Bigl\lvert Z_j - \expect [Z_j] \Bigr\rvert^r \right]
        = J^{-\frac{r}{2}} C_{\rm MZ}(r) \expect \left[ \Bigl\lvert Z_1 - \expect [Z_1] \Bigr\rvert^r \right],
    \end{align*}
    where we used the convexity of the function $x \mapsto x^{\frac{r}{2}}$ in the second inequality.
    Therefore,
    using the Markov inequality,
    we deduce that
    \begin{align*}
        \proba \left[ X  \ge R \right]
        \leq \proba \Bigl[ \bigl\lvert X - \expect [X]  \bigr\rvert^r  \ge \bigl(R - \expect [X]\bigr)^r \Bigr]
         & \leq \expect \left[ \frac{\bigl\lvert X - \expect[X] \bigr\rvert^r}
            {\bigl(R - \expect [X]\bigr)^r} \right]
        \leq \frac{C J^{-\frac{r}{2}}}{\bigl(R - \expect [X]\bigr)^r} ,
    \end{align*}
    which concludes the proof.
\end{proof}

\subsection{Quantitative mean-field limit for CBO}
\label{sec:mfl_cbo}

We now present and prove the mean-field results, starting with CBO.
Synchronous coupling for McKean-Vlasov diffusions with drift $b\colon\R^d \times \mathcal P(\R^d) \rightarrow  \R^d$
and diffusion $\sigma\colon \R^d \times \mathcal P(\R^d) \rightarrow  \R^{d\times d}$ assumes the following setting~\cite{ReviewChaintronI, ReviewChaintronII}.
Given i.i.d.\ Brownian motions $\bigl(\wienn{j}\bigr)_{j\in\N^+}$
and initial conditions $\bigl(\xn{j}_0\bigr)_{j\in\N^+}$ sampled i.i.d.\ from some $\mfldis_0\in\mathcal P(\R^d)$, we consider for each $J\in\N^+$ the interacting particle system
\begin{align}
    \label{eq:coupling-system-micro}
    \forall j\in \range{1}{J}, \qquad
    \xn{j}_t = \xn{j}_0 + \int_0^t b\Bigl(\xn{j}_s, \mu^J_s\Bigr) \, \d s +  \int_0^t \sigma\Bigl(\xn{j}_s, \mu^J_s\Bigr) \d W^{(j)}_s,
\end{align}
to which we couple the system of i.i.d.\ mean-field McKean-Vlasov diffusions
\begin{align}
    \label{eq:coupling-system-mfl}
    \forall N\in\N^+,\forall j\in \range{1}{J}, \qquad
    \left\{
    \begin{aligned}
        \xnl{j}_t & = \xn{j}_0 + \int_0^t b\Bigl(\xnl{j}_s, \mfldis_s\Bigr) \, \d s +  \int_0^t \sigma\Bigl(\xnl{j}_s, \mfldis_s\Bigr) \, \d W^{(j)}_s, \\
        \mfldis_t & = \operatorname{Law}\left(\xnl{j}_t\right),
    \end{aligned}
    \right.
\end{align}
driven by the same Brownian motions and initial conditions as \eqref{eq:coupling-system-micro}. In the case of CBO, the drift and diffusion coefficients for \eqref{eq:coupling-system-micro} and \eqref{eq:coupling-system-mfl} are given by
\begin{align}
    \label{eq:cbo-drift-and-diff}
    b(x, \mu) \coloneq -   \bigl( x - \wm\left(\mu \right) \bigr) &  & \text{and} &  & \sigma(x, \mu) \coloneq \cbodifffunc\bigl( x - \wm\left(\mu \right)\bigr),
\end{align}
where the diffusion operator $\cbodifffunc\colon \R^d\rightarrow \R^{d\times d}$ is, as always in this work, a globally Lipschitz continuous function.

\begin{theorem}
    [Mean-field limit for CBO]
    \label{thm:mfl-cbo}
    Suppose that $f \in \mathcal A(s, \ell, \ell)$ with $s,\ell\ge 0$ such that $\ell\le s+1$ and that $\mfldis_0\in \mathcal P(\R^d)$ has bounded moments of all orders.
    Consider the systems \eqref{eq:coupling-system-micro} and \eqref{eq:coupling-system-mfl} with the coefficients given by~\eqref{eq:cbo-drift-and-diff}.
    Then, for each $p>0$, there exists a constant $C > 0$ independent of $J$ such that
    \begin{equation}
        \label{eq:thm:mfl-cbo}
        \forall J\in\N^+, \qquad
        \forall j \in \range{1}{J}, \qquad
        \left(\expect \left[ \sup_{t\in[0,T]} \myabs[\Big]{\xn{j}_t-\xnl{j}_t}^p \right]\right)^{\frac{1}{p}}
        \le C J^{- \frac{1}{2}} .
    \end{equation}
\end{theorem}

\begin{remark}
    \label{remark:cbo-mfl-sharper}
    We prove below the following more general statement:
    Suppose that $f \in \mathcal A(s, \ell, \ell)$ with $s,\ell\ge 0$ such that $\ell\le s+1$  and that~$\mfldis_0\in \mathcal P_q(\R^d)$.
    If~$q \geq \max \bigl\{4, 2 p_{\mathcal M}(s, \ell) \bigr\}$,
    where $p_{\mathcal M}$ is given in~\eqref{eq:def_ps},
    then for all $p \in (0, \frac{q}{2}]$ there is~$C > 0$ independent of~$J$ such that
    \begin{equation}
        \label{eq:thm:mfl-cbo_sharp}
        \forall J\in\N^+, \qquad
        \forall j \in \range{1}{J}, \qquad
        \left(\expect \left[ \sup_{t\in[0,T]} \myabs[\Big]{\xn{j}_t-\xnl{j}_t}^p \right]\right)^{\frac{1}{p}}
        \le C J^{- \min \left\{ \frac{1}{2}, \frac{q-p}{2p^2}, \frac{q-(2 \vee p_{\mathcal M})}{2 (2 \vee p_{\mathcal M})^2} \right\}} .
    \end{equation}
\end{remark}

\begin{remark}
    \cref{thm:mfl-cbo} implies pathwise Wasserstein-$p$ chaos.
    More precisely, it holds that
    \begin{equation*}
        \wasserstein_p\Bigl(\rho^J_{[0,T]}, \mfldis^{\otimes J}_{[0,T]}\Bigr)
        \le C  J^{- \frac{1}{2}},
    \end{equation*}
    where  $\rho^J_{[0,T]}$ is the pathwise law of the microscopic particle system over~$[0,T]$ and $\mfldis_{[0,T]}$ is the pathwise mean-field law over~$[0,T]$.
    That is to say,
    \[
        \left\{
        \begin{aligned}
            \rho^J_{[0,T]}  & \coloneq \Law \left(\xn{1}, \dots \xn{J}\right) \in  \mathcal{P}\Bigl(C\bigl([0,T], \R^{d}\bigr)^J\Bigr), \\
            \mfldis_{[0,T]} & \coloneq \Law(\xl)\in \mathcal P \Bigl( C\bigl([0,T], \R^d\bigr)\Bigr).
        \end{aligned}
        \right.
    \]
    Here, $\wasserstein_p$ denotes the pathwise Wasserstein-$p$ distance,
    associated with the following metric on $\Bigl( C\bigl([0,T], \R^d\bigr)\Bigr)^J$:
    \[
        d(\mathbf X, \mathbf Y) = \sup_{t \in [0, T]} \frac{1}{J} \sum_{j=1}^J \left\lvert \xn{j}_t - \yn{j}_t \right\rvert,
        \qquad \mathbf X, \mathbf Y \in \Bigl( C\bigl([0,T], \R^d\bigr)\Bigr)^J.
    \]
    See~\cite[Definitions 3.1 and 3.5]{ReviewChaintronI} for more details.
\end{remark}

Before presenting the proof,
we summarize below a number of useful auxiliary results which are valid for~$f\in  \mathcal A(s, \ell, \ell)$.
These results are stated precisely and proved in~\cref{sec:aux}.
\begin{auxenum}
    \item \label{item:cbo:stab:improved}
    \textbf{Wasserstein stability estimate for the weighted mean}.
    We prove in \cref{lemma:stab:improved} that,
    for all~$R>0$ and for all $p \geq p_{\mathcal M}(s, \ell)$,
    where $p_{\mathcal M}(s, \ell) \in \real^+$ is given in~\eqref{eq:def_ps},
    there exists~$\widetilde L_{\mathcal M} = \widetilde L_{\mathcal M}(R,p)>0$ such that
    \begin{align*}
        \forall (\mu, \nu) \in \mathcal P_{p,R}\bigl(\real^d\bigr) \times \mathcal P_p\bigl(\real^d\bigr), \qquad
        \myabs*{\wm(\mu) - \wm(\nu)}
        \le \widetilde L_{\mathcal M} \, \wasserstein_p(\mu, \nu).
    \end{align*}

    \item \label{item:cbo:moment-bounds-empirical}
    \textbf{Moment estimate for the particle system}.
    We prove in~\cref{lemma:moment-bounds-empirical} that,
    if the initial condition satisfies $\mfldis_0 \in \mathcal P_q(\real^d)$,
    then there exists a constant $\kappa >0$ independent of $J$
    such that
    \[
        \forall J \in\N^+, \qquad
        \expect \left[ \sup_{t\in[0,T]} \left\lvert X_t^{(j)} \right\rvert^{q} \right]
        \quad \vee \quad
        \expect \left[ \sup_{t\in[0,T]}\left\lvert \wm(\mu^J_t) \right\rvert^{q} \right]
        \leq \kappa.
    \]

    \item
    \label{item:cbo:convergence_weighted_mean_iid}
    \textbf{Convergence of the weighted mean for i.i.d.\ samples.}
    We prove in~\cref{lemma:convergence_weighted_mean_iid} that,
    for any reals $r> p > 0$ and any $\mu \in \mathcal P_r(\R^d)$,
    there is~$C$ such that
    \[
        \expect \left\lvert
        \wm(\overline \mu^J)
        - \wm(\mu)
        \right\rvert^p
        \leq C J^{- \frac{p}{2}},
        \qquad
        \overline \mu^J \coloneq \frac{1}{J} \sum_{j=1}^{J} \delta_{\xnl{j}},
        \qquad
        \left\{ \xnl{j} \right\}_{j \in \N} \stackrel{\rm{i.i.d.}}{\sim} \mu.
    \]
\end{auxenum}

\begin{proof}
    [Proof of \cref{thm:mfl-cbo}]
    As mentioned in~\cref{remark:cbo-mfl-sharper},
    we prove the general result~\eqref{eq:thm:mfl-cbo_sharp}.
    It is sufficient to prove this statement for all~$p \in [2 \vee p_{\mathcal M}, \frac{q}{2}]$
    (by the assumption on $q$, this interval is nonempty).
    Indeed, if the statement holds true in this case,
    then by Jensen's inequality,
    it holds for all $\bar p \in (0, 2 \vee p_{\mathcal M}]$ that
    \begin{align*}
        \left(\expect \left[ \sup_{t\in[0,T]} \myabs[\Big]{\xn{j}_t-\xnl{j}_t}^{\bar p} \right]\right)^{\frac{1}{\bar p}}
         & \leq \left(\expect \left[ \sup_{t\in[0,T]} \myabs[\Big]{\xn{j}_t-\xnl{j}_t}^{2\vee p_{\mathcal M}} \right] \right)^{\frac{1}{2 \vee p_{\mathcal M}}} \\
         & \leq C J^{- \min \left\{ \frac{1}{2}, \frac{q - 2 \vee p_{\mathcal M}}{2 (2\vee p_{\mathcal M})^2} \right\}}
        = C J^{- \min \left\{ \frac{1}{2}, \frac{q - \bar p}{2 \bar p^2}, \frac{q - 2 \vee p_{\mathcal M}}{2 (2\vee p_{\mathcal M})^2} \right\}},
    \end{align*}
    where we used that the function $z \mapsto (q-z)/z^2$ is decreasing.
    Take $p \in [2 \vee p_{\mathcal M}, \frac{q}{2}]$ and fix $j\in\range{1}{J}$.
    As in Sznitman's argument \cite[Theorem 3.1]{ReviewChaintronII}, we begin by writing
    \begin{align*}
        \bigl\lvert X_{t}^{(j)} - \overline X_{t}^{(j)} \bigr\rvert^p
        \leq
         & \,2^{p-1}\left\lvert \int_{0}^{t} b\Bigl(X_{s}^{(j)}, \mu^J_{s}\Bigr) - b\Bigl(\overline X_{s}^{(j)}, \mfldis_{s} \Bigr) \, \d s \right\rvert^p               \\
         & + 2^{p-1} \left\lvert \int_{0}^{t} \sigma\Bigl(X_{s}^{(j)}, \mu^J_{s}\Bigr) - \sigma\Bigl(\overline X_{s}^{(j)}, \mfldis_{s} \Bigr) \, \d W_s \right\rvert^p.
    \end{align*}
    By the Burkholder--Davis--Gundy inequality,
    see~\cite[Theorem 7.3]{MR2380366},
    it holds for any $t \in [0, T]$ that
    \begin{align}\notag
        \expect \left[ \sup_{s \in [0, t]} \bigl\lvert X_{s}^{(j)} - \overline X_{s}^{(j)} \bigr\rvert^p \right]
         & \leq 2^{p-1} \expect \left( \int_{0}^{t} \left\lvert b\Bigl(X_{s}^{(j)}, \mu^J_{s}\Bigr) - b\Bigl(\overline X_{s}^{(j)}, \mfldis_{s} \Bigr) \right\rvert \, \d s \right)^p                                                    \\
        \notag
         & \qquad +  C_{\rm BDG} 2^{p-1} \expect \left( \int_{0}^{t} \left\lVert  \sigma\Bigl(X_{s}^{(j)}, \mu^J_{s}\Bigr) - \sigma\Bigl(\overline X_{s}^{(j)}, \mfldis_{s} \Bigr)  \right\rVert_{\rm F}^2 \, \d s \right)^{\frac{p}{2}} \\
        \notag
         & \leq (2T)^{p-1} \expect \int_{0}^{t} \left\lvert b\Bigl(X_{s}^{(j)}, \mu^J_{s}\Bigr) - b\Bigl(\overline X_{s}^{(j)}, \mfldis_{s} \Bigr) \right\rvert^p \, \d s                                                                \\
         & \qquad + \, C_{\rm BDG} 2^{p-1} T^{\frac{p}{2}-1}  \expect \int_{0}^{t} \left\lVert  \sigma\Bigl(X_{s}^{(j)}, \mu^J_{s}\Bigr) - \sigma\Bigl(\overline X_{s}^{(j)}, \mfldis_{s} \Bigr)  \right\rVert_{\rm F} ^p \, \d s.
        \label{eq:cbo-mfl-pf:BDG}
    \end{align}
    By the triangle inequality,
    the first expectation on the right-hand side can be bounded as follows:
    \begin{align}
        \notag
        \expect \int_{0}^{t} \left\lvert b\Bigl(X_{s}^{(j)}, \mu^J_{s}\Bigr) - b\Bigl(\overline X_{s}^{(j)}, \mfldis_{s} \Bigr) \right\rvert^p \, \d s
         & \leq 2^{p-1} \expect \int_{0}^{t} \left\lvert b\Bigl(X_{s}^{(j)}, \mu^J_{s}\Bigr) - b\Bigl(\overline X_{s}^{(j)}, \overline \mu^J_{s}\Bigr) \, \right\rvert^p \d s                   \\
        \label{eq:decomposition_drift_indicator}
         & \qquad + 2^{p-1} \expect  \int_{0}^{t} \left\lvert b\Bigl(\overline X_{s}^{(j)}, \overline \mu^J_{s}\Bigr) - b\Bigl(\overline X_{s}^{(j)}, \mfldis_{s}\Bigr) \right\rvert^p \, \d s.
    \end{align}

    \paragraph{Bounding the first term in~\eqref{eq:decomposition_drift_indicator}}
    From \eqref{eq:cbo-drift-and-diff} we have
    \begin{align}
        \notag
        \expect \left\lvert b\Bigl(X_{s}^{(j)}, \mu^J_{s}\Bigr) - b\Bigl(\overline X_{s}^{(j)}, \overline \mu^J_{s}\Bigr) \right\rvert^p
         & = \expect \left\lvert   \Bigl( \xnl{j}_{s} - \xn{j}_{s} + \wm\bigl(\mu^J_{s} \bigr) - \wm\bigl(\overline \mu^J_{s} \bigr) \Bigr) \right\rvert^p \\
        \label{eq:auxiliary_mfl_cbo_b}
         & \leq 2^{p-1}  \expect \left\lvert \xn{j}_{s} - \xnl{j}_{s} \right\rvert^p
        + 2^{p-1}  \expect \left\lvert \wm\bigl(\mu^J_{s} \bigr) - \wm\bigl(\overline \mu^J_{s} \bigr)  \right\rvert^p.
    \end{align}
    In order to deal with the fact that the weighted mean does not have global Lipschitz properties in the sense of McKean's framework \cite[Theorem 3.1]{ReviewChaintronII},
    we introduce for each $t\in [0,T]$ and $J\in\N^+$ the set
    \begin{equation}
        \label{eq:cbo-bad-set}
        \Omega_{J,t} \coloneq \left\{ \omega \in \Omega : \frac{1}{J} \sum_{j=1}^{J} \myabs*{\xnl{j}_t(\omega)}^p \geq R \right\},
    \end{equation}
    where
    \[
        Z_j=\sup_{t\in[0,T]}\myabs*{\xnl{j}_t}^p,
    \]
    and $R > \expect [Z_j]$ is fixed. By \cref{thm:cbo-mean-field-sde-well-posed} it holds that
    \(
    \expect \left[ \lvert Z_j \rvert^{q/p} \right] < \infty.
    \)
    Note that $q\ge 2p$. Therefore, by \cref{lemma:small_set}
    there exists a constant~$C>0$ such that
    \begin{align*}
        \forall t \in [0, T],
        \qquad \forall J \in \N^+, \qquad
        \proba \left[  \Omega_{J,t} \right]
        \le \proba \left[ \frac{1}{J} \sum_{j=1}^{J} Z_j \geq R \right]
        \le  C J^{-\frac{q}{2p}}.
    \end{align*}
    We split the second term on the right-hand of \eqref{eq:auxiliary_mfl_cbo_b} as follows:
    \begin{align}
        \notag
        \expect \left\lvert \wm\bigl(\mu^J_{s} \bigr) - \wm\bigl(\overline \mu^J_{s} \bigr)  \right\rvert^p
         & = \expect \left[ \left\lvert \wm\bigl(\mu^J_{s} \bigr) - \wm\bigl(\overline \mu^J_{s} \bigr)  \right\rvert^p \mathsf 1_{\Omega \setminus \Omega_{J,s}}\right] \\
         & \qquad + \expect \left[ \left\lvert \wm\bigl(\mu^J_{s} \bigr) - \wm\bigl(\overline \mu^J_{s} \bigr)  \right\rvert^p \mathsf 1_{\Omega_{J,s}}\right].
        \label{eq:three_terms}
    \end{align}
    Let us now bound each term on the right-hand side of \eqref{eq:three_terms} separately.
    \begin{itemize}
        \item
              Since the $p$-th moment of $\overline \mu_{s}^J$ is bounded from above by $R$ for all $\omega \in \Omega\setminus\Omega_{J,s}$,
              we can use the stability estimate on the weighted mean given (see \cref{item:cbo:stab:improved} and \cref{lemma:stab:improved}) to obtain
              \[
                  \expect \left[ \left\lvert \wm\bigl(\mu^J_{s} \bigr) - \wm\bigl(\overline \mu^J_{s} \bigr)  \right\rvert^p \mathsf 1_{\Omega \setminus \Omega_{J,s}}\right]
                  \leq C \expect \left[ \wasserstein_p\Bigl(\mu^J_{s}, \overline \mu^J_{s}\Bigr)^p \right]
              \]
              for the first term in \eqref{eq:three_terms}.
              Since in addition
              \[
                  \expect \left[ \wasserstein_p\Bigl(\mu^J_{s}, \overline \mu^J_{s}\Bigr)^p \right]
                  \leq \expect \left[ \frac{1}{J} \sum_{j=1}^{J} \left\lvert X_{s}^{(j)} - \xnl{j}_{s}  \right\rvert^p \right]
                  = \expect \left\lvert \xn{j}_s - \xnl{j}_s  \right\rvert^p,
              \]
              it follows that
              \[
                  \expect \left[ \left\lvert \wm\bigl(\mu^J_{s} \bigr) - \wm\bigl(\overline \mu^J_{s} \bigr)  \right\rvert^p \mathsf 1_{\Omega \setminus \Omega_{J,s}}\right]
                  \leq C \expect  \left\lvert \xn{j}_s - \xnl{j}_s  \right\rvert^p.
              \]

        \item
              We now bound the second term in~\eqref{eq:three_terms}. From \cref{item:cbo:moment-bounds-empirical} (see \cref{lemma:moment-bounds-empirical}) as well as \eqref{eq:cbo-moment-bound-mfl} and \cref{lemma:bound-on-weighted-moment} we have
              \[
                  \expect \left[ \left\lvert \wm\bigl(\mu^J_{s} \bigr) - \wm\bigl(\overline \mu^J_{s} \bigr)  \right\rvert^{q}\right]
                  \leq
                  2^{q-1} \left( \expect  \left\lvert \wm\bigl(\mu^J_{s} \bigr) \right\rvert^q   + \expect  \left\lvert \wm\bigl(\overline \mu^J_{s} \bigr)  \right\rvert^{q} \right)
                  \leq 2^{q-1} \left( \kappa + \overline \kappa\right)
              \]
              where  $\overline \kappa:= C_2 \expect \left[ \sup_{t\in[0,T]}  \myabs*{\xl_{t}}^q \right]
              < \infty$. Here, $C_2>0$ is the constant from \cref{lemma:bound-on-weighted-moment}.
              Therefore,
              we obtain by Hölder's inequality and \cref{lemma:small_set} that
              \begin{align*}
                  \expect \left[ \left\lvert \wm\bigl(\mu^J_{s} \bigr) - \wm\bigl(\overline \mu^J_{s} \bigr)  \right\rvert^p \mathsf 1_{\Omega_{J,s}}\right]
                  \le
                  \left(\expect \left[ \left\lvert \wm\bigl(\mu^J_{s} \bigr) - \wm\bigl(\overline \mu^J_{s} \bigr) \right\rvert^{q}\right]\right)^{\frac{p}{q}} \left(\proba \left[ \Omega_{J,s} \right] \right)^{^{\frac{q-p}{q}}}
                  \le  2^{\frac{(q-1)p}{q}}
                  \left( \kappa + \overline \kappa\right)^{\frac{p}{q}} J^{-\frac{q-p}{2p}}.
              \end{align*}
    \end{itemize}
    We have thus shown that
    \[
        \expect \left\lvert \wm\bigl(\mu^J_{s} \bigr) - \wm\bigl(\overline \mu^J_{s} \bigr)  \right\rvert^p
        \leq C\expect \left\lvert \xn{j}_s - \xnl{j}_s \right\rvert^p + C J^{-\frac{q -p}{2p}},
    \]
    which in turn by~\eqref{eq:auxiliary_mfl_cbo_b} implies that
    \[
        \expect \left\lvert b\Bigl(X_{s}^{(j)}, \mu^J_{s}\Bigr) - b\Bigl(\xnl{j}_{s}, \overline \mu^J_{s}\Bigr) \right\rvert^p
        \leq C\expect \left\lvert \xn{j}_s - \xnl{j}_s  \right\rvert^p + C J^{-\frac{q -p}{2p}}.
    \]

    \paragraph{Bounding the second term in~\eqref{eq:decomposition_drift_indicator}}
    For the second term on the right-hand side of~\eqref{eq:decomposition_drift_indicator},
    we have by \cref{item:cbo:convergence_weighted_mean_iid} (see \cref{lemma:convergence_weighted_mean_iid}) that
    \[
        \expect \left\lvert b\Bigl(\overline X_{s}^{(j)}, \overline \mu^J_{s}\Bigr) - b\Bigl(\overline X_{s}^{(j)}, \mfldis_{s}\Bigr) \right\rvert^p
        =  \expect \left\lvert \wm \bigl(\overline \mu^J_{s}\bigr) - \wm \bigl(\mfldis_s \bigr) \right\rvert^p
        \leq C J^{-\frac{p}{2}}.
    \]

    \paragraph{Conclusion}
    Combining the bounds on both terms in~\eqref{eq:decomposition_drift_indicator}, we deduce
    \[
        \expect \left[ \int_{0}^{t} \left\lvert b\Bigl(X_{s}^{(j)}, \mu^J_{s}\Bigr) - b\Bigl(\overline X_{s}^{(j)}, \mfldis_{s} \Bigr) \right\rvert^p \, \d s \right]
        \leq C J^{- \min \left\{ \frac{p}{2}, \frac{q -p}{2p} \right\} } + C \int_{0}^t \expect \left[ \sup_{u \in [0,s]} \left\lvert \xn{j}_u - \xnl{j}_u \right\rvert^p \right] \d s.
    \]
    Using that $S\colon \R\rightarrow \R^{d\times d}$ is globally Lipschitz continuous, we obtain analogously
    \[
        \expect \left[ \int_{0}^{t} \left\lVert \sigma \Bigl(X_{s}^{(j)}, \mu^J_{s}\Bigr) - \sigma\Bigl(\overline X_{s}^{(j)}, \mfldis_{s} \Bigr) \right\rVert^p \, \d s \right]
        \leq C J^{- \min \left\{ \frac{p}{2}, \frac{q -p}{2p} \right\} } + C \int_{0}^t \expect \left[ \sup_{u \in [0,s]} \left\lvert \xn{j}_u - \xnl{j}_u \right\rvert^p \right] \d s.
    \]
    Substituting in~\eqref{eq:cbo-mfl-pf:BDG},
    we finally obtain
    \begin{align*}
        \expect \left[ \sup_{s \in [0, t]} \bigl\lvert X_{s}^{(j)} - \overline X_{s}^{(j)} \bigr\rvert^p \right]
         & \leq C J^{- \min \left\{ \frac{p}{2}, \frac{q -p}{2p} \right\} } + C \int_{0}^t \expect \left[ \sup_{u \in [0, s]} \left\lvert \xn{j}_u - \xnl{j}_u \right\rvert^p \right] \d s.
    \end{align*}
    The conclusion follows from Grönwall's lemma,
    since $\frac{q-p}{2p^2} \leq \frac{q-(2 \vee p_{\mathcal M})}{2(2 \vee p_{\mathcal M})^2}$
    for the value of $p$ considered.
\end{proof}

\begin{remark}
    The proof is similar to that of McKean's theorem \cite[Theorem 3.1]{ReviewChaintronII}, with an important difference:
    we eliminate a set of small probability in order to apply the local stability estimate in  \cref{item:cbo:stab:improved} on its complement.

    A crucial detail here is that \cref{item:cbo:stab:improved} only assumes that $(\mu, \nu) \in \mathcal P_{p,R} \times \mathcal P_{p}$ and not~$(\mu, \nu) \in\mathcal P_{p,R} \times \mathcal P_{p,R}$;
    see \cref{lemma:stab:improved} below.
    This allows us to define $\Omega_{J,t}$ in \eqref{eq:cbo-bad-set} only in terms of the i.i.d.\ mean-field particles~$\xnl{j}_t$ for which \cref{lemma:small_set} can be applied;
    the event $\frac{1}{J} \sum_{j=1}^{J} \myabs*{\xn{j}_t}^p \geq R$ does not need to be included in the definition of $\Omega_{J,t}$.
    This trick allows to show the strong mean-field limit result~\eqref{eq:thm:mfl-cbo_sharp}, improving on the previous results from \cite{CBO-mfl-Huang2021, CBO-hypersurface-Fornasier_2020, fornasier2021consensusbased}.
\end{remark}

\subsection{Quantitative mean-field limit for CBS}
\label{sec:mfl_cbs}

Proving a mean-field limit for the CBS dynamics~\eqref{eq:cbs-particles} is more challenging than for CBO because of the square root of the weighted covariance in the diffusion coefficient.
Thanks to our novel stability estimate for the square root of the weighted covariance (see \cref{lemma:stab:improved} below), we can prove a mean-field result for the CBS dynamics~\eqref{eq:cbs-particles} in the same way as for CBO.

\begin{theorem}
    [Mean-field limit for CBS]
    \label{thm:mfl-cbs}
    Suppose that $f \in \mathcal A(s, \ell, \ell)$ with $s,\ell\ge 0$ such that $\ell\le s+1$ and that $\mfldis_0\in \mathcal P(\R^d)$ with~$\mathcal C(\mfldis_0) \succ 0$ has moments of all orders.
    Consider the systems \eqref{eq:coupling-system-micro} and \eqref{eq:coupling-system-mfl} with the coefficients given by
    \begin{align}
        \label{eq:cbs-drift-and-diff}
        b(x, \mu) \coloneq - \bigl(x - \wm\left(\mu \right) \bigr) &  & \text{and} &  &
        \sigma(x, \mu) \coloneq \sqrt{2 \lambda^{-1} \wc(\mu)}.
    \end{align}
    Then for all $p >0$, there is $C > 0$ independent of~$J$ such that
    \begin{equation}
        \label{eq:thm:mfl-cbs}
        \forall J\in\N^+, \qquad
        \forall j\in\range{1}{J},\qquad
        \left(\expect \left[ \sup_{t\in[0,T]} \left\lvert \xn{j}_t-\xnl{j}_t \right\rvert^p \right]\right)^{\frac{1}{p}}
        \le C J^{-\frac{1}{2}}.
    \end{equation}
\end{theorem}
\begin{remark}
    \label{remark:cbs-mfl-sharper}
    In the same way as for CBO, it is possible to prove the following more general statement.
    Suppose that $f \in \mathcal A(s, \ell, \ell)$ with $s,\ell\ge 0$ such that $\ell\le s+1$. Assume that $\mfldis_0\in \mathcal P_q(\R^d)$ with $\mathcal C(\mfldis_0) \succ 0$.
    If~$q \geq 4 p_{\mathcal M}(s, \ell)$,
    then for all~$p \in (0, \frac{q}{2}]$, there is $C > 0$ independent of~$J$ such that
    \begin{equation}
        \label{eq:thm:mfl-cbs_sharp}
        \forall J\in\N^+, \qquad
        \forall j \in \range{1}{J}, \qquad
        \left(\expect \left[ \sup_{t\in[0,T]} \myabs[\Big]{\xn{j}_t-\xnl{j}_t}^p \right]\right)^{\frac{1}{p}}
        \le C J^{-  \min \left\{ \frac{1}{2}, \frac{q-p}{2p^2}, \frac{q-2p_{\mathcal M}}{2 (2p_{\mathcal M})^2} \right\}}.
    \end{equation}
\end{remark}

We summarize below the additional auxiliary results necessary for the proof,
which are valid for an objective function~$f\in  \mathcal A(s, \ell, \ell)$.
These results are stated precisely and proved in~\cref{sec:aux}.
\begin{itemize}
    \item
          \textbf{Wasserstein stability estimate for the weighted covariance}.
          The crucial ingredient for the proof of~\cref{thm:mfl-cbs} is a stability estimate for the square root of the weighted covariance.
          We prove in \cref{lemma:stab:improved} that,
          for all~$R>0$ and for all $p \geq  2 p_{\mathcal M}(\ell, s)$,
          there exists $\widetilde L_{\sqrt{\mathcal C}}(R,p)>0$ such that
          \begin{align*}
              \forall (\mu, \nu) \in \mathcal P_{p,R} \bigl(\real^d\bigr) \times \mathcal P_p \bigl(\real^d\bigr), \qquad
              \left\lVert \sqrt{\wc(\mu)} - \sqrt{\wc(\nu)} \right\rVert
              \le \widetilde L_{\sqrt{\mathcal C}} \, \wasserstein_p(\mu, \nu).
          \end{align*}

    \item
          \textbf{Moment estimate for the particle system}.
          We prove in~\cref{lemma:moment-bounds-empirical} that,
          under the assumptions of~\cref{thm:mfl-cbs},
          there exists a constant $\kappa >0$ independent of $J$
          such that for all $J \in\N^+, \qquad$
          \[
              \expect \left[ \sup_{t\in[0,T]} \left\lvert X_t^{(j)} \right\rvert^{q} \right]
              \quad \vee \quad
              \expect \left[ \sup_{t\in[0,T]}\left\lvert \wm(\mu^J_t) \right\rvert^{q} \right]
              \quad \vee \quad
              \expect \left[ \sup_{t\in[0,T]}\left\lvert \sqrt{\wc(\mu^J_t)} \right\rvert^{q} \right]
              \leq \kappa.
          \]

    \item \textbf{Convergence of the square root of the weighted covariance for i.i.d.\ samples.}
          We prove in~\cref{lemma:convergence_weighted_covariance_iid} that,
          for any reals $r> 2p > 0$ and any $\mu \in \mathcal P_r(\R^d)$ such that~$\mathcal C_{\beta}(\mu) \succ 0$,
          then there is~$C$ such that
          \[
              \expect
              \left\lVert
              \sqrt{\wc(\overline \mu^J)} - \sqrt{\wc(\mu)}
              \right\rVert^p
              \leq C J^{- \frac{p}{2}},
              \qquad
              \overline \mu^J \coloneq \frac{1}{J} \sum_{j=1}^{J} \delta_{\xnl{j}},
              \qquad
              \left\{ \xnl{j} \right\}_{j \in \N} \stackrel{\rm{i.i.d.}}{\sim} \mu.
          \]
    \item \textbf{No collapse in finite time.}
          In order to apply the third item with~$\mu = \mfldis_t$,
          we prove in \cref{prop:cbs-minimal-eigenvalue-wcov} that the
          assumption~$\mathcal C(\mfldis_0)\succ 0$ implies that
          \[
              \eta \coloneq \inf_{t \in[0,T]}\lambda_\text{min}\bigl(\wc(\mfldis_t) \bigr) >0.
          \]
\end{itemize}

With these ingredients at hand, we can prove~\cref{thm:mfl-cbs} in the same way as for CBO,
so we skip the proof here.
Let us mention, however,
that a different proof based on a technique using stopping times is presented in~\cref{sub:stopping_times}.

\subsection{Numerical experiment}%
\label{sub:experiment}
In this section,
we present a numerical experiment illustrating our first main result,
on the quantitative mean field limit for CBO; see~\cref{thm:mfl-cbo}.
To this end, we apply CBO to minimize the Ackley function,
a standard benchmark function in optimization given by
\begin{equation}
    \label{eq:ackley}
    f_{\rm Ack}(x) = -20 \exp \left( - \frac{1}{5} \sqrt{\frac{1}{d} \sum_{i=1}^{d} |x_i|^2} \right)
    -\exp \left( \frac{1}{d} \sum_{i=1}^{d} \cos\bigl(2 \pi x_i\bigr) \right) + \e \,+\, 20,
\end{equation}
where $d$ is the dimension.
This function, which is illustrated in~\cref{fig:evolution_errors}, is globally Lipschitz continuous and bounded from above and below uniformly over~$\real^d$,
and so $f_{\rm Ack} \in \mathcal A(0, 0, 0)$.
To illustrate the theorem,
we first approximate the mean field dynamics by simulating the interacting particle system~\eqref{eq:cbo-particles} using a large number~$J_{\infty}$ of particles.
Then, for various smaller values of $J \ll J_{\infty}$,
we simulate the CBO evolution of a system of~$J$ particles,
initialized at the same positions and driven by the same Brownian motions as the first~$J$ particles of the large system of size~$J_{\infty}$.
In all cases, we use CBO with isotropic noise, discretized using the Euler--Maruyama method with time step~$\Delta t$.
All the parameters used in this experiment are given in~\cref{table:training_params}.

To exhibit the scaling as~$J^{-\frac{1}{2}}$ in~\cref{thm:mfl-cbo},
we fix $p = 2$ and approximate the left-hand side of~\eqref{eq:thm:mfl-cbo}
using a Monte Carlo approach with $M$ independent realizations.
More precisely,
omitting discretization for simplicity,
we use the following estimator~$\widehat E(J)$ of~$E(J)$, where
\[
    E(J) \coloneq \expect \left[ \sup_{t\in[0,T]} \myabs[\Big]{\xn{j}_t-\xnl{j}_t}^2 \right],
    \qquad
    \widehat E(J) \coloneq \frac{1}{MJ} \sum_{m=1}^{M} \sum_{j=1}^{J} \sup_{t\in[0,T]} \myabs[\Big]{\xn{j/J,m}_t-\xn{j/J_{\infty},m}_t}^2.
\]
Here $\xn{j/J,m}_t$ denotes the position of the $j$-th particle in the system of size~$J$,
for the $m$-th independent realization of the particle systems.

\begin{table}[htb]
    \footnotesize
    \begin{center}
        \begin{tabular}{ccc}
            \toprule
            \textbf{Parameter}                                       & \textbf{Notation} & \textbf{Value}                                \\
            \midrule
            \phantom{\LARGE$\int$} Dimension                         & $d$               & 2                                             \\
            \phantom{\LARGE$\int$} Inverse temperature           & $\beta$           & $3$                                             \\
            \phantom{\LARGE$\int$} Diffusion coefficient             & $\sigma$          & $\frac{1}{5}$                                             \\
            \phantom{\LARGE$\int$} Initial law                       & $\mfldis_0$       & $\normal (0, \I_2)$                            \\
            \phantom{\LARGE$\int$} Time step                         & $\Delta t$        & 0.01                                          \\
            \phantom{\LARGE$\int$} Final time                        & $T$               & 1 \\
            \phantom{\LARGE$\int$} Size for mean field approximation & $J_{\infty}$      & $10^6$                                          \\
            \phantom{\LARGE$\int$} Number of independent simulations & $M$               & $100$                                           \\
            \bottomrule
        \end{tabular}
        \caption{Parameters of the numerical experiment described in~\cref{sub:experiment}.}
        \label{table:training_params}
    \end{center}
\end{table}

The results of the numerical experiment are illustrated in the right panel~\cref{fig:evolution_errors}.
The expected scaling~$E(J) \propto \frac{1}{J}$ is clearly observed for large values of~$J$.
We note, however, that this scaling of the error holds only for relatively large~$J$.
For small~$J$, the effective number of particles contributing to the weighted average,
as quantified by the effective sample size~\cite{MR3696003},
is close to 1 in the early stages of the simulation.
Therefore, the Monte Carlo scaling is not expected in this regime.

\begin{figure}[ht]
    \centering
    \includegraphics[width=0.4\linewidth]{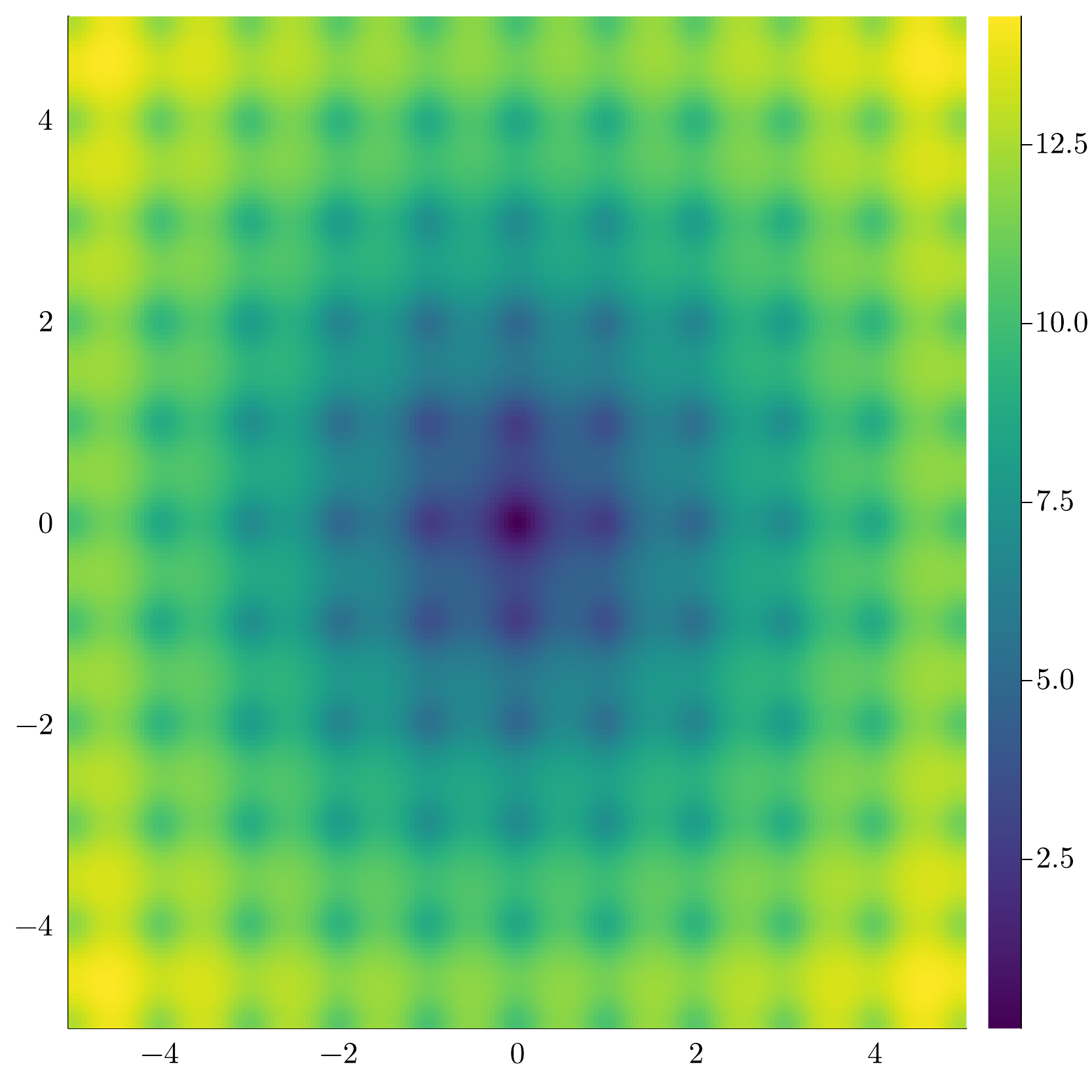}
    \includegraphics[width=0.49\linewidth]{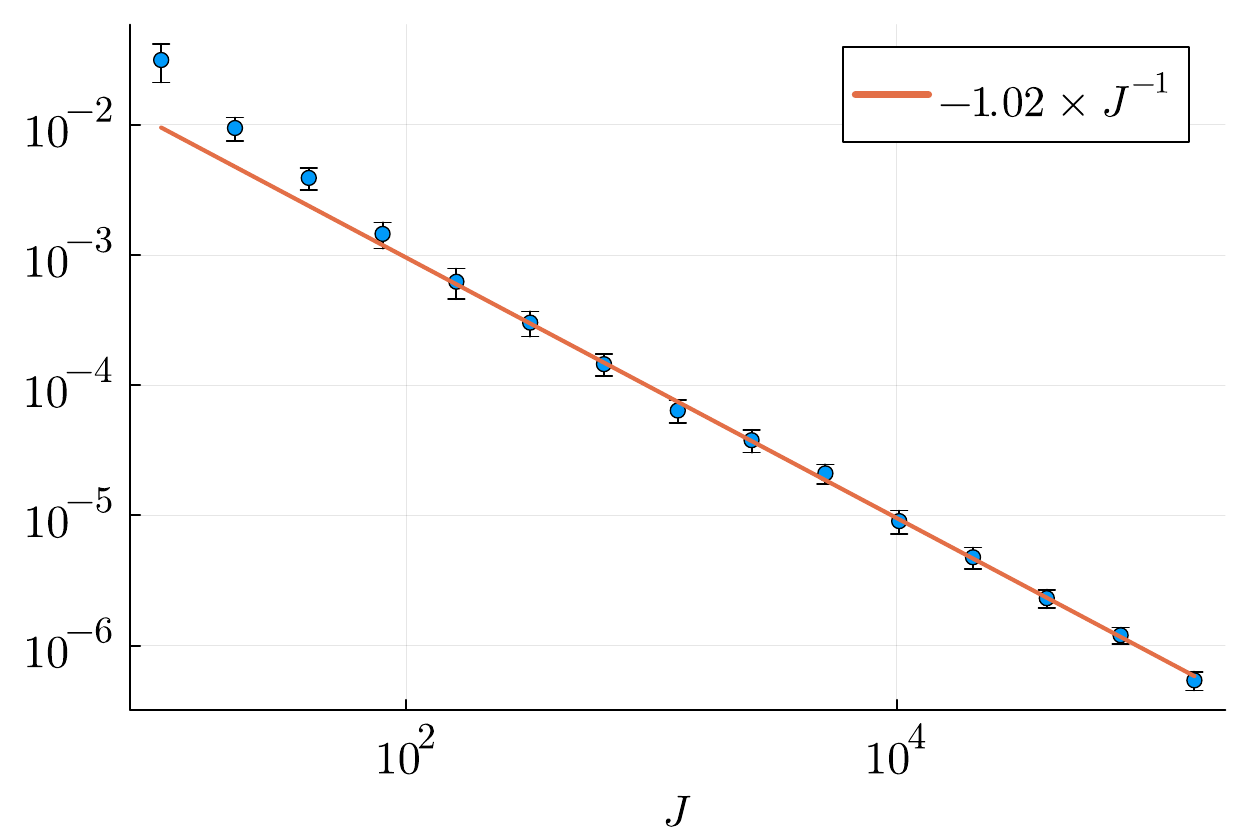}
    \caption{
            {\bf Left:} Ackley function.
            {\bf Right:} Monte Carlo estimator~$\widehat E(J)$ of~$E(J)$
            for $J\in \{10 \cdot 2^k\}_{k=0}^{14}$,
            and straight line exhibiting the asymptotic scaling as~$J^{-1}$ of~$E(J)$.
            Error bars corresponding to two standard deviations are depicted on both sides of the Monte Carlo estimator.
    }%
    \label{fig:evolution_errors}
\end{figure}

\section{Auxiliary results}
\label{sec:aux}

This section is organized as follows.
In~\cref{sub:stability_estimates}, we improve on existing Wasserstein stability estimates for the weighted mean and prove novel ones for
the weighted covariance
and the square root of the weighted covariance.
Then, in~\cref{sub:moment},
we present moment estimates for the empirical measures associated with the~CBO~\eqref{eq:cbo-particles} and CBS~\eqref{eq:cbs-particles} dynamics.
In~\cref{sub:iid},
we prove convergence estimates for the weighted moments when particle positions are drawn independently from a given probability distribution,
using results from the statistics literature.
Finally, in~\cref{sub:no_collapse},
we prove that if $\wc(\overline \rho_0) \succ 0$,
then the covariance matrix~$\wc(\overline \rho_t) \succ 0$ associated with the solution to the mean-field CBS dynamics~\eqref{eq:cbs-mf} remains strictly positive definite for all times.

\subsection{Wasserstein stability estimates for the weighted moments}
\label{sub:stability_estimates}

In this subsection, we generalize \cite[Lemma 3.2]{carrillo2018analytical} to the more general conditions in \cref{assump:main} and,
more importantly, to the weighted covariance and its square root.

\begin{proposition}
    [Basic stability estimates for the weighted mean and the weighted covariance]
    \label{lemma:stab:wmean-wcov-basic}
    Suppose that $f \in \mathcal A(s, \ell, u)$ with $s \ge 0$ and $u \geq \ell \geq 0$.
    Then:
    \begin{enumerate}[label=(\alph*)]
        \item \label{item:lem-stab-wmean}
              For all $R>0$ and for all $p\ge p_{\mathcal M}(s, \ell)$,
              where~$p_{\mathcal M}(s, \ell)$ is given in~\eqref{eq:def_ps},
              there exists $L_{\mathcal M} = L_{\mathcal M}(p,R)>0$ such that
              \begin{equation}
                  \label{eq:stab_wmean_simple}
                  \forall (\mu, \nu) \in \mathcal P_{p,R}\bigl(\real^d\bigr) \times \mathcal P_{p,R}\bigl(\real^d\bigr), \qquad
                  \myabs{\wm(\mu) - \wm(\nu)}
                  \le L_{\mathcal M} \, \wasserstein_p(\mu, \nu).
              \end{equation}
        \item \label{item:lem-stab-wcov}
              For all $R > 0$ and for all $p\ge p_{\mathcal C}(s, \ell)$,
              where~$p_{\mathcal C}(s, \ell)$ is given in~\eqref{eq:def_ps},
              there exists $L_{\mathcal C} = L_{\mathcal C}(p, R)>0$ such that
              \begin{equation}
                  \label{eq:stab_wcov_simple}
                  \forall (\mu, \nu) \in \mathcal P_{p,R}\bigl(\real^d\bigr) \times \mathcal P_{p,R} \bigl(\real^d\bigr), \qquad
                  \Bigl\lVert \wc(\mu) - \wc(\nu) \Bigr\rVert_{\rm F}
                  \le L_{\mathcal C} \, \wasserstein_p(\mu, \nu).
              \end{equation}

                  \end{enumerate}
\end{proposition}

\begin{proof}
    The proof for both \cref{item:lem-stab-wmean} and then \cref{item:lem-stab-wcov} relies on the technical result~\cref{lemma:stab-aux} in the appendix.
    \paragraph{Proof of \cref{item:lem-stab-wmean}}
    In the notation of \cref{lemma:stab-aux} we can write
    \[
        \wm(\mu)= \frac{\mu[g]}{\mu[h]}, \qquad g(x):=\e^{-\beta f(x)}x, \qquad h(x) \coloneq \e^{-\beta f(x)}.
    \]
    Condition~\cref{item:lem-stab-aux:loc-lip} is satisfied for~$\varsigma := s + 1$ if~$\ell = 0$ or~$\varsigma := 0$ if~$\ell > 0$.
    Therefore, the claim follows from \cref{lemma:stab-aux}.

    \paragraph{Proof of \cref{item:lem-stab-wcov}}
    First observe that
    \begin{align*}
        \wc(\mu)
        =\frac{\mu[g]}{\mu[h]}
        - \wm(\mu) \otimes \wm(\mu),
        \qquad g(x):= \e^{-\beta f(x)} (x \otimes x), \qquad h(x):=\e^{-\beta f(x)}.
    \end{align*}
    By the triangle inequality,
    we have that
    \begin{equation}
        \label{eq:triangle_inequality_wcov_bound}
        \bigl\lVert \wc(\mu) - \wc(\nu) \bigr\rVert_{\rm F}
        \leq
        \Bigl\lVert \wm(\mu) \otimes \wm(\mu) - \wm(\nu) \otimes \wm(\nu) \Bigr\rVert_{\rm F}
        + \left\lVert \frac{\mu[g]}{\mu[h]} - \frac{\nu[g]}{\nu[h]} \right\rVert_{\rm F}\,.
    \end{equation}
    We shall use that for all $(x, y) \in \R^d \times \R^d$, it holds that
    \begin{align}
        \notag
        \left\lVert x \otimes x - y \otimes y \right\rVert_{\rm F}
         & = \left\lVert x (x - y)^\t + (x - y) y^\t \right\rVert_{\rm F}                                     \\
        \label{eq:aux_weighted_mean}
         & \leq \left\lVert x (x - y)^\t \right\rVert_{\rm F} + \left\lVert (x - y) y^\t \right\rVert_{\rm F}
        = \myabs{x} \cdot \myabs{x - y} + \myabs{y} \cdot \myabs{x - y}
        = \bigl(\myabs{x} + \myabs{y}\bigr) \myabs{x - y}.
    \end{align}
    Using this inequality,
    that $p_{\mathcal C} \ge p_{\mathcal M}$ and \cref{item:lem-stab-wmean},
    we obtain for the first term in~\eqref{eq:triangle_inequality_wcov_bound} that
    \begin{align*}
        \Bigl\lVert
        \wm(\mu) \otimes \wm(\mu)
        - \wm(\nu) \otimes \wm(\nu)
        \Bigr\rVert_{\rm F}
         & \le \Bigl(\myabs*{\wm(\mu) } + \myabs*{\wm(\nu)}\Bigr)
        \cdot \bigl\lvert \wm(\mu) - \wm(\nu) \bigr\rvert                                               \\
         & \le L_{\mathcal M} \, \Bigl(\myabs*{\wm(\mu) } + \myabs*{\wm(\nu)}\Bigr) \, \wasserstein_p(\mu, \nu)    \\
         & \le L_{\mathcal M}^2 \, \Bigl(\wasserstein_p(\mu, \delta_0) + W_p(\nu, \delta_0)\Bigr) \, \wasserstein_p(\mu, \nu)
        \le 2 L_{\mathcal M}^2 R \, \wasserstein_p(\mu, \nu).
    \end{align*}
    In order to bound the second term in~\eqref{eq:triangle_inequality_wcov_bound},
    we again use~\eqref{eq:aux_weighted_mean} to
    deduce that the function~$g$ satisfies the local Lipschitz estimate
    \[
        \forall x, y \in \R^d, \qquad
        \bigl\lvert g(x)-g(y) \bigr\rvert \le L \bigl( 1 + |x|+|y|\bigr)^{\varsigma} \myabs{x-y},
        \qquad \varsigma :=
        \begin{cases}
            s+2 \quad & \text{ if~$\ell = 0$}, \\
            0 \quad   & \text{ if~$\ell > 0$}.
        \end{cases}
    \]
    We deduce that the functions~$g$ and~$h$ satisfy the assumption~\eqref{item:lem-stab-aux:loc-lip} of~\cref{lemma:stab-aux} with this~$\varsigma$.
    Therefore, it holds that
    \begin{align*}
        \forall (\mu, \nu) \in \mathcal P_{p,R}\bigl(\real^d\bigr) \times \mathcal P_{p,R}\bigl(\real^d\bigr),
        \qquad
        \left\lVert
        \frac{\mu[\e^{-\beta f(x)} x \otimes x]}{\mu[\e^{-\beta f(x)}]}
        - \frac{\nu[\e^{-\beta f(x)} x \otimes x]}{\nu[\e^{-\beta f(x)}]}
        \right\rVert_{\rm F}
        \le C \wasserstein_p(\mu, \nu),
    \end{align*}
    which completes the proof.
\end{proof}

\begin{proposition}
    [Basic stability estimate for the square root covariance]
    \label{proposition:stability:sqrt-wcov}
    Suppose that $f \in \mathcal A(s, \ell, u)$ with~$s \geq 0$ and $u \geq \ell \geq 0$.
    Then for all $R > 0$ and all $p \geq 2 p_{\mathcal M}(s, \ell)$,
    there exists~$L_{\sqrt{\mathcal C}} = L_{\sqrt{\mathcal C}}(p, R) >0$ such that
    \begin{align}
        \label{eq:wmean-wcov-emp-local-lip}
        \forall (\mu, \nu) \in \mathcal P_{p,R} \bigl(\real^d\bigr) \times \mathcal P_{p,R} \bigl(\real^d\bigr),
        \qquad
        \norm{
            \sqrt{\wc(\mu)} -
            \sqrt{ \wc (\nu) }}_{\rm F}
        \le L_{\sqrt{\mathcal C}} \, \wasserstein_p(\mu, \nu).
    \end{align}
\end{proposition}

\begin{proof}
    It is sufficient to check the claim for measures $\mu$ and $\nu$ of the form
    \begin{equation}
        \label{eq:form_empirical}
        \mu^J = \frac{1}{J} \sum_{j=1}^{J} \delta_{\xn{j}},
        \qquad
        \nu^J = \frac{1}{J} \sum_{j=1}^{J} \delta_{\yn{j}},
        \qquad J \in \N^+.
    \end{equation}
    Indeed,
    assume the claim holds for all such pairs of probability measures,
    fix $R > 0$ and take any pair of probability measures $(\mu, \nu) \in \mathcal P_{p,R} \times \mathcal P_{p,R}$.
    By~\cite[Theorem 6.18]{MR2459454},
    there exists a sequence $\bigl\{(\mu^J, \nu^J)\bigr\}_{J \in \N^+}$ in $\mathcal P_{p} \times \mathcal P_{p}$
    such that $\wasserstein_p(\mu^J, \mu) \to 0$ and~$\wasserstein_p(\nu^J, \nu) \to 0$ in the limit as $J \to \infty$.
    Since convergence in $\wasserstein_p$ implies convergence of the moments of order up to~$p$ by~\cite[Theorem 6.9]{MR2459454},
    we can assume, discarding the beginning of the sequence if necessary,
    that $(\mu^J, \nu^J) \in \mathcal P_{p,2R}$ for all $J$.
    Then
    \begin{align*}
        \norm{
            \sqrt{\wc(\mu}) -
            \sqrt{\wc (\nu)}
        }_{\rm F}
        \leq
         & \norm{
            \sqrt{\wc(\mu^J)} -
            \sqrt{\wc(\nu^J)}
        }_{\rm F}   \\
         & \qquad +
        \norm{
            \sqrt{\wc(\nu)} -
            \sqrt{\wc(\nu^J)}
        }_{\rm F}
        +
        \norm{
            \sqrt{\wc(\mu)} -
            \sqrt{\wc(\mu^J)}
        }_{\rm F}.
    \end{align*}
    The first term is bounded from above by $C  \wasserstein_p(\mu^J, \nu^J)$ by the base case,
    while the other two terms converge to 0 in the limit as~$J \to \infty$ by~\cref{lemma:stab:wmean-wcov-basic}.
    Taking the limit $J \to \infty$, we deduce that
    \[
        \norm{
            \sqrt{\wc(\mu)} -
            \sqrt{\wc (\nu)}
        }_{\rm F}
        \leq
        C  \wasserstein_p(\mu, \nu).
    \]

    \paragraph{Proof of the statement for empirical measures}
    Fix $R > 0$.
    By~\cite[p.5]{MR1964483},
    the Wasserstein distance between empirical measures $\mu^J$ and $\nu^J$ of the form~\eqref{eq:form_empirical}
    is equal to
    \[
        \wasserstein_p(\mu^J, \nu^J)
        = \min_{\sigma \in \mathcal S_J} \left(\frac{1}{J} \sum_{j=1}^{J} \left\lvert \xn{j} - Y^{\sigma(j)} \right\rvert^p\right)^{\frac{1}{p}},
    \]
    where $\mathcal S_J$ denotes the set of permutations in $\{1, \dotsc, J\}$.
    Thus, the claim will follow if we can prove that,
    for any pair $(\mu^J, \nu^J) \in \mathcal P_{p,R}$ of the form~\eqref{eq:form_empirical},
    it holds that
    \begin{equation}
        \label{eq:target_inequality}
        \norm{
            \sqrt{\wc(\mu^J)} -
            \sqrt{\wc (\nu^J)}
        }_{\rm F}
        \leq
        C \, \left(\frac{1}{J} \sum_{j=1}^{J} \left\lvert \xn{j} - \yn{j} \right\rvert^p\right)^{\frac{1}{p}}.
    \end{equation}
    We henceforth drop the superscript~$J$ in $\mu^J, \nu^J$ for simplicity of notation,
    and write~$\mathbf X = \left(\xn{1}, \dotsc, \xn{J}\right)$ and~$\mathbf Y = \left(\yn{1}, \dotsc, \yn{J} \right)$.
    First note that
    \[
        \wc(\mu)
        = \sum_{j=1}^J w_j(\mathbf{X}) \Bigl(\xn{j}-\wm(\mu)\Bigr) \otimes \Bigl(\xn{j}-\wm(\mu)\Bigr),
        \qquad w_j(\mathbf{X}):= \frac{ \e^{-\beta f(\xn{j})} }{ \sum_{k=1}^J \e^{-\beta f(\xn{k})}}.
    \]
    It is easily seen that $\wc(\mu) = M_{\mathbf X} M_{\mathbf X}^\t$,
    where $M_{\mathbf X}$ is the following $\R^{d\times J}$ matrix
    \[
        M_{\mathbf X} :=
        \begin{pmatrix}
            \sqrt{w_1(\mathbf{X})}  \left(\xn{1} - \wm(\mu) \right)
             & \hdots &
            \sqrt{w_J(\mathbf{X})}  \left(\xn{J} - \wm(\mu) \right)
        \end{pmatrix}.
    \]
    Proceeding in the same manner,
    we construct a matrix $M_{\mathbf{Y}} \in \R^{d\times J}$ such that $\wm(\nu)=M_{\mathbf{Y}} M_{\mathbf{Y}}^\t$.
    A result by Araki and Yamagami~\cite{Araki1981},
    later generalized by Kittaneh~\cite{MR0787884} and Bhatia~\cite{MR1275617},
    states for any two matrices~$A$ and $B$ with the same shape,
    it holds that
    \[
        \norm{\sqrt{A^\t A}-\sqrt{B^\t B}}_{\rm F} \le \sqrt{2}\norm{A-B}_{\rm F},
    \]
    see also~\cite[Theorem VII.5.7]{MR1477662} for a textbook presentation.
    This yields
    \begin{align*}
         & \norm{
            \sqrt{\wc(\mu)} -
        \sqrt{\wc (\nu) }}_{\rm F}                                                        \\
         & \qquad\le \sqrt{2} \left( \sum_{j=1}^J \myabs*{
            \sqrt{w_j(\mathbf{X})} \left(\xn{j} - \wm(\mu) \right)
        - \sqrt{w_j(\mathbf{Y})} \left(\yn{j} - \wm(\nu) \right)}^2 \right)^{\frac{1}{2}} \\
         & \qquad \le \sqrt{2}
        \left( \sum_{j=1}^J \myabs*{ \sqrt{w_j(\mathbf{X})} \xn{j}
            - \sqrt{w_j(\mathbf{Y})} \yn{j} }^2 \right)^{\frac{1}{2}}
        + \sqrt{2} \left( \sum_{j=1}^J \myabs*{ \sqrt{w_j(\mathbf{X})} \wm(\mu)
            - \sqrt{w_j(\mathbf{Y})} \wm(\nu) }^2 \right)^{\frac{1}{2}},
    \end{align*}
    where we applied the triangle inequality for the Frobenius norm.
    Let $\pi \in \Pi(\mu, \nu) \subset \mathcal P(\R^d \times \R^d)$ denote the probability measure
    \[
        \pi = \frac{1}{J} \sum_{j=1}^{J} \delta_{\left(\xn{j}, \yn{j} \right)}
    \]
    and set $\gamma = \frac{\beta}{2}$.
    We may rewrite the previous inequality as
    \begin{align*}
        \frac{1}{\sqrt{2}} \norm{
            \sqrt{\wc(\mu)} -
            \sqrt{\wc ( \nu) }}_{\rm F}
         & \leq \left( \iint_{\R^d \times \R^d} \left\lvert \frac{\e^{- \gamma f(x)} x}{\sqrt{\mu\left[\e^{-\beta f}\right]}}
        - \frac{\e^{-\gamma f(y)} y} {\sqrt{\nu \left[\e^{-\beta f}\right]}} \right\rvert^2 \, \pi (\d x \, \d y) \right)^{\frac{1}{2}}   \\
         & \qquad + \left( \iint_{\R^d \times \R^d} \left\lvert \frac{\e^{- \gamma f(x)} \wm(\mu)}{\sqrt{\mu \left[\e^{-\beta f}\right]}}
        - \frac{\e^{-\gamma f(y)} \wm(\nu)} {\sqrt{\nu \left[\e^{-\beta f}\right]}} \right\rvert^2 \, \pi (\d x \, \d y) \right)^{\frac{1}{2}}.
    \end{align*}
    The conclusion then follows from~\cref{lemma:aux_aux_rt_cov} in the appendix:
    \begin{itemize}
        \item
              For the first term we let $g(x) = \e^{- \gamma f(x)} x$, $h(x)= \e^{-\beta f(x)}$, and $\mathcal S(\mu)=1$.
              The conditions of~\cref{lemma:aux_aux_rt_cov} are satisfied for $\varsigma := s+1$ if~$\ell = 0$
              and~$\varsigma := 0$ if~$\ell > 0$.
        \item
              For the second term let $g(x) = \e^{- \gamma f(x)}$, $h(x)= \e^{-\beta f(x)}$ and $\mathcal S(\mu)=\wm(\mu)$.
              The conditions of~\cref{lemma:aux_aux_rt_cov} are satisfied for $\varsigma := s$ if~$\ell = 0$
              and~$\varsigma := 0$ if~$\ell > 0$.
    \end{itemize}
    This completes the proof.
\end{proof}

The key for our mean-field proofs is the following improvement of \cref{lemma:stab:wmean-wcov-basic,proposition:stability:sqrt-wcov}. Recall that
\begin{align*}
    p_{\mathcal M} := p_{\mathcal M}(s, \ell) := \begin{cases}
                                                     s+2,\, & \text{if $\ell = 0$}, \\
                                                     1,\,   & \text{if $\ell > 0$}.
                                                 \end{cases}
\end{align*}

\begin{corollary}
    [Improved stability estimates]
    \label{lemma:stab:improved}
    Suppose that $f \in \mathcal A(s, \ell, \ell)$ with $s, \ell \geq 0$.
    Then
    \begin{enumerate}[label=(\alph*), wide]
        \item \label{item:lem-stab-wmean-one-mom}
              For all $R>0$ and for all $p\ge p_{\mathcal M}$,
              there exists $\widetilde L_{\mathcal M} = \widetilde L_{\mathcal M}(R,p)>0$ such that
              \begin{align*}
                  \forall (\mu, \nu) \in \mathcal P_{p,R}(\real^d) \times \mathcal P_p(\real^d), \qquad
                  \myabs*{\wm(\mu) - \wm(\nu)}
                  \le \widetilde L_{\mathcal M} \, \wasserstein_p(\mu, \nu).
              \end{align*}

        \item \label{item:lem-stab-wcov-sqrt}
              For all $R>0$ and $p\ge 2 p_{\mathcal M}$,
              there exists $\widetilde L_{\mathcal C} = \widetilde L_{\mathcal C}(R,p)>0$ such that
              \begin{align*}
                  \forall (\mu, \nu) \in \mathcal P_{p,R}(\real^d) \times \mathcal P_{p}(\real^d),
                  \qquad
                  \Bigl\lVert \sqrt{\wc(\mu)} - \sqrt{\wc(\nu)} \Bigr\rVert_{\rm F}
                  \le \widetilde L_{\sqrt{\mathcal C}} \, \wasserstein_p(\mu, \nu).
              \end{align*}
    \end{enumerate}
\end{corollary}

\begin{remark}
    Observe that in contrast to \cref{lemma:stab:wmean-wcov-basic,proposition:stability:sqrt-wcov}, no upper bounds on the $p$-th moments of $\nu$ are needed in \cref{lemma:stab:improved}, it is only required that they are finite.
\end{remark}
\begin{proof}
    Let $\mathcal T := \wm$ or $\mathcal T := \sqrt{\wc}$,
    and denote $p_{\mathcal T} = p_{\mathcal M}$ if $\mathcal T = \wm$ or $p_{\mathcal T} = 2 p_{\mathcal M}$ if $\mathcal T = \sqrt{\mathcal C_{\beta}}$.
    Note that the following statements hold in either case.
    \begin{itemize}
        \item
              For all $R>0$ and $p \geq p_{\mathcal T}$,
              there exists $L_{\mathcal T}(R)>0$ such that
              \begin{equation}
                  \label{eq:improved_moment_bound_ass1}
                  \forall (\mu, \nu) \in \mathcal P_{p,R}(\real^d) \times \mathcal P_{p,R}(\real^d),
                  \qquad
                  \bigl\lVert \mathcal{T}(\mu) - \mathcal{T}(\nu) \bigr\rVert
                  \le L_{\mathcal T}(R) \, \wasserstein_p(\mu, \nu).
              \end{equation}
              Here, $\|\cdot\|$ denotes either $|\cdot|$ or $\|\cdot\|_{\rm F}$. Equation~\eqref{eq:improved_moment_bound_ass1} follows from~\cref{lemma:stab:wmean-wcov-basic} if $\mathcal T = \wm$
              and \cref{proposition:stability:sqrt-wcov} if $\mathcal T = \sqrt{\wc}$.
        \item
              For all $p \geq p_{\mathcal T}$,
              there is $K > 0$ such that
              \begin{equation}
                  \label{eq:improved_moment_bound_ass2}
                  \forall \mu \in \mathcal P_p(\R^d),
                  \qquad
                  \bigl\lVert \mathcal{T}(\mu) \bigr\rVert \leq K \wasserstein_p(\mu, \delta_0).
              \end{equation}
              When $\mathcal T = \wm$,
              this follows immediately from the inequality $\e^{-\beta f^*}\le \e^{-\beta f(\placeholder)} \le \e^{-\beta f_*}$ if~$\ell = 0$ and from~\cref{lemma:bound-on-weighted-moment} if~$\ell > 0$,
              where $f^*$ the supremum of $f$ over $\real^d$.
              When~$\mathcal T = \sqrt{\wc}$,
              the inequality~\eqref{eq:improved_moment_bound_ass2} follows in the same way after noticing that
              \begin{align}
                  \label{eq:bound_sqrt_cov}
                  \Bigl\lVert \sqrt{\wc(\mu)} \Bigr\rVert_{\rm F}^2
                  = \mytrace\Bigl(\wc\left(\mu\right) \Bigr)
                  = \frac{\int_{\R^d} \myabs{x}^2 \e^{-\beta f(x)} \, \mu(\d x)} {\int_{\R^d} \e^{-\beta f(x)} \, \mu(\d x)}
                  - \Bigl\lvert \wm(\mu) \Bigr\rvert^2
                  \leq \frac{\int_{\R^d} \myabs{x}^2 \e^{-\beta f(x)} \, \mu(\d x)} {\int_{\R^d} \e^{-\beta f(x)} \, \mu(\d x)}
              \end{align}
              for all $\mu \in \mathcal P_p(\R^d)$.
    \end{itemize}
    Now, assume for contradiction that there exists $R>0$ and $\{(\mu_i, \nu_i)\}_{i \in \N}$ in $\mathcal P_{p,R}(\real^d) \times \mathcal P_{p}$ such that
    \begin{align}
        \label{eq:lem:lip-wasserstein-contrary}
        \forall i \in \N, \qquad
        \frac{\Bigl\lVert \mathcal{T}(\mu_i) -\mathcal{T}(\nu_i)  \Bigr\rVert }
        {\wasserstein_p\left(\mu_i, \nu_i\right)} \ge i.
    \end{align}
    In particular $\nu_i[\myabs{x}^p] \to \infty$ in the limit as $i \to \infty$,
    otherwise the left-hand side of~\eqref{eq:lem:lip-wasserstein-contrary} would be bounded from above uniformly in~$i$ by \eqref{eq:improved_moment_bound_ass1}.
    Therefore, $\wasserstein_p(\nu_i, \delta_0) \to \infty$.
    On the other hand,
    by the triangle inequality,
    \begin{align*}
        \wasserstein_p\left(\mu_i, \nu_i\right)
        \ge \wasserstein_p\left(\nu_i, \delta_0 \right) -  \wasserstein_p\left(\mu_i, \delta_0 \right)
        \ge \wasserstein_p\left(\nu_i, \delta_0 \right) -  R.
    \end{align*}
    Therefore, by \eqref{eq:improved_moment_bound_ass2},
    we deduce that
    \begin{align*}
        \limsup_{i \to \infty}
        \frac
        {\Bigl\lVert \mathcal{T}(\mu_i) -\mathcal{T}(\nu_i)  \Bigr\rVert}
        {\wasserstein_p\left(\mu_i, \nu_i\right)}
        \le K  \limsup_{i \to \infty} \frac{\wasserstein_p(\nu_i, \delta_0)  +  R}
        {\wasserstein_p\left(\nu_i, \delta_0 \right) - R}
        = K,
    \end{align*}
    which contradicts~\eqref{eq:lem:lip-wasserstein-contrary}.
\end{proof}

\subsection{Moment estimates for the empirical measures}
\label{sub:moment}

\begin{lemma}
    \label{lemma:moment-bounds-empirical}
    Let $f \in \mathcal A(s, \ell, \ell)$ with $s, \ell \geq 0$,
    and suppose that $\mfldis_0\in \mathcal P_{p}(\R^d)$ for some $p \geq 2$.
    \begin{enumerate}[label=(\alph*)]
        \item \label{item:moment-particle-CBO} Consider the particle system~\eqref{eq:cbo-particles} with $\mfldis_0^{\otimes J}$-distributed initial data,
              and let~$\mu^J_t$ denote the corresponding empirical measures.
              Then there exists a constant $\kappa>0$ independent of $J$
              such that for all $J\in\N^+$
              \begin{align*}
                  \expect \left[ \sup_{t\in[0,T]} \left\lvert X_t^{(j)} \right\rvert^{p} \right]
                  \quad \vee \quad
                  \expect \left[ \sup_{t\in[0,T]} \mu^J_t \bigl[ |x|^{p} \bigr] \right]
                  \quad \vee \quad
                  \expect \left[ \sup_{t\in[0,T]}\left\lvert \wm(\mu^J_t) \right\rvert^{p} \right]
                  \leq \kappa.
              \end{align*}
        \item \label{item:moment-particle-CBS}   Consider the particle system~\eqref{eq:cbs-particles} with $\mfldis_0^{\otimes J}$-distributed initial data,
              and let~$\mu^J_t$ denote the corresponding empirical measures.
              Then there exists a constant $\kappa>0$ independent of $J$ such that for all $J\in\N^+$
              \begin{align*}
                  \expect \left[ \sup_{t\in[0,T]} \left\lvert X_t^{(j)} \right\rvert^{p} \right]
                  \vee
                  \expect \left[ \sup_{t\in[0,T]} \mu^J_t \bigl[ |x|^{p} \bigr] \right]
                  \vee
                  \expect \left[ \sup_{t\in[0,T]}\left\lvert \wm(\mu^J_t) \right\rvert^{p} \right]
                  \vee
                  \expect \left[ \sup_{t\in[0,T]} \left\lVert \sqrt{\wc(\mu^J_t)} \right\rVert^{p} \right]
                  \leq \kappa.
              \end{align*}
    \end{enumerate}
\end{lemma}
\begin{proof}
    The result can be proved using the same reasoning as in~\cite[Lemma 3.4]{carrillo2018analytical} or~\cite[Lemma~15]{fornasier2021consensusbased}.
    For convenience, we denote by~$C$ any constant depending only on $(f, \beta, \lambda, p, T)$.
    We first observe by~\eqref{eq:bound_sqrt_cov} and~\cref{lemma:bound-on-weighted-moment} that
    \begin{equation}
        \label{eq:use_lemma_bounded_moments}
        \forall \mu \in \mathcal P_p(\R^d), \qquad
        \Bigl\lvert \wm(\mu) \Bigr\rvert
        \vee
        \Bigl\lVert \sqrt{\wc(\mu)} \Bigr\rVert_{\rm F}
        \leq C\left(\int_{\R^d} \myabs{x}^p \, \d \mu\right)^{\frac{1}{p}}.
    \end{equation}
    Writing the drift and diffusion coefficients for \cref{item:moment-particle-CBO} and \cref{item:moment-particle-CBS}
    as in~\eqref{eq:cbo-drift-and-diff} and \eqref{eq:cbs-drift-and-diff},
    respectively, it holds in both cases that
    \begin{align*}
        \forall x\in\R^d, \qquad \forall \mu\in \mathcal{P}(\R^d), \qquad
        \bigl\lvert b(x,\mu) \bigr\rvert \vee \bigl\lVert \sigma(x,\mu) \bigr\rVert_{\rm F}
        \le C\left(\myabs{x}^p+ \int_{\R^d} |y|^p \, \mu(\d y)\right)^{\frac{1}{p}}.
    \end{align*}
    Arguing in parallel for \cref{item:moment-particle-CBO} and \cref{item:moment-particle-CBS}, we fix $j\in \range{1}{J}$ and apply the Burkholder--Davis--Gundy inequality to find for all $t \in [0, T]$ that
    \begin{align*}
        \frac{1}{3^{p-1}} \expect \left[ \sup_{s \in [0, t]} \left\lvert X_s^{(j)} \right\rvert^{p} \right]
         & \leq \expect \left\lvert X_0^{(j)} \right\rvert^{p}
        + T^{p-1} \int_{0}^t \expect \left\lvert  b\left( X_s^{(j)}, \mu^J_s\right) \right \rvert^p \, \d s                                                                             \\
         & \qquad \qquad + C_{\rm BDG}T^{\frac{p}{2}-1} \int_{0}^t \expect \left\lVert \sigma\left( X_s^{(j)}, \mu^J_s\right) \right\rVert_{\rm F}^p  \, \d s                           \\
         & \leq \expect \left\lvert X_0^{(j)} \right\rvert^{p} + C \int_{0}^t \expect \left[ \left\lvert X_s^{(j)} \right\rvert^p + \int_{\R^d} |y|^p \, \mu^J_s(\d y) \right] \, \d s.
    \end{align*}
    Since the common law of $\left(\xn{J}, \dots, \xn{J}\right)$ is invariant under permutations of $\llbracket 1, J\rrbracket$,
    its marginals corresponding to the particles are all the same.
    Therefore, given that $\mu^J_s= \frac{1}{J}\sum_{j=1}^J \delta_{\xn{j}_s}$,
    we find that
    \[
        \expect \left[ \int_{\R^d} |y|^p \, \mu^J_s(\d y) \right] = \expect \left\lvert \xn{j}_s \right\rvert^p.
    \]
    Consequently, it holds that
    \begin{align*}
        \label{eq:moment_particle_i}
        \expect \left[ \sup_{s \in [0, t]} \lvert X_s^{(j)} \rvert^{p} \right]
         & \leq C \expect \left\lvert X_0^{(j)} \right\rvert^{p} +
        C \int_{0}^t  \expect \left[ \sup_{u \in [0, s]}  \left\lvert X_u^{(j)} \right\rvert^{p} \right] \, \d s.
    \end{align*}
    Gr\"onwall's inequality then gives that
    \[
        \expect \left[ \sup_{t \in [0, T]} \bigl\lvert X_t^{(j)} \bigr\rvert^{p}  \right]  \leq \kappa.
    \]
    The second bound in \cref{item:moment-particle-CBO} and in \cref{item:moment-particle-CBS} follows from the inequality
    \[
        \expect \left[ \sup_{t \in [0, T]} \mu^J_t\bigl[ \left\lvert x \right\rvert^p \bigr] \right]
        = \expect \left[ \sup_{t \in [0, T]} \frac{1}{J} \sum_{j=1}^{J} \bigl\lvert \xn{j}_t \bigr\rvert^p \right]
        \leq \expect \left[ \frac{1}{J} \sum_{j=1}^{J} \sup_{t \in [0, T]} \bigl\lvert \xn{j}_t \bigr\rvert^p \right]
        = \expect \left[ \sup_{t \in [0, T]} \bigl\lvert \xn{j}_t \bigr\rvert^p \right],
    \]
    while the third bound in \cref{item:moment-particle-CBO} as well as the third and fourth bound in  \cref{item:moment-particle-CBS} follow from \eqref{eq:use_lemma_bounded_moments}.
\end{proof}

\subsection{Convergence of weighted moments for i.i.d.\ samples}
\label{sub:iid}

For the proof of the lemmas in this section,
we will need \cite[Theorem 1]{MR2597592}, which we recall for the reader's convenience in \cref{theorem:doukhan} below, slightly adapted to the i.i.d.\ setting.
We also refer to \cite[Theorem 2.3]{MR3696003}, where similar results for $p = 1$ and $p = 2$ are proved in the context of importance sampling,
with explicit constant prefactors.

\begin{theorem}
    [$L^p$ convergence of weighted sums]
    \label{theorem:doukhan}
    Let $(w_j, V_j)_{j \in \N}$ be i.i.d.\ random variables with values in~$\R \times \R^d$ or $\R \times \R^{d\times d}$ with $w_j > 0$ almost surely,
    and set
    \[
        \widehat N_J = \frac{1}{J} \sum_{j=1}^{J} w_j V_j, \qquad
        \widehat D_J = \frac{1}{J} \sum_{j=1}^{J} w_j, \qquad
        \widehat R_J = \frac{\widehat N_J}{\widehat D_J}.
    \]
    Let also $N = \expect \left[\widehat N_J\right]$, $D = \expect \left[\widehat D_J\right]$, and $R = N/D$.
    Let $0 < p < q$ and assume that for some $c>0$,
    \[
        \norm{w_1}_{L^q(\Omega)} \leq c, \qquad
        \norm{V_1}_{L^r(\Omega)} \leq c, \qquad
        \norm{w_1 V_1}_{L^s(\Omega)} \leq c,
        \qquad r := \frac{p(q+2)}{q-p},
        \qquad s := \frac{pq}{q-p}.
    \]
    Then there is~$C$ depending only on~$(c,p,q,r,s)$ such that
    \[
        \bigl\lVert  \widehat R_J - R  \bigr\rVert_{{L^p(\Omega)}}
        \leq C J^{-\frac{1}{2}}.
    \]
\end{theorem}

Using~\cref{theorem:doukhan},
we now prove the convergence of the weighted mean and weighted covariance for empirical measures formed from i.i.d.\ samples,
in the limit of infinitely many samples.

\begin{lemma}
    [Convergence of the weighted mean for i.i.d.\ samples]
    \label{lemma:convergence_weighted_mean_iid}
    Suppose that~$f\colon \real^d \to \real$ is bounded from below
    and let~$0 < p < r$.
    Then for all~$\mu \in \mathcal P_{r}(\R^d)$,
    there is $C$ depending only on $(f, \beta, p, r)$ and the $r$-th moment of~$\mu$ such that
    for all $J \in \N$
    \[
        \expect \left\lvert
        \wm\left(\overline \mu^J\right)
        - \wm\bigl(\mu\bigr)
        \right\rvert^p
        \leq C J^{- \frac{p}{2}},
        \qquad
        \overline \mu^J := \frac{1}{J} \sum_{j=1}^{J} \delta_{\xnl{j}},
        \qquad
        \left\{ \xnl{j} \right\}_{j \in \N} \stackrel{\rm{i.i.d.}}{\sim} \mu.
    \]
\end{lemma}

\begin{proof}
    We apply~\cref{theorem:doukhan} with $q = \frac{p(r+2)}{r-p}$,
    in which case
    \(
    \frac{p(q+2)}{q-p} = r.
    \)
    In our setting
    \[
        (w_j, V_j) = \left( \e^{- \beta f\left(\xnl{j}\right)}, \xnl{j} \e^{-\beta f\left(\xnl{j}\right)} \right).
    \]
    Since $w_j$ is bounded from above almost surely,
    it obviously holds that $\norm{w_1}_{L^q(\Omega)} < \infty$.
    Since $\mu \in \mathcal P_r(\real^d)$, it also holds that $\norm{V_1}_{L^r(\Omega)} < \infty$
    and $\norm{w_1 V_1}_{L^s(\Omega)} < \infty$ since $s := \frac{pq}{q-p} < r$.
    Consequently, the assumptions of~\cref{theorem:doukhan} are satisfied,
    and we obtain the statement.
\end{proof}

\begin{lemma}
    [Convergence of the weighted covariance for i.i.d. samples]
    \label{lemma:convergence_weighted_covariance_iid}
    Suppose that~$f\colon \real^d \to \real$ is bounded from below
    and let~$0 < 2 p < r$.
    Then for all~$\mu \in \mathcal P_{r}(\R^d)$ satisfying $\wc(\mu) \succcurlyeq \eta \I_d \succ 0$,
    there is $C$ depending only on~$(f, \beta, p, r, \eta)$ and the $r$-th moment of~$\mu$ such that
    for all $J \in \N^+$
    \[
        \expect \left\lVert \sqrt{\wc(\overline \mu^J)} - \sqrt{\wc(\mu)}
        \right\rVert_{\rm F}^p
        \leq C J^{-\frac{p}{2}},
        \qquad
        \overline \mu^J := \frac{1}{J} \sum_{j=1}^{J} \delta_{\xn{j}},
        \qquad
        \left\{ \xnl{j} \right\}_{j \in \N} \stackrel{\rm{i.i.d.}}{\sim} \mu.
    \]
\end{lemma}

\begin{proof}
    By assumption, it holds that $\sqrt{\wc(\overline \mu^J)} \succcurlyeq \sqrt{\eta} \I_d$.
    By an inequality due to van Hemmen and Ando~\cite{vanHemmen1980},
    see also~\cite[Problem X.5.5]{MR1477662},
    it holds that
    \[
        \left\lVert \sqrt{\wc(\overline \mu^J)} - \sqrt{\wc(\mu)} \right\rVert_{\rm F}
        \leq \frac{1}{\eta} \Bigl\lVert \wc(\overline \mu^J) - \wc(\mu) \Bigr\rVert_{\rm F}.
    \]
    Using the triangle inequality, we have that
    \begin{align*}
        \Bigl\lVert \wc(\overline \mu^J) - \wc(\mu) \Bigr\rVert_{\rm F}^p
         & = \Bigl\lVert L_{\beta} \overline \mu^J \left[x \otimes x\right] - L_{\beta} \mu \left[x \otimes x\right] - \wm(\overline \mu^J) \otimes \wm(\overline \mu^J) + \wm(\mu) \otimes \wm(\mu) \Bigr\rVert_{\rm F}^p \\
         & \leq 2^{p-1}\Bigl\lVert L_{\beta} \overline \mu^J \left[x \otimes x\right] - L_{\beta} \mu \left[x \otimes x\right] \Bigr\rVert_{\rm F}^p                                                                       \\
         & \qquad \qquad + 2^{p-1}\Bigl\lVert \wm(\overline \mu^J) \otimes \wm(\overline \mu^J) - \wm(\mu) \otimes \wm(\mu) \Bigr\rVert_{\rm F}^p.
    \end{align*}
    Convergence to zero of the expectation of the first term with rate $J^{-\frac{p}{2}}$ follows from~\cref{theorem:doukhan},
    this time with~$q = \frac{p(r+4)}{r-2p}$,
    in which case
    \(
    \frac{p(q+2)}{q-p} = \frac{r}{2}.
    \)
    For the second term,
    we use~\eqref{eq:aux_weighted_mean}
    to obtain that
    \begin{align*}
         & \expect \Bigl\lVert \wm(\overline \mu^J) \otimes \wm(\overline \mu^J) - \wm(\mu) \otimes \wm(\mu) \Bigr\rVert_{\rm F}^p                                                                               \\
         & \qquad \leq \expect \left[ \bigl(\myabs*{\wm(\overline \mu^J)} + \myabs{\wm(\mu)}\bigr)^p \myabs{\wm(\overline \mu^J) - \wm(\mu)}^p \right]                                                           \\
         & \qquad \leq \left( \expect \left[ \bigl(\myabs*{\wm(\overline \mu^J)} + \myabs{\wm(\mu)}\bigr)^{2p} \right] \expect \left[ \myabs{\wm(\overline \mu^J) - \wm(\mu)}^{2p} \right] \right)^{\frac{1}{2}} \\
         & \qquad \leq C J^{-\frac{p}{2}} \left( \expect \left[ \bigl(\myabs*{\wm(\overline \mu^J)} + \myabs{\wm(\mu)}\bigr)^{2p} \right] \right)^{\frac{1}{2}},
    \end{align*}
    where we used~\cref{lemma:convergence_weighted_mean_iid}.
    Since $\wm(\overline \mu^J)$ converges to $\wm (\mu)$ in $L^{2p}(\Omega)$ again by~\cref{lemma:convergence_weighted_mean_iid},
    this inequality enables to conclude the proof.
\end{proof}

\subsection{No collapse in finite time for CBS}
\label{sub:no_collapse}

In order to deal with the matrix square root in the CBS dynamics, we prove that the weighted covariance stays positive definite for finite times. This is particularly relevant when CBS is run in optimization mode with~$\lambda=1$, since then $\mathcal{C}(\mfldis_t) \to 0$ as $t\to \infty$ under appropriate assumptions \cite{CBS-Carrillo2021}.

\begin{proposition}[No collapse in finite time]
    \label{prop:cbs-minimal-eigenvalue-wcov}
    Let the cost function $f\colon \R^d \rightarrow \R$ be such that the macroscopic McKean SDE for CBS \eqref{eq:cbs-mf}
    has a solution $(\xl_t)_{t\in [0,T]}$ over some finite time interval $[0,T]$.
    Then
    \begin{align}
        \label{eq:cbs-cov-min-eigen-expo-decay}
        \forall t \in [0, T], \qquad
        \mathcal{C}(\mfldis_t)
        \succcurlyeq  \mathcal{C}(\mfldis_0) \e^{-2t}.
    \end{align}
    In particular, if $\mathcal{C}(\mfldis_0) \succ 0$ and if $t \mapsto \wc(\rho_t)$ is continuous (which is true under the assumptions of \cref{thm:cbs-mean-field-sde-well-posed}), then
    \begin{align}
        \label{eq:cbs-minimal-eigenvalue-cov}
        \inf_{t \in[0,T]}\lambda_{\rm min}\bigl(\mathcal{C}(\mfldis_t) \bigr) >0
         &  & \text{and} &  &
        \inf_{t \in[0,T]}\lambda_{\rm min}\bigl(\wc(\mfldis_t) \bigr) >0.
    \end{align}
\end{proposition}

\begin{proof}
    The evolution of the covariance matrix is governed by \cite[Equation (2.18)]{CBS-Carrillo2021}
    \begin{align*}
        \partial_t \left( \mathcal{C}(\mfldis_t)\right)
        = -2 \mathcal{C}(\mfldis_t)
        + 2\lambda^{-1} \wc(\mfldis_t).
    \end{align*}
    For a given $y\in\R^d$ such that $|y|=1$, define  $g(t):= y^\t \mathcal{C}(\mfldis_t) y$ for $t\in [0,T]$. Then
    \begin{align*}
        g'(t) = -2g(t) + 2\lambda^{-1} y^T \wc (\mfldis_t) y \geq -2 g(t),
    \end{align*}
    since $\wc (\mfldis_t)$ is positive semi-definite.
    Applying Gr\"onwall's inequality to $-g$,
    we deduce that
    \(
    g(t) \geq g(0) \e^{-2 t}.
    \)
    This shows \eqref{eq:cbs-cov-min-eigen-expo-decay} and the first inequality in~\eqref{eq:cbs-minimal-eigenvalue-cov}.
    In order to prove the second inequality in \eqref{eq:cbs-minimal-eigenvalue-cov}, assume that it was wrong.
    By continuity of $t \mapsto \wc(\rho_t)$, we have~$\lambda_{\rm min}\bigl(\wc(\mfldis_t) \bigr) =0$ for some $t\in[0,T]$.
    Then~$y^\t \wc (\mfldis_t) y = 0$ for some $y\in\R^d$  such that $|y|=1$.
    In particular
    \[
        \int_{\R^d} \left\lvert y^\t \Bigl( x - \wm(\mfldis_t) \Bigr) \right\rvert ^2 \e^{-\beta f(x)} \, \mfldis_t(\d x) = 0.
    \]
    Thus, we obtain $y^\t \bigl(x - \wm(\mfldis_t)\bigr) = 0$ for $\mfldis_t$-almost every $x\in\R^d$,
    and so $y^\t \wm(\mfldis_t) =  y^\t \mathcal M(\mfldis_t)$
    by integrating against $\mfldis_t$.
    But then,
    \[
        y^\t  \mathcal{C} (\mfldis_t) y
        = \int_{\R^d} \left\lvert y^\t \Bigl(x - \mathcal M(\mfldis_t)\Bigr) \right\rvert^2 \, \mfldis_t(\d x)
        = \int_{\R^d} \left\lvert y^\t \Bigl(x - \wm(\mfldis_t)\Bigr) \right\rvert^2 \, \mfldis_t(\d x)
        = 0,
    \]
    which is a contradiction.
    Thus, $C_{\beta}(\mfldis_t) \succ 0$ for all $t \in [0, T]$, which implies the second inequality in~\eqref{eq:cbs-minimal-eigenvalue-cov}.
\end{proof}

\section{Proof of the well-posedness results}
\label{sec:cbo-cbs-wellp}

\subsection{Well-posedness for the microscopic models}

For a collection of particles $\mathbf{X} \in \R^{d J}$, the notation $\lvert \mathbf{X} \rvert$ denotes the usual Euclidean norm;
that is
\[
    \lvert \mathbf{X} \rvert \coloneq \sqrt{\sum_{j=1}^{J} \left\lvert \xn{j} \right\rvert^2}.
\]
For the proof of \cref{thm:cbo-cbs-particles-well-posed} we need the following stability result, which only requires local Lipschitz continuity of $f$,
in contrast to the results in \cref{sub:stability_estimates}.

\begin{lemma}
    \label{lemma:cbs-microscopic-loc-lip}
    For all $\mathbf{X}\in \R^{d J}$ it holds that
    \begin{equation}
        \label{eq:wmean-wcov-emp-bound}
        \myabs{ \wm (\mathbf{X})} \le  \myabs{\mathbf{X}}
        \qquad \text{and} \qquad
        \norm{ \wc (\mathbf{X})} \le  \myabs{\mathbf{X}}^2.
    \end{equation}
    Moreover,  if $f\colon \R^d \rightarrow \R$ is locally Lipschitz continuous,
    then for each $R>0$ and $J\in \N^+$ there exists a constant~$C_{J,R}>0$ such that
    \begin{align}
        \label{eq:wmean-wcov-local-lip-wellposed}
        \myabs{
            \wm(\mathbf{X})- \wm ( \mathbf{Y}) }
        \le C_{J,R} \myabs{ \mathbf{X} -\mathbf{Y}}
         &  & \text{and} &  &
        \norm{
            \sqrt{\wc(\mathbf{X})} -
            \sqrt{ \wc ( \mathbf{Y}) }}_F
        \le C_{J,R} \myabs{ \mathbf{X} -\mathbf{Y}}
    \end{align}
    for all $\mathbf{X},\mathbf{Y}\in \R^{dJ}$ satisfying $\myabs{\textbf{X}}, |\mathbf{Y}|\le R$.
\end{lemma}

\begin{proof}
    The bound on $\wm(\mathbf{X})$ in \eqref{eq:wmean-wcov-emp-bound} follows immediately from the definition:
    We have that
    \[
        \Bigl\lvert \wm(\mathbf{X}) \Bigr\rvert^2
        = \left\lvert \sum_{j=1}^{J} w_j(\mathbf{X}) \xn{j} \right\rvert^2
        \leq  \sum_{j=1}^{J} w_j(\mathbf{X}) \left\lvert \xn{j} \right\rvert^2
        \leq  \sum_{j=1}^{J} \left\lvert \xn{j} \right\rvert ^2
        = \left\lvert \mathbf{X} \right\rvert^2,
    \]
    where we used convexity of the square function and the
    fact that the weights, given hereafter,
    are bounded from above by 1:
    \[
        w_j(\mathbf{X}) \coloneq \frac{\e^{-\beta f\left(\xn{j}\right)}}{\sum_{k=1}^{J} \e^{-\beta f\left(\xn{k}\right)}}.
    \]
    Similarly, for the weighted covariance,
    it holds by the triangle inequality that
    \begin{align*}
        \left\lVert \wc(\mathbf{X}) \right\rVert_{\rm F}
         & = \left\lVert \sum_{j=1}^{J} w_j(\mathbf{X}) \Bigl(  \xn{j} \otimes \xn{j} \Bigr) -  \wm(\mathbf{X}) \otimes \wm(\mathbf{X}) \right\rVert_{\rm F} \\
         & \leq \left\lVert \sum_{j=1}^{J} w_j(\mathbf{X}) \Bigl(  \xn{j} \otimes \xn{j} \Bigr) \right\rVert_{\rm F}
        \leq \sum_{j=1}^{J} w_j(\mathbf{X})  \left\lVert \xn{j} \otimes \xn{j} \right\rVert_{\rm F}
        = \sum_{j=1}^{J} w_j(\mathbf{X})  \left\lvert \xn{j} \right\rvert^2 \leq \left\lvert \mathbf{X} \right\rvert^2.
    \end{align*}
    The local Lipschitz estimate on $\wm(\mathbf{X})$ in~\eqref{eq:wmean-wcov-local-lip-wellposed} is \cite[Lemma 2.1]{carrillo2018analytical}.
    In order to prove the local Lipschitz continuity of $\wc(\mathbf{X})$, fix $R>0$ and $J>0$.
    As in the proof of \cref{proposition:stability:sqrt-wcov} we obtain
    \begin{align*}
        \norm{
            \sqrt{\wc(\mathbf{X})} -
            \sqrt{\wc (\mathbf{Y}) }}_{\rm F}
         & \le \sqrt{2} \left( \sum_{j=1}^J \myabs*{
            \sqrt{w_j(\mathbf{X})} \left(\xn{j} - \wm(\mathbf{X}) \right)
        - \sqrt{w_j(\mathbf{Y})} \left(\yn{j} - \wm(\mathbf{Y}) \right)}^2 \right)^{\frac{1}{2}} \\
         & \le \sqrt{2}  \sum_{j=1}^J \myabs*{
            \sqrt{w_j(\mathbf{X})} \left(\xn{j} - \wm(\mathbf{X}) \right)
            - \sqrt{w_j(\mathbf{Y})} \left(\yn{j} - \wm(\mathbf{Y}) \right)}.
    \end{align*}
    Using the triangle inequality, we have that
    \begin{align*}
        \frac{1}{\sqrt{2}}
        \norm{
            \sqrt{\wc(\mathbf{X})} -
            \sqrt{\wc (\mathbf{X}) }}_{\rm F}
         & \le  \sum_{j=1}^J
        \sqrt{w_j(\mathbf{X})}  \myabs* {\xn{j} - \wm(\mathbf{X}) - \yn{j} + \wm(\mathbf{Y})}                                                                      \\
         & \qquad + \sum_{j=1}^J  \left\lvert \left( \sqrt{w_j(\mathbf{X})} -  \sqrt{w_j(\mathbf{Y})} \right) \left(\yn{j} - \wm(\mathbf{Y}) \right) \right\rvert.
    \end{align*}
    Since $\mathbf{X}\mapsto \wm(\mathbf{X})$ is a locally Lipschitz continuous function by \cite[Lemma 2.1]{carrillo2018analytical},
    it suffices to prove that  $\mathbf{X} \mapsto \sqrt{w_j(\mathbf{X})}$ is locally Lipschitz as well.
    To this end, observe that
    \begin{align*}
        \sqrt{w_j(\mathbf{X})}
        - \sqrt{w_j(\mathbf{Y})}
        =
        \frac{w_j(\mathbf{X})
            - w_j(\mathbf{Y})}
        {
            \sqrt{w_j(\mathbf{X})}
            + \sqrt{w_j(\mathbf{Y})}
        }.
    \end{align*}
    The denominator can be controlled by noting that
    \[
        \forall \lvert \mathbf{X} \rvert \leq R, \qquad
        w_j(\mathbf{X})
        = \frac{\e^{-\beta f\left(\xn{j} \right)}}{\sum_{k=1}^{J} \e^{-\beta f\left(\xn{k} \right) }}
        \geq \frac{1}{J} \exp \left( - \beta \left( \max_{\myabs{x} \leq R} f(x) - \min_{\myabs{x} \leq R} f(x) \right) \right).
    \]
    Therefore, there exists a constant $c_{R,J}>0$ such that $\sqrt{w_j(\mathbf{X})} \ge c_{R,J}$ for all $\myabs{\mathbf{X}} \le R$.
    Since $\mathbf{X} \mapsto  {w_j(\mathbf{X})}$ is locally Lipschitz continuous as a composition of locally Lipschitz continuous functions,
    the claim follows.
\end{proof}

\begin{proof}[Proof of \cref{thm:cbo-cbs-particles-well-posed}]
    We prove only the claim for the CBS dynamics~\eqref{eq:cbs-particles},
    since the proof for CBO is similar.
    We write the system \eqref{eq:cbs-particles} in the form
    \[
        \d \mathbf{X}_t = F(\mathbf{X}_t) \, \d t + G(\mathbf{X}_t) \, \d \mathbf{ \wien}_t
    \]
    where $\mathbf{X}_t \coloneq \Bigl(\xn{1}_t, \dots, \xn{J}_t\Bigr)$.
    The drift and diffusion coefficients are given by
    \begin{align*}
        F(\mathbf{X}_t) \coloneq -
        \begin{pmatrix}
            \xn{1}_t - \wm \left(\emp{J}_t\right) \\
            \vdots                                \\
            \xn{J}_t - \wm \left(\emp{J}_t\right)
        \end{pmatrix} \in \R^{d J},
    \end{align*}
    and
    \begin{align*}
        G(\mathbf{X}_t) \coloneq &
        \sqrt{2\lambda^{-1}}
        \operatorname{diag}\left(
        \sqrt{\wc\left(\emp{J}_t\right)}, \dots, \sqrt{\wc\left(\emp{J}_t\right)}\right)\in \R^{d J \times d J}.
    \end{align*}
    By \cref{lemma:cbs-microscopic-loc-lip},
    the functions $F$ and $G$ have sublinear growth and are locally Lipschitz continuous.
    Thus, by the non-explosion criterion from stochastic Lyapunov theory  \cite[Theorem 3.5]{khasminskii2013stochastic},
    it suffices to find a function~$\phi \in C^2\left(\R^{d J}, [0,\infty)\right)$ such that
    \begin{itemize}
        \item The function $\phi$ is coercive: $\phi(\mathbf{X}) \to \infty$ as $|\mathbf{X}|\to\infty$ and
        \item There exists $c>0$ such that
              \[
                  \forall \mathbf{X}\in\R^{d J}, \qquad
                  \mathcal{L}\phi(\mathbf{X}) \coloneq F(\mathbf{X}) \cdot \nabla \phi(\mathbf{X}) + \frac{1}{2} \mytrace\left(G^\t(\mathbf{X}) \hessian \phi(\mathbf{X}) G(X)\right)  \le c \phi(\mathbf{X}).
              \]
    \end{itemize}
    Take $\phi(\mathbf{X})\coloneq\tfrac{1}{2} |\mathbf{X}|^2$.
    Then
    \begin{align*}
        \mathcal{L}\phi(\mathbf{X})
         & = - \sum_{j=1}^{J} \left(\xn{j} - \wm \left(\emp{J}_t\right)\right) \cdot \xn{j} + J  \lambda^{-1}\mytrace \Bigl( \wc \left(\emp{J}_t\right) \Bigr)                                            \\
         & = - \sum_{j=1}^{J} \left(\xn{j} - \wm \left(\emp{J}_t\right)\right) \cdot \xn{j} + \lambda^{-1} \sum_{j=1}^{J} w_j(\mathbf{X}) \left\lvert \xn{j} - \wm \left(\emp{J}_t\right) \right\rvert^2.
    \end{align*}
    From \eqref{eq:wmean-wcov-emp-bound} it follows that this is a valid choice,
    concluding the proof.
\end{proof}

\subsection{Well-posedness for the mean-field SDEs}

We adapt the argument from \cite{carrillo2018analytical},
using the more general \cref{assump:main} and the improved stability estimates from \cref{sub:stability_estimates}.
We first prove \cref{thm:cbs-mean-field-sde-well-posed},
and then comment on the necessary changes for proving \cref{thm:cbo-mean-field-sde-well-posed} in \cref{rmk:cbo-wellp-proof}. In the following, we endow the vector space $\mathcal X \coloneq \R^d \times \R^{d \times d}$ with the norm
\[
    \norm{(m,\Gamma)} \coloneq \left({|m|^2 + \norm{\Gamma}^2_{\rm F}}\right)^{\frac{1}{2}}.
\]
The Hölder space $ C^{0,\frac{1}{2}}\bigl([0, T], \mathcal X \bigr)$ is given by
\[
    C^{0,\frac{1}{2}}\bigl([0, T], \mathcal X \bigr) \coloneq
    \left\{ \varphi\colon [0, T]\rightarrow \mathcal X  : \norm{\varphi}_{C^{0,\frac{1}{2}}} < \infty \right\},
\]
where
\begin{align*}
    \norm{\varphi}_{C^{0,\frac{1}{2}}}\coloneq \sup_{t\in [0,T]} \bigl\lVert  \varphi(t) \bigr\rVert  + [\varphi]_{C^{0,\frac{1}{2}}} < \infty
    \qquad \text{and} \qquad [\varphi]_{C^{0,\frac{1}{2}}} \coloneq \sup_{0 \leq s < t \leq T} \frac{ \bigl\lVert \varphi(t)- \varphi(s) \bigr\rVert}{|t-s|^{\frac{1}{2}}}.
\end{align*}
In order to ensure that certain matrices are positive semidefinite,
we use the map $\mathcal{R}\colon \R^{d\times d} \rightarrow \R^{d\times d}$ defined by $\mathcal R(\Gamma) = \sqrt{\Gamma \Gamma^\t}$.
We will make use of the following auxiliary lemma.

\begin{lemma}\label{lem:cont-moments}
    For $f$ and $\mfldis_0$ as in \cref{thm:cbs-mean-field-sde-well-posed}, fix $(m,\Gamma) \in C([0,T], \mathcal X)$ and $y_0 \sim \mfldis_0$. Then the stochastic differential equation
    \begin{align}
        \label{eq:cbs-existence-SDE-for-u-and-G}
        \d Y_t = -(Y_t - m_t) \, \d t + \sqrt{2\lambda^{-1} \mathcal R (\Gamma_t)} \, \d W_t,
        \qquad Y_0 = y_0\,,
    \end{align}
    admits a unique continuous strong solution~$Y\in C([0,T], \R^d)$.
    Moreover, denoting by $\rho_t\coloneq\Law(Y_t) \in \mathcal{P}(\R^d)$ the pointwise law of~$Y_t$, the functions $t \mapsto \wm(\rho_t)$ and $t \mapsto \wc(\rho_t)$ belong to $C([0, T], \R^d)$ and $C([0, T], \R^{d\times d})$ respectively.
\end{lemma}
\begin{proof}
    The existence of a unique continuous strong solution~$Y\in C([0,T], \R^d)$ to~\eqref{eq:cbs-existence-SDE-for-u-and-G} follows from classical theory of stochastic differential equations~\cite[Theorem 5.2.1]{MR2001996}. It remains to prove the continuity statement for the weighted mean and covariance.
    For any~$0 \leq r \leq t \leq T$,
    it holds by the Burkholder--Davis--Gundy inequality that
    \begin{align}
        \notag
        \expect  \left[ \sup_{u \in [r, t]} \left\lvert Y_u - Y_r \right\rvert^{p} \right]
         & \leq \bigl(2(t-r)\bigr)^{p-1} \int_{r}^{t} \expect \bigl\lvert Y_u - m_u \bigr\rvert^{p} \, \d u                                                                                                \\
        \label{eq:bdg_well_posedness}
         & \qquad + C_{\rm BDG} \, 2^{p-1} \left(2\lambda^{-1}\right)^{\frac{p}{2}} (t-r)^{\frac{p}{2}-1} \int_{r}^{t}  \left\lVert \sqrt{\mathcal R\bigl(\Gamma(u)\bigr)} \right\rVert_{\rm F}^p \, \d u.
    \end{align}
    Taking~$r = 0$ in this inequality and rearranging,
    we obtain that
    \[
        \expect \left[ \sup_{u \in [0, t]} \left\lvert Y_u \right\rvert^{p} \right]
        \leq C \left( \expect \left\lvert Y_0 \right\rvert^{p}
        + \lVert m \rVert_{L^{\infty}([0, T])}^{p}
        + \lVert \Gamma \rVert_{L^{\infty}([0, T])}^{\frac{p}{2}}
        + \int_{0}^{t} \expect \myabs{Y_u}^p \, \d u\right),
    \]
    for a constant $C$ depending only on $(T, \lambda, p)$.
    It follows from this inequality, by Gr\"onwall's lemma,
    that
    \begin{equation}
        \label{eq:gronwall_well_posedness}
        \expect \left[ \sup_{t \in [0, T]} |Y_t|^{p} \right]
        \leq C \left( \expect \left\lvert Y_0 \right\rvert^{p}  + \lVert m \rVert_{L^{\infty}([0, T])}^{p} + \lVert \Gamma \rVert_{L^{\infty}([0, T])}^{\frac{p}{2}} \right) \e^{C T}
        < \infty.
    \end{equation}
    Therefore, by dominated convergence,
    we have for any $t \in [0, T]$ that
    \[
        \wm(\rho_r)
        = \frac{\expect \left[Y_r \e^{-\beta f(Y_r)} \right]}{\expect \left[\e^{-\beta f(Y_r)}\right]}
        \xrightarrow[r \to t]{}
        \frac{\expect \left[Y_t \e^{-\beta f(Y_t)} \right]}{\expect \left[\e^{-\beta f(Y_t)}\right]}
        = \wm(\rho_t),
    \]
    implying that the function $t \mapsto \wm(\rho_t)$ belongs to $C([0, T], \R^d)$.
    An analogous reasoning shows that the function $t \mapsto \wc(\rho_t)$ belongs to $C([0, T], \R^{d\times d})$.
\end{proof}

\begin{proof}[Proof of \cref{thm:cbs-mean-field-sde-well-posed}]
    Similarly to the proof of \cite[Theorem 3.1]{carrillo2018analytical},
    we first construct a map
    \[
        \mathcal{T}\colon C\left([0,T], \mathcal X \right) \rightarrow  C\left([0,T], \mathcal X \right)
    \]
    whose fixed points correspond to solutions of \eqref{eq:cbs-mf}.

    \paragraph{Step 1: Constructing the map~$\mathcal T$}
    Consider the map
    \begin{align*}
        \mathcal{T}\colon C([0,T], \mathcal X) & \rightarrow  C([0,T], \mathcal X)             \\
        (m,\Gamma)                             & \mapsto \bigl(\wm(\rho_t), \wc(\rho_t)\bigr),
    \end{align*}
    where $\rho_t \coloneq {\rm Law}(Y_t)$ and $(Y_t)_{t \in [0, T]}$ is the unique solution to~\eqref{eq:cbs-existence-SDE-for-u-and-G}.
    Thanks to Lemma~\ref{lem:cont-moments}, the map $\mathcal{T}$ is well-defined.
    Fixed points of $\mathcal{T}$ correspond to solutions of the McKean-Vlasov SDE~\eqref{eq:cbs-mf}.
    The existence of fixed points follows from applying the Leray-Schauder Fixed Point Theorem \protect{\cite[Chapter 11]{gilbarg1977elliptic}} in the space~$C\left([0,T], \mathcal X \right)$,
    once we have proved that $\mathcal T$ is compact and that the following set is bounded:
    \begin{align}
        \label{eq:leray-schauder-set}
        \Bigl\{(m,\Gamma) \in C\bigl([0,T], \mathcal X\bigr) :  \exists \xi\in[0,1]\text{ such that } (m,\Gamma)=\xi \mathcal{T}\bigl((m,\Gamma)\bigr)\Bigr\}.
    \end{align}

    \paragraph{Step 2: Showing that~$\mathcal T$ is compact}
    In order to prove that $\mathcal{T}$ is a compact operator,
    fix $R>0$ and consider the ball
    \[
        B_R \coloneq \left\{ (m,\Gamma)\in  C([0,T], \mathcal X) : \norm{(m,\Gamma)}_{L^\infty([0,T])} \le  R \right\}.
    \]
    By the Arzelà–Ascoli theorem, we have a compact embedding
    \begin{align*}
        C^{0,\frac{1}{2}}\left([0,T], \mathcal X\right) \hookrightarrow C\left([0,T], \mathcal X\right).
    \end{align*}
    Thus, it suffices to show that $\mathcal{T}(B_R)$ is bounded in $C^{0,\frac{1}{2}}\left([0,T], \mathcal X\right)$.
    To this end, take $(m,\Gamma) \in B_R$
    and consider the corresponding solution $Y_t$ of \eqref{eq:cbs-existence-SDE-for-u-and-G} with pointwise law~$\rho_t$,
    and let~$\varphi = \mathcal T(m, \Gamma)$.
    We first observe that
    \[
        \bigl\lVert  \varphi(0)  \bigr\rVert^2
        =
        \bigl\lvert \wm(\mfldis_0) \bigr\rvert^2 + \bigl\lVert \wc(\mfldis_0) \bigr\rVert_{\rm F}^2
        < \infty.
    \]
    The right-hand side is independent of $m$ and $\Gamma$,
    and it is indeed finite by~\cref{lemma:bound-on-weighted-moment}.
    Given that
    \[
        \norm{\varphi}_{C^{0,\frac{1}{2}}}
        \leq  \left\lVert \varphi(0) \right\rVert  + [\varphi]_{C^{0,\frac{1}{2}}} \bigl(\sqrt{T} + 1 \bigr),
    \]
    it only remains to show that~$[\varphi]_{C^{0,\frac{1}{2}}}$ is bounded by a finite constant depending on $(m,\Gamma)$ only through~$R$.
    In view of \eqref{eq:bdg_well_posedness} and \eqref{eq:gronwall_well_posedness},
    there is a constant $L = L(R)$ such that
    \[
        \expect  \left\lvert Y_t - Y_r \right\rvert^{p}  \leq L \myabs{t - r}^{\frac{p}{2}}.
    \]
    It follows that $\wasserstein_p(\rho_t, \rho_r) \leq L^{\frac{1}{p}} \myabs{t - r}^{\frac{1}{2}}$,
    which by~\cref{lemma:stab:wmean-wcov-basic} implies
    the H\"older continuity with exponent~$\frac{1}{2}$ of the functions~$t \mapsto \wm(\rho_t)$ and $t \mapsto \wc(\rho_t)$.
    This completes the proof that $\mathcal{T}$ is compact.

    \paragraph{Step 3: Showing that the set~\eqref{eq:leray-schauder-set} is bounded}
    To this end, let $(m, \Gamma)\in  C([0,T], \mathcal X)$ be such that
    \begin{equation}
        \label{eq:assumption_boundedness}
        (m,\Gamma)= \xi \mathcal{T}(m,\Gamma)
    \end{equation}
    for some $\xi\in[0,1]$,
    and let $(Y_t)$ denote the corresponding solution to~\eqref{eq:cbs-existence-SDE-for-u-and-G}.
    By \eqref{eq:assumption_boundedness},
    the stochastic process~$(Y_t)$ is also a solution to
    \[
        \d Y_t = -\Bigl(Y_t - \xi \wm(\rho_t) \Bigr) \, \d t + \sqrt{2\lambda^{-1} \xi \wc (\rho_t)} \, \d W_t,
        \qquad \rho_t \coloneq {\rm Law}(Y_t).
    \]
    But then, reasoning as in~\cref{sub:moment} using a Gr\"onwall-type argument and \cref{lemma:bound-on-weighted-moment} with $u=\ell$,
    we find that
    \[
        \expect \left[ \sup_{t \in [0, T]} \bigl\lvert Y_t \bigr\rvert^{p}  \right]
        \leq C \expect \lvert Y_0 \rvert^p = C \int_{\R^d} \myabs{x}^p \, \mfldis_0(\d x).
    \]
    By \cref{lemma:bound-on-weighted-moment},
    it follows that $\wm(\rho_t)$ and $\wc(\rho_t)$
    can be bounded uniformly in $[0, T]$ in terms of~$\mfldis_0$,
    implying that the set~\eqref{eq:leray-schauder-set} is indeed bounded.
    Thus, we established the existence of a fixed point of~$\mathcal T$,
    and we obtain as a byproduct the moment bounds in~\eqref{eq:cbs-moment-bound-mfl}.

    \paragraph{Step 4: Showing uniqueness}
    For uniqueness,
    let $(m_t, \Gamma_t)$ and $(\widehat{m}_t, \widehat{\Gamma}_t)$ be two fixed points of $\mathcal{T}$ with corresponding solutions~$Y_t, \widehat{Y}_t \in C([0,T], \mathcal X)$ of \eqref{eq:cbs-existence-SDE-for-u-and-G} and let $z_t\coloneq Y_t - \widehat{Y}_t$.
    By definition of $\mathcal{T}$ and since~$\wc(\mu)$ is symmetric and positive semidefinite for all~$\mu\in\mathcal{P}_2(\R^d)$,
    we have $\mathcal{R}\left(\Gamma_t\right) = \Gamma_t$ and $\mathcal{R}\bigl(\widehat{\Gamma}_t\bigr) = \widehat{\Gamma}_t$ for all~$t\in[0,T]$.
    Let~$\rho_t, \widehat \rho_t \in \mathcal P(\R^d)$ denote the marginal laws of $Y_t$ and $\widehat Y_t$,
    respectively.
    By the Burkholder--Davis--Gundy inequality,
    we have for all $t \in [0, T]$
    \begin{align*}
        \frac{1}{2^{p-1}}\expect \left[
        \sup_{r \in [0, t]} |z_r|^{p} \right]
         & \leq T^{p-1} \int_{0}^{t} \expect  \left\lvert Y_r - \widehat Y_r - \wm(\rho_r) + \wm(\widehat \rho_r) \right\rvert^p \, \d r                                                \\
         & \qquad + C_{\rm BDG} T^{\frac{p}{2}-1} (2 \lambda^{-1})^{\frac{p}{2}}\int_{0}^t \left\lVert \sqrt{\wc(\rho_r)} - \sqrt{\wc(\widehat \rho_r)} \right\rVert_{\rm F}^p \, \d r.
    \end{align*}
    By \cref{prop:cbs-minimal-eigenvalue-wcov},
    we have that
    \[
        \eta\coloneq\inf\left\{\lambda_{\rm min} \Bigl(\wc(\rho_t)\Bigr): t\in [0,T]\right\} >0.
    \]
    By virtue of the van Hemmen--Ando inequality \cite[Proposition 3.2]{vanHemmen1980},
    it holds for any unitarily invariant norm,
    thus in particular for the Frobenius norm, that
    \[
        \left\lVert \sqrt{\wc(\rho_r)} - \sqrt{\wc(\widehat \rho_r)}  \right\rVert_{\rm F}
        \leq \eta^{-\frac{1}{2}} \left\lVert \wc(\rho_r) - \wc(\widehat \rho_r)  \right\rVert_{\rm F}\,.
    \]
    Using \cref{lemma:stab:wmean-wcov-basic} and noting that $\wasserstein_p(\rho_r, \widehat \rho_r)^p \leq \expect \lvert z_r \rvert^p$ by definition of the Wasserstein distance,
    we conclude that
    \[
        \expect \left[ \sup_{r \in [0, t]} \bigl\lvert z_r \bigr\rvert^{p}  \right]
        \leq C \int_{0}^{t} \expect \left[ \sup_{u \in [0, r]} \bigl\lvert z_u \bigr\rvert^{p}  \right] \, \d r.
    \]
    By Gr\"onwall's lemma, and given that $z_0 = 0$,
    it follows that
    \[
        \expect \left[ \sup_{t \in [0, T]} \lvert z_t \rvert^p  \right] = 0,
    \]
    which concludes the proof of uniqueness.

    \paragraph{Step 5: Showing that $\mfldis_t$ solves the nonlocal Fokker--Planck equation \eqref{eq:cbs:FKP}}
    This part of the proof is standard,
    and we present only a formal derivation.
    Let $(m, \Gamma) \in C([0, T], \mathcal X)$  denote the fixed point of~$\mathcal T$,
    in which case $\mathcal R(\Gamma_t) = \Gamma_t$,
    and let $\overline Y_t$ denote the corresponding solution to~\eqref{eq:cbs-existence-SDE-for-u-and-G}.
    Denote also by~$\mfldis_t$ the law of~$\overline Y_t$.
    For a smooth function $\phi\colon \R^d \to \R$ with compact support,
    it holds by It\^o's lemma that
    \begin{equation}
        \label{eq:ito}
        \expect \phi\left(\overline Y_t\right) - \expect \phi\left(\overline Y_r\right)
        =  \expect \left[ \int_{r}^{t} \mathcal L_u \phi\left(\overline Y_u\right) \, \d u \right],
    \end{equation}
    where
    \(
    \mathcal L_u \phi(y) \coloneq -\bigl(y - \wm\left(\mfldis_u\right)\bigr) \cdot \nabla \phi(y) + \lambda^{-1} \wc\left(\mfldis_u\right) : \hessian \phi(y).
    \)
    It follows from~\eqref{eq:ito} that~$\mfldis_t$ satisfies,
    for all smooth compactly supported $\phi\colon \R^d \to \R$,
    \begin{equation}
        \label{eq:fokker_planck}
        \frac{\d}{\d t} \int_{\R^d} \phi(x) \, \mfldis_t(\d x)
        = \int_{\R^d} \mathcal L_t \phi \, \mfldis_t(\d x),
    \end{equation}
    which is indeed a weak formulation of the Fokker--Planck equation~\eqref{eq:CBO:FKP}.
    For rigorous uniqueness and existence results for the Fokker--Planck equation~\eqref{eq:fokker_planck}
    as well as a more precise discussion of its connection with~\eqref{eq:cbs-existence-SDE-for-u-and-G},
    see for example~\cite{MR2830391}.
\end{proof}

\begin{remark}[Proof of \cref{thm:cbo-mean-field-sde-well-posed}]
    \label{rmk:cbo-wellp-proof}
    Observe that the drift coefficients of CBO and CBS are identical and that the diffusion coefficient of CBO is a globally Lipschitz continuous function of its drift coefficient.
    By, among other things, replacing $\sqrt{2\lambda^{-1} \mathcal R (\Gamma_t)}$ in~\eqref{eq:cbs-existence-SDE-for-u-and-G} by $ \cbodifffunc\left(\xl_t - \wm( \mfldis_t) \right)$,
    the above proof can be easily adapted to provide a proof for~\cref{thm:cbo-mean-field-sde-well-posed}. See also  \cite{carrillo2018analytical}, where a weaker version of \cref{thm:cbo-mean-field-sde-well-posed} was already proven in the case where~$p=4$ and~$f \in \mathcal A(1, 0, 0) \cup \mathcal A(1, 2, 2)$.
\end{remark}

\begin{remark}
    Note that in case of $p\ge 2p_{\mathcal M}$, we can drop the requirement of $\mathcal{C}(\mfldis_0)$ to be positive definite by using \cref{proposition:stability:sqrt-wcov} instead of \cref{lemma:stab:wmean-wcov-basic}.
\end{remark}

\section{Extensions and perspectives for future work}
\label{sec:conclusion}

This section is organized as follows.
First, in \cref{sub:stopping_times},
we present an alternative approach based on stopping times for proving our main results~\cref{thm:mfl-cbo,thm:mfl-cbs}.
Then, in~\cref{sub:different_l_u},
we show how the stopping time approach presented in~\cref{sub:stopping_times} can be employed in order to obtain a mean-field result in the setting where~$f \in \mathcal A(s, \ell, u)$ with $u > \ell > 0$,
in which the improved stability estimates in~\cref{lemma:stab:improved} are no longer valid.

\subsection{Proving mean-field results using stopping times}
\label{sub:stopping_times}

In this section, we present an alternative approach for proving~\cref{thm:mfl-cbo} and~\cref{thm:mfl-cbs}.
This approach is based on stopping times,
and inspired by a method employed in~\cite{MR1949404} to handle local Lipschitz continuity issues in the context of numerical analysis for the Euler--Maruyama method.
A similar approach was employed in~\cite{MR4553241} to obtain a non-quantitative mean-field estimate,
and more recently in~\cite{vaes2024sharppropagationchaosensemble} to prove sharp propagation of chaos estimates for an ensemble sampler related to the ensemble Kalman filter~\cite{EKS-1-Hofmann-Stuart-Li-GI-2019,MR4123680}.
Before presenting the proof,
we briefly discuss the main differences between this approach and the indicator set approach presented in~\cref{sec:main}.
\begin{itemize}
    \item
          The main advantage of the approach used in~\cref{sec:main} is that the event $\Omega_{J,t}$ for which our stability estimates do not apply
          can be time-dependent.
          For the proof of \cref{thm:mfl-cbo},
          this flexibility is in fact not necessary,
          as we could have employed instead of~$\Omega_{J,t}$ in~\eqref{eq:cbo-bad-set} the event
          \begin{equation}
              \label{eq:event_stopping_time}
              \Omega_{J} := \left\{ \omega \in \Omega : \frac{1}{J} \sup_{t \in [0, T]}\sum_{j=1}^{J}  \myabs*{\xnl{j}_t(\omega)}^p \geq R \right\}.
          \end{equation}
          In certain contexts, however,
          being able to exclude a time-dependent event could be very useful.

    \item
          The stopping time approach presented hereafter does not enjoy this flexibility,
          but is convenient in settings where moments of both $\mu^J$ and $\overline \mu^J$ need to be controlled in order to apply the stability estimates obtained in~\cref{lemma:stab:wmean-wcov-basic} for the weighted moments.
          This is illustrated in~\cref{sub:different_l_u}.
\end{itemize}

We illustrate the stopping times approach only for~\cref{thm:mfl-cbs},
noting that~\cref{thm:mfl-cbo} can be proved in exactly the same manner.
Furthermore, we assume for simplicity that~$\mfldis_0$ has infinitely many moments.
\begin{proof}
    [Alternative proof of~\cref{thm:mfl-cbs}]
    For $R>0$, let $\theta_J(R)$ be the stopping time
    \begin{equation}
        \label{eq:basic_stopping_time}
        \theta_J(R) := \inf \Bigl\{ t \geq 0 : \frac{1}{J} \sum_{j=1}^{J} \left\lvert \xnl{j}_t \right\rvert^p \geq R \Bigr\},
    \end{equation}
    Note that the event $\{\theta_J \leq T\}$ coincides with the event $\Omega_J$ given in \eqref{eq:event_stopping_time}.
    By Hölder's inequality,
    it holds for all~$q > p$ that
    \begin{align}
        \notag
         & \frac{1}{2^{p-1}} \expect \left[ \sup_{t\in[0,T]} \myabs[\Big]{\xn{j}_t - \xnl{j}_t}^p \right]                                                \\
        \notag
         & \qquad \qquad = \expect \left[ \sup_{t\in[0,T]} \myabs[\Big]{\xn{j}_t - \xnl{j}_t}^p \mathsf 1_{\theta_J(R) > T} \right]
        + \expect \left[ \sup_{t\in[0,T]} \myabs[\Big]{\xn{j}_t - \xnl{j}_t}^p \mathsf 1_{\theta_J(R) \leq T} \right]                                    \\
        \label{eq:main_equation_cbs}
         & \qquad\qquad \leq \expect \left[ \sup_{t\in[0,T]} \myabs[\Big]{X_{t \wedge \theta_J}^{(j)} - \overline X_{t \wedge \theta_J}^{(j)}}^p \right]
        + \left(\expect \left[ \sup_{t\in[0,T]} \myabs[\Big]{\xn{j}_t - \xnl{j}_t}^{q} \right]\right)^{\frac{p}{q}}
        \Bigl( \proba [\theta_J(R) \leq T]\Bigr)^{\frac{q-p}{q}}.
    \end{align}

    \paragraph{Bounding the first term in~\eqref{eq:main_equation_cbs}}
    We again employ an adapted version of Sznitman's argument.
    We have
    \begin{align*}
        \frac{1}{2^{p-1}}
        \left\lvert X_{t \wedge \theta_J}^{(j)} - \overline X_{t \wedge \theta_J}^{(j)} \right\rvert^p
        \leq
         & \left\lvert \int_{0}^{t \wedge \theta_J} b\Bigl(\xn{j}_s, \mu^J_{s}\Bigr) - b\Bigl(\xnl{j}_s, \mfldis_{s} \Bigr) \, \d s \right\rvert^p                    \\
         & +  \left\lvert \int_{0}^{t \wedge \theta_J} \sigma\Bigl(X_{s}^{(j)}, \mu^J_{s}\Bigr) - \sigma\Bigl(\xnl{j}_s, \mfldis_{s} \Bigr) \, \d W_s \right\rvert^p.
    \end{align*}
    Let us introduce the martingale
    \[
        M_t = \int_{0}^{t} \sigma\Bigl(X_{s}^{(j)}, \mu^J_{s}\Bigr) - \sigma\Bigl(\xnl{j}_s, \mfldis_{s} \Bigr) \, \d W_s.
    \]
    By Doob's optional stopping theorem~\cite[Theorem 3.3]{MR1949404},
    see also~\cite[Equation 2.29, p.285]{MR0838085},
    the stopped process $(M_{t \wedge \theta_J})_{t\geq 0}$ is a martingale,
    with a quadratic variation process given by the stopped quadratic variation process of $M_t$,
    i.e.\ by $(\langle M \rangle_{t \wedge \theta_J})_{t \geq 0}$.
    Therefore, by the Burkholder--Davis--Gundy inequality,
    we have for all $t \in [0, T]$ that
    \begin{align}
        \notag
        \expect \left[ \sup_{s \in [0,t]} \bigl\lvert \xn{j}_{s \wedge \theta_J} - \xnl{j}_{s \wedge \theta_J} \bigr\rvert^p \right]
         & \leq (2T)^{p-1} \expect \int_{0}^{t \wedge \theta_J} \left\lvert b\Bigl(\xn{j}_{s}, \mu^J_{s}\Bigr) - b\Bigl(\xnl{j}_{s}, \mfldis_{s} \Bigr) \right\rvert^p \, \d s \\
        \label{eq:main_first_term}
         & \qquad + C_{\rm BDG} 2^{p-1}T^{\frac{p}{2} - 1} \expect \int_{0}^{t \wedge \theta_J}
        \left\lVert \sigma\Bigl(\xn{j}_{s}, \mu^J_{s}\Bigr) - \sigma\Bigl(\xnl{j}_{s}, \mfldis_{s} \Bigr) \right\rVert_{\rm F}^p \, \d s.
    \end{align}
    Applying the triangle inequality gives that
    \begin{align}
        \notag
        \frac{1}{2^{p-1}}\expect \int_{0}^{t \wedge \theta_J} \bigl\lvert b\Bigl(\xn{j}_s, \mu^J_{s}\Bigr) - b\Bigl(\xnl{j}_{s}, \mfldis_{s} \Bigr) \bigr\rvert^p \, \d s
         & \leq \int_{0}^{t} \expect \left\lvert b\Bigl(\xn{j}_{s \wedge \theta_J}, \mu^J_{s \wedge \theta_J}\Bigr)
        - b\Bigl(\xnl{j}_{s \wedge \theta_J}, \overline \mu^J_{s \wedge \theta_J}\Bigr) \right\rvert^p \, \d s                                                 \\
        \label{eq:decomposition_drift_cbs}
         & \qquad + \int_{0}^{t} \expect \left\lvert b\Bigl(\xnl{j}_s, \overline \mu^J_{s}\Bigr) - b\Bigl(\xnl{j}_s, \mfldis_{s}\Bigr) \right\rvert^p \, \d s.
    \end{align}
    Similarly, for the diffusion term,
    we have
    \begin{align}
        \notag
        \frac{1}{2^{p-1}}\expect \int_{0}^{t \wedge \theta_J} \left\lVert \sigma\Bigl(\xn{j}_s, \mu^J_{s}\Bigr)
        - \sigma\Bigl(\xnl{j}_s, \mfldis_{s} \Bigr) \right\rVert_{\rm F}^p \, \d s
         & \leq \int_{0}^{t} \expect \left\lVert \sigma\Bigl(\xn{j}_{s\wedge \theta_J}, \mu^J_{s \wedge \theta_J}\Bigr)
        - \sigma\Bigl(\xnl{j}_{s\wedge \theta_J}, \overline \mu^J_{s \wedge \theta_J}\Bigr) \right\rVert_{\rm F}^p \, \d s \\
        \label{eq:decomposition_diffusion_cbs}
         & \qquad + \int_{0}^{t} \expect \left\lVert \sigma\Bigl(\xnl{j}_s, \overline \mu^J_{s}\Bigr)
        - \sigma\Bigl(\xnl{j}_s, \mfldis_{s}\Bigr) \right\rVert_{\rm F}^p \, \d s.
    \end{align}
    Next, we need to bound the terms on the right-hand side of~\eqref{eq:decomposition_drift_cbs} and~\eqref{eq:decomposition_diffusion_cbs}.
    Since the drift terms and the diffusion terms are bounded in a similar manner, we only present the bounds for the diffusion terms.
    \begin{itemize}
        \item
              For the first diffusion term,
              we have that
              \begin{align*}
                  \left\lVert \sigma\Bigl(\xn{j}_{s\wedge \theta_J}, \mu^J_{s \wedge \theta_J}\Bigr)
                  - \sigma\Bigl(\xnl{j}_{s\wedge \theta_J}, \overline \mu^J_{s \wedge \theta_J}\Bigr) \right\rVert_{\rm F}^p
                   & = \left(2 \lambda^{-1}\right)^{\frac{p}{2}} \biggl\lVert \sqrt{\wc\left(\mu^J_{s \wedge \theta_J} \right)} - \sqrt{\wc\left(\overline \mu^J_{s \wedge \theta_J} \right) } \biggr\rVert_{\rm F}^p.
              \end{align*}
              Using~\cref{lemma:stab:improved}, together with the fact
              \[
                  \expect \Bigl[\wasserstein_p\bigl(\mu^J_{s \wedge \theta_J}, \overline \mu^J_{s \wedge \theta_J}\bigr)^p\Bigr]
                  \leq \expect \left[ \frac{1}{J}\sum_{k=1}^{J} \left\lvert \xn{k}_{s\wedge \theta_J} - \xnl{k}_{s\wedge \theta_J} \right\rvert^p \right]
                  = \expect \left\lvert \xn{j}_{s\wedge \theta_J} - \xnl{j}_{s\wedge \theta_J} \right\rvert^p,
              \]
              we deduce that
              \[
                  \expect \left\lVert \sigma\Bigl(\xn{j}_{s\wedge \theta_J}, \mu^J_{s \wedge \theta_J}\Bigr)
                  - \sigma\Bigl(\xnl{j}_{s\wedge \theta_J}, \overline \mu^J_{s \wedge \theta_J}\Bigr) \right\rVert_{\rm F}^p
                  \leq C \expect \left\lvert \xn{j}_{s \wedge \theta_J} - \xnl{j}_{s \wedge \theta_J} \right\rvert^p.
              \]
        \item
              For the other diffusion term,
              we use~\cref{lemma:convergence_weighted_covariance_iid}.
              The assumptions of this lemma are satisfied,
              since it holds that $\mfldis_s \in \mathcal P_q(\R^d)$ for all $q > 0$ by~\cref{thm:cbs-mean-field-sde-well-posed},
              and additionally~$\wc(\mfldis_s) \succ 0$ by~\cref{prop:cbs-minimal-eigenvalue-wcov}.
              We deduce that
              \[
                  \expect \left\lVert \sigma\Bigl(\xnl{j}_s, \overline \mu^J_{s}\Bigr)
                  - \sigma\Bigl(\xnl{j}_s, \mfldis_{s}\Bigr) \right\rVert_{\rm F}^p
                  = \left(2\lambda^{-1}\right)^{\frac{p}{2}} \expect \left\lVert \sqrt{\wc \left(\overline \mu^J_{s}\right)} - \sqrt{\wc \left(\mfldis_s \right)} \right\rVert_{\rm F}^p
                  \leq C J^{-\frac{p}{2}}.
              \]
    \end{itemize}
    Combining the bounds on both terms,
    we deduce that
    \[
        \expect \int_{0}^{t \wedge \theta_J} \left\lVert \sigma\Bigl(\xn{j}_s, \mu^J_{s}\Bigr) - \sigma\Bigl(\xnl{j}_s, \mfldis_{s} \Bigr) \right\rVert_{\rm F}^p \, \d s
        \leq
        C J^{-\frac{p}{2}} +
        C \int_{0}^t \expect \left[ \sup_{u \in [0, s]}\left\lvert \xn{j}_{u \wedge \theta_J} - \xnl{j}_{u \wedge \theta_J} \right\rvert^p \right] \, \d s.
    \]
    Applying a similar reasoning for the drift terms,
    we finally obtain from~\eqref{eq:main_first_term} that
    \[
        \expect \left[ \sup_{s\in[0,t]} \myabs[\Big]{\xn{j}_{s\wedge \theta_J} - \xnl{j}_{s\wedge \theta_J}}^p \right]
        \leq C J^{- \frac{p}{2}} +
        C \int_{0}^{t} \expect \left[ \sup_{u\in[0,s]} \Bigl\lvert X_{u \wedge \theta_J}^{(j)} - \overline X_{u \wedge \theta_J}^{(j)} \Bigr\rvert^p \right] \, \d s.
    \]
    By Gr\"onwall's inequality,
    this implies that
    \[
        \expect \left[ \sup_{t\in[0,T]} \Bigl\lvert X_{t \wedge \theta_J}^{(j)} - \overline X_{t \wedge \theta_J}^{(j)} \Bigr\rvert^p \right]
        \leq C J^{- \frac{p}{2}}.
    \]
    \paragraph{Bounding the second term in~\eqref{eq:main_equation_cbs}}
    We have by~\cref{lemma:moment-bounds-empirical} and  \eqref{eq:cbs-moment-bound-mfl}
    that
    \[
        \expect \left[ \sup_{t\in[0,T]} \myabs[\big]{\xn{j}_t - \xnl{j}_t}^{q} \right]
        \leq 2^{q-1} \left( \expect \left[ \sup_{t\in[0,T]} \myabs[\big]{\xn{j}_t}^{q} \right]
        + \expect \left[ \sup_{t\in[0,T]} \myabs[\big]{\xnl{j}_t}^{q} \right] \right)
        \leq 2^{q} \kappa.
    \]
    In addition, by~\cref{lemma:small_set}, it holds for any sufficiently large fixed~$R$,
    that for all~$\alpha > 0$ there is $C_{\alpha}$ such that
    \[
        \proba [\theta_J(R) \leq T]
        = \proba \left[\sup_{t\in[0,T]} \frac{1}{J} \sum_{j=1}^{J} \left\lvert \xnl{j}_t \right\rvert^p \geq R\right]
        \leq C_{\alpha} J^{-\alpha}.
    \]

    \paragraph{Conclusion}
    We finally obtain from~\eqref{eq:main_equation_cbs} that
    \[
        \expect \left[ \sup_{t\in[0,T]} \myabs[\Big]{\xn{j}_t - \xnl{j}_t}^p \right]
        \leq C \left(J^{-\frac{p}{2}} + J^{- \frac{\alpha(q-p)}{q}} \right),
    \]
    which concludes the proof since~$q$ and~$\alpha$ can be taken arbitrarily large.
\end{proof}
\subsection{\texorpdfstring
    {Handling the case where $u > \ell>0$ in \cref{assump:main}}
    {Handling more general growth assumptions}
}
\label{sub:different_l_u}
Our main results stand only when $f \in \mathcal A(s, \ell, u)$ with $u = \ell$.
This is a convenient setting because it then holds by~\cref{lemma:bound-on-weighted-moment} that
the $p$-th moment of the reweighted measure $L_{\beta} \mu$ is controlled by the $p$-th moment of~$\mu$,
which is key in order
to prove well-posedness of the mean-field system,
to obtain moment estimates for empirical measures,
and, most importantly, to improve the Wasserstein stability estimates for the weighted mean and weighted covariance,
see~\cref{lemma:stab:improved}.
In the interest of simplicity, we refrained from stating precise results for the cases $u > \ell = 0$
and $u > \ell > 0$.
We do not address the former setting,
but note that a quantitative mean-field estimate can still be obtained in the latter.
The main ingredients to this end are the following:
\begin{itemize}
    \item
          Moment estimates for the empirical measure and well-posedness of the mean-field equation
          can be obtained by using~\cref{lemma:bound-on-weighted-moment} together with the Bihari--LaSalle inequality instead of Gr\"onwall's inequality,
          using the same reasoning as in~\cref{lemma:moment-bounds-empirical} and~\cref{sec:cbo-cbs-wellp},
          respectively.
          In contrast with the setting $\ell = u$,
          these results hold only up to some finite time $T < T_+$, where $T_+\in(0,\infty)$ is the time at which the bounds given by the Bihari--LaSalle inequality become singular.

    \item
          In order to obtain a mean-field result up to some time $T < T_+$,
          the approach presented in~\cref{sub:stopping_times} can be used,
          but with the stopping time~$\theta_J := \tau_J \wedge \overline \tau_J$,
          where
          \begin{align*}
              \tau_J(R)           & = \inf \Bigl\{ t \geq 0 : \mu_t^J\bigl[\lvert x \rvert^p\bigr] \geq R \Bigr\},           \\
              \overline \tau_J(R) & = \inf \Bigl\{ t \geq 0 : \overline \mu_t^J\bigl[\lvert x \rvert^p\bigr] \geq R \Bigr\}.
          \end{align*}
          Employing this stopping time instead of~\eqref{eq:basic_stopping_time} enables to use the basic stability estimates for the weighted moments of~\cref{lemma:stab:wmean-wcov-basic},
          which are valid even when $\ell < u$,
          instead of the improved stability estimates in~\cref{lemma:stab:improved},
          which require that $\ell = u$.
          It is then simple to obtain a propagation of chaos estimate for the stopped particle system,
          of the form
          \[
              \expect \left[ \sup_{t\in[0,T]} \Bigl\lvert X_{t \wedge \theta_J}^{(j)} - \overline X_{t \wedge \theta_J}^{(j)} \Bigr\rvert^p \right]
              \leq C J^{- \frac{p}{2}}.
          \]
          In order to complete the proof as in~\cref{sub:stopping_times},
          it remains to bound the probability $\proba \left[\theta_J \leq T \right]$,
          which can be achieved by noticing that
          \[
              \proba [\theta_J(R) \leq T]
              =
              \proba \left[\tau_J(R) \leq T < \overline \tau_J(R) \right]
              + \proba \left[\overline \tau_J(R) \leq T\right].
          \]
          Using the almost sure continuity of the solution to the interacting particle system by~\cref{thm:cbo-cbs-particles-well-posed},
          together with the triangle inequality,
          we bound the first probability as follows:
          \begin{align*}
              \proba [\tau_J \leq T < \overline \tau_J]
               & \leq \proba \left[\sup_{t\in[0,T]} \frac{1}{J}\sum_{j=1}^J \Bigl\lvert \xn{j}_{t \wedge \theta_J} \Bigr\rvert^p = R \right]                     \\
               & \leq
              \proba \left[\frac{2^{p-1}}{J}\sum_{j=1}^J \sup_{t\in[0,T]} \Bigl\lvert \xnl{j}_t \Bigr\rvert^p
              + \frac{2^{p-1}}{J}\sum_{j=1}^J \sup_{t\in[0,T]} \Bigl\lvert \xn{j}_{t \wedge \theta_J} - \xnl{j}_{t \wedge \theta_J} \Bigr\rvert^p \geq R \right] \\
               & \leq
              \proba \left[\frac{1}{J}\sum_{j=1}^J \sup_{t\in[0,T]} \Bigl\lvert \xnl{j}_t \Bigr\rvert^p \geq \frac{R}{2^p} \right]
              + \proba \left[ \frac{1}{J} \sum_{j=1}^J \sup_{t\in[0,T]} \Bigl\lvert \xn{j}_{t \wedge \theta_J} - \xnl{j}_{t \wedge \theta_J} \Bigr\rvert^p \geq \frac{R}{2^p} \right],
          \end{align*}
          where we used that $\proba[A +  B \geq k] \leq \proba [A \geq k/2] + \proba [B \geq k/2]$
          for any two real-valued random variables~$A$ and~$B$,
          because $\{A + B \geq k\} \subset \{A \geq k/2\} \cup \{B \geq k/2\}$.
          A similar trick, of bounding the probability of an event associated with a stochastic system
          in terms of a similar event associated with a simpler system and the discrepancy between the two systems,
          is employed in~\cite[Theorem 2.10]{2023arXiv230503637D}.
          The probability $\proba [\theta_J \leq T]$ can then be bounded as follows:
          \begin{align*}
              \proba [\theta_J(R) \leq T]
               & =
              \proba \left[\tau_J(R) \leq T < \overline \tau_J \right]
              + \proba \left[\overline \tau_J(R) \leq T\right]                                                                              \\
               & \leq 2\proba \left[\frac{1}{J}\sum_{j=1}^J \sup_{t\in[0,T]} \Bigl\lvert \xnl{j}_t \Bigr\rvert^p \geq \frac{R}{2^p} \right]
              + \proba \left[ \frac{1}{J}\sum_{j=1}^J \sup_{t\in[0,T]} \Bigl\lvert \xn{j}_{t\wedge \theta_J} - \xnl{j}_{t \wedge \theta_J} \Bigr\rvert^p \geq \frac{R}{2^{p}} \right].
          \end{align*}
          Using Markov's inequality for the second term,
          we then obtain
          \begin{equation}
              \label{eq:auxiliary_bound_proba}
              \proba [\theta_J(R) \leq T]
              \leq 2\proba \left[\frac{1}{J}\sum_{j=1}^J \sup_{t\in[0,T]} \Bigl\lvert \xnl{j}_t \Bigr\rvert^p \geq \frac{R}{2^p} \right]
              + \frac{2^p}{R} \expect \left[ \frac{1}{J}\sum_{j=1}^J \sup_{t\in[0,T]} \Bigl\lvert \xn{j}_{t\wedge \theta_J} - \xnl{j}_{t\wedge \theta_J} \Bigr\rvert^p \right].
          \end{equation}
          By~\cref{lemma:small_set},
          For any~$\alpha > 0$ there exists $C_{\alpha} > 0$ such that
          the first term is bounded from above by~$C_{\alpha} J^{-\alpha}$ for all~$J \in \nat^+$,
          whereas the second term can be bounded by~$CJ^{-\frac{p}{2}}$ as in the proof presented in~\cref{sub:stopping_times}.
          Thus, the probability $\proba [\theta_J(R) \leq T]$ scales as~$J^{-\min\left\{\alpha, \frac{p}{2}\right\}}$,
          instead of~$J^{-\alpha}$ in the case where $u = \ell$.
          As a consequence,
          we obtain a mean-field result of the type
          \[
              \forall J \in \N^+, \qquad
              \expect \left[ \sup_{t\in[0,T]} \left\lvert X_t^{(j)} - \xnl{j}_t \right\rvert^p \right]
              \le C J^{- \frac{p}{2} \left( \frac{q-p}{q} \right)}.
          \]
          Since~$q$ was arbitrary given the assumption that~$\mfldis_0$ has infinitely many moments,
          it follows that for every~$\varepsilon > 0$,
          there exists $C_{\varepsilon}$ such that
          \[
              \forall J \in \N^+, \qquad
              \expect \left[ \sup_{t\in[0,T]} \left\lvert X_t^{(j)} - \xnl{j}_t \right\rvert^p \right]
              \le C_{\varepsilon} J^{- \frac{p}{2} \left( 1 - \varepsilon \right)}.
          \]
          Hence, we recover almost the classical Monte Carlo rate.
          By using the same method as in~\cite{vaes2024sharppropagationchaosensemble},
          it is in fact possible to prove that this statement holds also for~$\varepsilon = 0$,
          but for conciseness we do not present the details here.
\end{itemize}

\appendix
\section{Further technical results}
\label{sec:aux-results}

\subsection{Stability estimates for ratios of integrals}

\begin{lemma}
    \label{lemma:stab-aux}
    For a real finite-dimensional vector space~$\mathcal V$ with norm~$\norm{\placeholder}$,
    let $g \colon \R^d\rightarrow \mathcal V$ and $h\colon \R^d\rightarrow (0,\infty)$ be functions such that the following condition is satisfied for some $\varsigma \geq 0$ and $L>0$:
    \begin{equation}
        \label{item:lem-stab-aux:loc-lip}
        \forall x, y \in \R^d, \qquad
        \bigl\lVert g(x)-g(y) \bigr\rVert \vee \bigl\lvert h(x)-h(y) \bigr\rvert
        \le L \Bigl( 1 + |x| + |y| \Bigr)^{\varsigma} \myabs{x-y}.
    \end{equation}
    Then for all $p\ge \varsigma+1$ and all $R>0$,
    there exist a constant $C=C(p, R)$ such that
    \begin{align*}
        \forall (\mu, \nu) \in \mathcal P_{p,R}(\real^d) \times \mathcal P_{p,R}(\real^d),
        \qquad
        \norm{ \frac{\mu[g]}{\mu[h]} - \frac{\nu[g]}{\nu[h]} }
        \le C \wasserstein_p(\mu, \nu).
        \qquad
    \end{align*}
\end{lemma}
\begin{proof} Fix $p\ge \varsigma+1$ and $R>0$.
    \paragraph{Step 1}
    Let us first generalize \cite[Lemma 3.1]{carrillo2018analytical} by showing that there exists a constant $\eta = \eta(p, R) > 0$ such that
    \begin{equation}
        \label{eq:lem-stab-aux:denom-pos}
        \forall \mu \in \mathcal P_{p,R}(\R^d),
        \qquad
        \int_{\R^d} h \, \d \mu \geq \eta.
    \end{equation}
    Define $B_K \coloneq \{ x \in \R^d : \myabs{x} \leq K \}$ where $K\coloneq 2^{1/p}R$.
    By Markov's inequality, we have
    \[
        \forall \mu \in \mathcal P_{p,R}(\R^d),
        \qquad
        \mu \left(B_K\right) \ge 1- \mu\left(\R^d \setminus B_K\right) \ge 1- \frac{R^p}{K^p} = \frac{1}{2}.
    \]
    Therefore we obtain
    \[
        \forall \mu \in \mathcal P_{p,R}(\R^d), \qquad
        \int_{\R^d} h \, \d \mu
        \geq \int_{B_K} h  \, \d \mu
        \geq \frac{1}{2} \min \bigl\{ h (x) : x \in B_K \bigr\}  =: \eta>0,
    \]
    since $h$ is continuous by \cref{item:lem-stab-aux:loc-lip}.
    This proves~\eqref{eq:lem-stab-aux:denom-pos}.
    \paragraph{Step 2}
    Fix probability measures $\mu, \nu \in \mathcal P_{p, R}$.
    We omit the integration domains for conciseness.
    For any coupling $\pi \in\Pi(\mu, \nu)$,
    we have by \eqref{item:lem-stab-aux:loc-lip} and \eqref{eq:lem-stab-aux:denom-pos} that
    \begin{align*}
        \norm{\frac{\mu[g]}{\mu[h]} -
            \frac{\nu[g]}{\nu[h]}}
         & \le
        \norm{ \frac{\iint g(x)-g(y) \, \pi(\d x \, \d y)}{\mu[h]}}
        +         \norm{\int g \, \d \nu}\cdot\myabs*{ \frac{\iint h(x)-h(y) \, \pi(\d x \, \d y)}{\mu[h] \nu[h]}}
        \\
         & \le L \left( \eta^{-1} + \eta^{-2}\int \norm{g} \, \d \nu  \right)
        \iint \Bigl(1+|x|+|y| \Bigr)^{\varsigma} \myabs{x-y} \, \pi(\d x \, \d y).
    \end{align*}
    By \eqref{item:lem-stab-aux:loc-lip}, there exists $C>0$ such that $\norm{g(x)} \le C \bigl(1+|x|\bigr)^{\varsigma+1}$ for all $x \in \R^d$.
    The claim then follows by first applying Hölder's inequality to obtain the bound
    \[
        \iint \Bigl(1+|x|+|y| \Bigr)^{\varsigma} \myabs{x-y} \, \pi(\d x \, \d y)
        \leq
        \left(\iint \Bigl(1+|x|+|y| \Bigr)^{\frac{\varsigma p}{p-1}} \, \pi(\d x \, \d y)\right)^{\frac{p-1}{p}}
        \left(\iint \lvert x - y \rvert^p \, \pi(\d x \, \d y)\right)^{\frac{1}{p}},
    \]
    noting that~$\frac{\varsigma p}{p - 1} \leq p$ under the assumption that $p \geq \varsigma + 1$,
    and then infimizing over $\pi\in\Pi(\mu, \nu)$.
\end{proof}

\begin{lemma}
    \label{lemma:aux_aux_rt_cov}
    Let $g\colon\R^d\rightarrow \R^{m\times n}$, $h\colon\R^d\rightarrow (0,\infty)$ and $\mathcal S\colon \mathcal{P}(\R^d) \rightarrow \R^n$
    be functions such that the following conditions hold.
    \begin{enumerate}[label=(\alph*)]
        \item \label{item:lem-stab-aux-special:loc-lip}
              There is $\varsigma \geq 0$ and $L > 0$ such that
              \[
                  \forall x, y \in \R^d, \qquad
                  \bigl\lVert g(x)-g(y) \bigr\rVert \vee \bigl\lvert h(x)-h(y) \bigr\rvert
                  \le L \Bigl( 1 + |x| + |y| \Bigr)^{\varsigma} \myabs{x-y}.
              \]

        \item \label{item:lem-stab-aux-special:S-lip}
              The map $\mathcal S\colon \mathcal{P}(\R^d) \rightarrow \R^n$ satisfies
              \[
                  \forall (\mu, \nu) \in \mathcal P_{p,R}(\real^d) \times \mathcal P_{p,R}(\real^d), \qquad
                  \myabs{\mathcal S(\mu)-\mathcal S(\nu)}\le L_{\mathcal S} \wasserstein_p(\mu, \nu).
              \]
    \end{enumerate}
    Then, for all $R > 0$ and $p \geq 2\varsigma +2$,
    there exist a constant~$C > 0$ such that for all~$\mu, \nu \in \mathcal P_{p,R}$
    and all couplings $\pi \in \Pi(\mu, \nu)$, it holds that
    \begin{align}
        \label{eq:rhs_lemma_empirical_measure}
        \left( \iint_{\R^d\times \R^d}
        \norm{ \frac{g(x)\mathcal S(\mu)}{\sqrt{\mu[h]}} -
            \frac{g(y)\mathcal S(\nu)}{\sqrt{\nu[h]}}}^2 \, \pi(\d x \, \d y)
        \right)^{1/2}
        \le C \, \left( \int \myabs{x - y}^{p} \, \pi(\d x \, \d y) \right)^{1/p}.
    \end{align}
\end{lemma}

\begin{proof}
    Fix $p\ge 2\varsigma +2$ and $R>0$.
    Again, there is a constant $\eta = \eta(R, p) > 0$ such that \eqref{eq:lem-stab-aux:denom-pos} holds.
    From the triangle inequality we obtain
    \begin{align*}
        \left\lVert  \frac{g(x)\mathcal S(\mu)}{\sqrt{\mu[h]}} -
        \frac{g(y)\mathcal S(\nu)}{\sqrt{\nu[h]}} \right\rVert_{L^2(\pi)}
         & \le
        \left\lVert  \frac{g(x)}{\sqrt{\mu[h]}} \Bigl(\mathcal S(\mu)- \mathcal S(\nu)\Bigr) \right\rVert_{L^2(\pi)}
        + \left\lVert  \Bigl(g(x) - g(y)\Bigr) \frac{\mathcal S(\nu)}{\sqrt{\mu[h]}}  \right\rVert_{L^2(\pi)}
        \\
         & \quad + \left\lVert {g(y)\mathcal S(\nu)} \left( \frac{1}{\sqrt{\mu[h]}} -  \frac{1}{\sqrt{\nu[h]}}\right) \right\rVert_{L^2(\pi)},
    \end{align*}
    where, by a slight abuse of notation,
    we wrote $g(x)$ and $g(y)$ in order to refer to the functions~$(x, y) \mapsto g(x)$ and~$(x, y) \mapsto g(y)$,
    respectively.
    In the rest of the proof,
    we denote by~$C$ any constant depending only on the parameters~$(R, L, \varsigma, \eta, p)$.

    \paragraph{Bounding the first term}
    It follows from \cref{item:lem-stab-aux-special:loc-lip} that $\norm{g(x)}^2\le C (1+|x|)^{2\varsigma +2}$ for all $x\in\R^d$, for a constant~$C$.
    Using~\cref{item:lem-stab-aux-special:S-lip}, we obtain the bound
    \begin{align*}
        \left( \iint
        \norm{ \frac{g(x)}{\sqrt{\mu[h]}}
            \bigl(\mathcal S(\mu)- \mathcal S(\nu)\bigr)}^2  \, \pi(\d x \, \d y)\right)^{1/2}
        \le
        \sqrt{\frac{\mu[\lVert g(x) \rVert^2]}{\mu[h]}}
        \bigl\lvert  \mathcal S(\mu)- \mathcal S(\nu) \bigr\rvert
        \le \sqrt{\frac{C}{\eta}}L_{\mathcal S} \wasserstein_p(\mu, \nu).
    \end{align*}

    \paragraph{Bounding the second term}
    For the second term, we use~\cref{item:lem-stab-aux-special:loc-lip} to find
    \[
        \forall x, y \in \R^d, \qquad
        \myabs{g(x)-g(y)}^2 \le L^2\left( 1 + |x|+|y|\right)^{2 \varsigma} \myabs{x-y}^2.
    \]
    Using H\"older's inequality,
    we therefore obtain
    \begin{align*}
        \bigl\lVert g(x) - g(y) \bigr\rVert_{L^2(\pi)}
         & \le L
        \left( \int
        \Bigl( 1+|x|+|y|\Bigr)^{\frac{2\varsigma p}{p-2}} \, \pi(\d x \, \d y)
        \right)^{\frac{p-2}{2p}}
        \left( \int
        \myabs{x - y}^{p} \, \pi(\d x \, \d y)
        \right)^{\frac{1}{p}}
        \\
         & \le
        C
        \left( \int
        \myabs{x - y}^{p} \, \pi(\d x \, \d y)
        \right)^{\frac{1}{p}},
    \end{align*}
    where we used that  $\frac{2\varsigma p}{p-2} \le p$ since $p \geq 2 + 2 \varsigma$.
    In addition, from \cref{item:lem-stab-aux-special:S-lip} we have
    \begin{equation}
        \label{eq:bound_lipschitz_s}
        \forall \mu \in \mathcal P_{p, R}(\real^d), \qquad
        \myabs[\big]{\mathcal S(\mu)}
        \le \myabs[\big]{\mathcal S(\delta_0)} + L_{\mathcal S} \wasserstein_p(\delta_0, \mu)
        \le \myabs[\big]{\mathcal S(\delta_0)} + L_{\mathcal S} R,
    \end{equation}
    and so we conclude that
    \begin{equation}
        \label{eq:second_term_rhs}
        \left\lVert  \Bigl(g(x) - g(y)\Bigr) \frac{\mathcal S(\nu)}{\sqrt{\mu[h]}}  \right\rVert_{L^2(\pi)}
        \leq
        C\left( \int
        \myabs{x - y}^{p} \, \pi(\d x \, \d y)
        \right)^{1/p}.
    \end{equation}

    \paragraph{Bounding the third term}
    Using again~\eqref{eq:bound_lipschitz_s},
    we have
    \begin{align*}
        \left\lVert {g(y)\mathcal S(\nu)} \left( \frac{1}{\sqrt{\mu[h]}} -  \frac{1}{\sqrt{\nu[h]}}\right) \right\rVert_{L^2(\pi)}
        =
        \lVert g \rVert_{L^2(\mu)} \myabs*{\mathcal S(\nu)}
        \left(
        \frac{1}{\sqrt{\mu[h]}}
        -  \frac{1}{\sqrt{\nu[h]}}\right)
        \le C \left\lvert \frac{\sqrt{\mu[h]} - \sqrt{\nu[h]}}{\sqrt{\mu[h]} \sqrt{\nu[h]}} \right\rvert.
    \end{align*}
    Since $\sqrt{a} - \sqrt{b} = (a - b) / (\sqrt{a} + \sqrt{b})$ for all $a, b > 0$,
    we deduce that
    \begin{align*}
        \left\lVert {g(x)\mathcal S(\nu)} \left( \frac{1}{\sqrt{\mu[h]}} -  \frac{1}{\sqrt{\nu[h]}}\right) \right\rVert_{L^2(\pi)}
         & \le C \eta^{-{\frac{3}{2}}} \iint \bigl\lvert h(x) - h(y) \bigr\rvert \, \pi(\d x \, \d y) \\
         & \le C \iint \bigl(1+|x|+|y|\bigr)^{\varsigma }|x-y| \, \pi(\d x \, \d y).
    \end{align*}
    By employing H\"older's inequality, the integral can be bounded by the right-hand side of~\eqref{eq:second_term_rhs} up to multiplication by a constant.

    \paragraph{Conclusion} The claim follows then by combining the bounds on the three terms,
    and noting,
    for the first term,
    that $\wasserstein_p(\mu, \nu)$ is bounded from above by the bracket on the right-hand side of~\eqref{eq:rhs_lemma_empirical_measure}
    by definition of the Wasserstein distance.
\end{proof}

\subsection{Bound on weighted moments}
\label{sub:bound_weighted_moments}

If the function $f$ is bounded from below and above,
it is trivial to show that
\begin{equation*}
    \forall \mu \in \mathcal P_p(\R^d), \qquad
    \frac
    {\int_{\R^d} \myabs{x}^p \, \e^{-\beta f(x)} \, \mu(\d x)}
    {\int_{\R^d} \e^{-\beta f(x)} \, \mu(\d x)}
    \leq
    C \int_{\R^d} \myabs{x}^p \, \mu(\d x)
\end{equation*}
for a constant $C$ depending only on $\beta$ and $f$.
When~$f$ is unbounded,
proving an inequality of this type is more involved.

    To this end, 
    we first state and prove the following auxiliary lemma.

\begin{lemma}
    \label{lemma:auxiliary_auxiliary_lemma}
    For all $p, q, \beta > 0$,
    there is $C = C(p, q, \beta)$ such that
    \begin{equation}
        \label{eq:aux_aux}
        \forall \nu \in \mathcal P_q(\R^+), \qquad
        \frac
        {\int_{\R^+} y^p \, \e^{-\beta y} \, \nu(\d y)}
        {\int_{\R^+} \e^{-\beta y} \, \nu(\d y)}
        \leq
        C \left(1 + \int_{\R^+} y^q \, \nu(\d y) \right)^{\frac{p}{q}},
    \end{equation}
    where $\R^+ \coloneq [0, \infty)$.
    The constant is given by
    \[
        C(q,p,\beta) = \left(1 - \frac{1}{2^q}\right)^{-1}  (p \beta^{-1})^p \e^{2 \beta -p}.
    \]
\end{lemma}
\begin{proof}
    Fix~$\nu \in \mathcal P(\R^+)$.
    If the right-hand side of~\eqref{eq:aux_aux} is infinite,
    then the inequality~\eqref{eq:aux_aux} holds with any constant~$C$,
    so we assume from now on that
    \[
        \lambda \coloneq \left( \int_{\R^+} y^q \, \nu(\d y) \right)^{\frac{1}{q}} < \infty.
    \]

    \paragraph{Step 1}
    Suppose first that $\lambda \leq 1$.
    Then, by Markov's inequality,
    it holds that
    \[
        \nu \bigl([0, 2]\bigr)
        = 1 - \int_{\R^+\setminus [0, 2]} \, \nu(\d y)
        \geq 1 - \frac{1}{2^q} \int_{\R^+} y^q \, \nu(\d y)
        \geq 1 - \frac{1}{2^q}.
    \]
    Using this inequality and the bound $y^p \e^{-\beta y} \leq (p \beta^{-1})^p \e^{-p}$,
    which is valid for all $y \in \R^+$ and obtained by maximization of $y \mapsto y^p \e^{-\beta y}$,
    we have that
    \[
        \frac
        {\int_{\R^+} y^p \, \e^{-\beta y} \, \nu(\d y)}
        {\int_{\R^+} \e^{-\beta y} \, \nu(\d y)}
        \leq
        \frac
        {\int_{\R^+} (p \beta^{-1})^p \e^{-p} \, \nu(\d y)}
        {\int_{[0, 2]} \e^{-\beta y} \, \nu(\d y)}
        \leq
        \left(1 - \frac{1}{2^q}\right)^{-1}  (p \beta^{-1})^p \e^{2 \beta -p}.
    \]
    It follows that~\eqref{eq:aux_aux} is satisfied in this case.

    \paragraph{Step 2}
    Assume now that $\lambda \geq 1$,
    and let $\pi$ denote the pushforward of~$\nu$ under the contraction~$x \mapsto \lambda^{-1} x$.
    Then, applying a change of variable,
    using \textbf{Step 3} and using \textbf{Step 1} step,
    we have that
    \begin{align}
        \notag
        \frac
        {\int_{\R^+} y^p \, \e^{- \beta y} \, \nu (\d y)}
        {\int_{\R^+} \e^{- \beta y} \, \nu (\d y)}
         & = \lambda^p
        \frac
        {\int_{\R^+} z^p \, \e^{- \beta \lambda z} \, \pi (\d z)}
        {\int_{\R^+} \e^{- \beta \lambda z} \, \pi (\d z)} \\
        \label{eq:where_we_used_ineq}
         & \leq
        \lambda^p
        \frac
        {\int_{\R^+} z^p \, \e^{- \beta z} \, \pi (\d z)}
        {\int_{\R^+} \e^{- \beta z} \, \pi (\d z)}
        \leq \lambda^p C(p,q,\beta),
    \end{align}
    which is the required inequality.

    \paragraph{Step 3}
    In the first inequality of~\eqref{eq:where_we_used_ineq},
    we used that,
    for any increasing function $g \colon \R^+ \to \R$,
    the function
    \[
        F(\gamma) \coloneq \frac
        {\int_{\R^+} g(z) \, \e^{- \gamma z} \, \pi (\d z)}
        {\int_{\R^+} \e^{- \gamma z} \, \pi (\d z)}
    \]
    is a non-increasing function of~$\gamma$.
    This follows from~\cite[Lemma A.3]{CBS-Carrillo2021},
    and could also be proved by showing that~$F'(\gamma)$ is non-positive for all $\gamma \geq 0$.
    A simple calculation indeed gives that
    \[
        F'(\gamma) =
        - \int_{\R^+} g(z) \left(z - \overline z\right) \, L_{\gamma} \pi(\d z),
        \qquad L_{\gamma} \pi(\d z) \coloneq \frac{\e^{- \gamma z} \pi(\d z)}{\int_{\R^+} \e^{- \gamma z} \, \pi(\d z)},
        \qquad \overline z \coloneq \int_{\R^+} z \, L_{\gamma} \pi(\d z).
    \]
    By decomposing the integration domain into the union $[0, \overline z) \cup \{\overline z\} \cup (\overline z, \infty)$,
    and using the fact that $g$ is non-decreasing,
    we have
    \begin{align*}
        F'(\gamma)
         & = - \int_{[0,\overline z)} g(z) \left( z - \overline z\right) \, \pi(\d z) - \int_{(\overline z,\infty)} g(z) \left( z - \overline z\right) \, L_{\gamma} \pi(\d z) \\
         & \leq - g(\overline z)  \int_{[0,\overline z) \cup (\overline z, \infty)} \left( z - \overline z\right) \, L_{\gamma} \pi(\d z)  = 0,
    \end{align*}
    which completes the proof.
\end{proof}

    Armed with this result,
    we are ready to prove the following generalization of~\cite[Lemma 3.3]{carrillo2018analytical}.

\begin{proposition}
    \label{lemma:bound-on-weighted-moment}
    Suppose that~$f$ is bounded from below by $f_{\star}$ and that
    there are real numbers~$u \geq \ell > 0$ and positive constants
    $c_u, C_u, c_{\ell}, C_{\ell}$ such that
    \begin{equation}
        \label{eq:assumptions_for_bound_weighted}
        \forall x \in \R^d, \qquad
        \left\{
        \begin{aligned}
            f(x) - f_{\star} \le c_u |x|^{u}+ C_{u} \\
            f(x) - f_{\star} \ge c_{\ell} |x|^{\ell} - C_{\ell}.
        \end{aligned}
        \right.
    \end{equation}
    Then for all $p, q > 0$
    there are constants $C_1$ and $C_2$ depending on~$(p,q,u,\ell,c_u,C_u,c_{\ell},C_{\ell})$ such that
    \begin{equation}
        \label{eq:statement_bound_weighted_moment}
        \forall \mu \in \mathcal P_q(\R^d), \qquad
        \frac
        {\int_{\R^d} \myabs{x}^p \, \e^{-\beta f(x)} \, \mu(\d x)}
        {\int_{\R^d} \e^{-\beta f(x)} \, \mu(\d x)}
        \leq
        \left( C_1 + C_2 \int_{\R^d} \myabs{x}^q \, \mu(\d x) \right)^{\frac{pu}{q\ell}}
    \end{equation}
    Furthermore, if $u = \ell$ and $p \leq q$,
    then this statement holds with $C_1 = 0$.
\end{proposition}

\begin{proof}
    We assume without loss of generality that $f_{\star} = 0$.
    By the second inequality in~\eqref{eq:assumptions_for_bound_weighted}, we have that
    \begin{equation}
        \label{eq:first_inequality_bound_f}
        \myabs{x}^p = \left( \myabs{x}^\ell \right)^{\frac{p}{\ell}}
        \leq  \left( \frac{f(x) + C_{\ell}}{c_{\ell}}  \right)^{\frac{p}{\ell}}.
    \end{equation}
    For convenience,
    we denote by~$C$ any constant depending only on~$(p,q,\ell,u,c_{\ell},C_{\ell},c_u, C_u)$.
    From~\eqref{eq:first_inequality_bound_f} we deduce that
    \begin{align*}
        \frac
        {\int_{\R^d} \myabs{x}^p \, \e^{-\beta f(x)} \, \mu(\d x)}
        {\int_{\R^d} \e^{-\beta f(x)} \, \mu(\d x)}
         & \leq
        C
        \frac
        {\int_{\R^d} \bigl(f(x) + 1  \bigr)^{\frac{p}{\ell}} \, \e^{-\beta f(x)} \, \mu(\d x)}
        {\int_{\R^d} \e^{-\beta f(x)} \, \mu(\d x)} \\
         & =
        C
        \frac
        {\int_{\R^+} \bigl(y + 1  \bigr)^{\frac{p}{\ell}} \, \e^{-\beta y} \, f_\sharp \mu(\d y)}
        {\int_{\R^+} \e^{-\beta y} \, f_\sharp \mu(\d y)},
    \end{align*}
    where $f_{\sharp} \mu \in \mathcal P(\R^+)$ is the pushforward of $\mu$ by~$f$.
    Using~\cref{lemma:auxiliary_auxiliary_lemma}
    and then the first inequality in~\eqref{eq:assumptions_for_bound_weighted},
    we obtain
    \begin{align*}
        \frac
        {\int_{\R^d} \myabs{x}^p \, \e^{-\beta f(x)} \, \mu(\d x)}
        {\int_{\R^d} \e^{-\beta f(x)} \, \mu(\d x)}
         & \leq C \left( 1 + \int_{\R^+} \myabs{y}^{\frac{q}{u}} \, f_{\sharp} \mu(\d y) \right)^{\frac{p u}{q \ell}}       \\
         & = C \left( 1 + \int_{\R^d} \bigl\lvert f(x) \bigr\rvert^{\frac{q}{u}} \, \mu (\d x) \right)^{\frac{p u}{q \ell}}
        \leq C \left( 1 + \int_{\R^d} \bigl\lvert x \bigr\rvert^{q} \, \mu (\d x) \right)^{\frac{p u}{q \ell}},
    \end{align*}
    which completes the proof of~\eqref{eq:statement_bound_weighted_moment}.

    \paragraph{Proof that~\eqref{eq:statement_bound_weighted_moment} holds with $C_1 = 0$ when $\ell = u$ and $q \geq p$}
    The case $q > p$ follows from the case~$q = p$,
    so we assume the latter.
    Suppose for contradiction that this was not true.
    Then there would exist a sequence~$\{\mu_i\}_{i \in \N}$ of probability measures such that
    \begin{equation}
        \label{eq:contradiction_auxiliary_bound_weighted}
        \frac
        {\int_{\R^d} \myabs{x}^p \, \e^{-\beta f(x)} \, \mu_i(\d x)}
        {\int_{\R^d} \e^{-\beta f(x)} \, \mu_i(\d x)}
        \geq
        i \int_{\R^d} \myabs{x}^{p} \, \mu_i(\d x).
    \end{equation}
    By the first part of the proof,
    this inequality implies that
    \[
        \frac{C_1 + C_2 \int_{\R^d} \myabs{x}^{p} \, \mu_i(\d x)}{\int_{\R^d} \myabs{x}^{p} \, \mu_i(\d x)} \geq i,
    \]
    and so the integral in this expression necessarily converges to~0 in the limit as~$i \to \infty$.
    In turn, this implies that the sequence $\{\mu_i\}_{i \in \N}$ converges in distribution to the Dirac measure at~$0$ when $i \to \infty$.
    But then it holds that
    \[
        \frac
        {\int_{\R^d} \myabs{x}^p \, \e^{-\beta f(x)} \, \mu_i(\d x)}
        {\int_{\R^d} \e^{-\beta f(x)} \, \mu_i(\d x)}
        \left( \int_{\R^d} \myabs{x}^{p} \, \mu_i(\d x) \right)^{-1}
        \leq
        \frac{\e^{- \beta f_{\star}}}{\int_{\R^d} \e^{-\beta f(x)} \, \mu_i(\d x)}
        \xrightarrow[i \to \infty]{} \e^{-\beta \bigl(f_{\star} - f(0)\bigr)},
    \]
    which contradicts~\eqref{eq:contradiction_auxiliary_bound_weighted}.
\end{proof}
\begin{remark}
    When $p = q > 0$, \cref{lemma:bound-on-weighted-moment} implies that
    \begin{equation}
        \label{eq:contradiction}
        \forall \mu \in \mathcal P_p(\R^d), \qquad
        \frac
        {\int_{\R^d} \myabs{x}^p \, \e^{-\beta f(x)} \, \mu(\d x)}
        {\int_{\R^d} \e^{-\beta f(x)} \, \mu(\d x)}
        \leq \left( C_1 + C_2 \int_{\R^d} \myabs{x}^{p} \, \mu(\d x) \right)^{\frac{u}{\ell}}.
    \end{equation}
    We motivate in this remark that the exponent $\frac{u}{\ell}$ is optimal.
    Indeed, let $p = 1$ and consider the function~$f\colon \R^2 \to \R$ given by $f(x_1, x_2) = \lvert x_1 \rvert^\ell + \lvert x_2 \rvert^{u+\varepsilon}$,
    for any $u \geq \ell > 0$ and $\varepsilon > 0$.
    Then for the probability measures
    \[
        \mu_i = \frac{1}{i^{u}} \delta_{\left(i^u, 0\right)} + \left(1 - \frac{1}{i^{u}} \right) \delta_{\left(0,i^\ell\right)},
    \]
    it holds that
    \[
        \frac
        {\int_{\R^d} \myabs{x} \, \e^{-\beta f(x)} \, \mu_i(\d x)}
        {\int_{\R^d} \e^{-\beta f(x)} \, \mu_i(\d x)} \sim i^u
        \qquad
        \text{ and }
        \qquad
        \int_{\R^2} \myabs{x} \, \mu_i(\d x) \sim i^\ell
        \qquad
        \text{ as $i \to \infty$},
    \]
    where $\sim$ signifies asymptotic equivalence.
    This example suggests that~\cref{lemma:auxiliary_auxiliary_lemma} is sharp.
\end{remark}

\subsection{Generalized Wasserstein stability estimates}
By adapting the proof of~\cref{lemma:stab:wmean-wcov-basic},
it is simple to prove Wasserstein stability estimates with different values of~$p$ on the left and right sides of the inequality.
More precisely, we have the following statement.
\begin{lemma}
    \label{lemma:nonsym_stab_estimates}
    Suppose that $f \in \mathcal A(s, 0, u)$.
    Then
    \begin{itemize}
        \item
              for all $R > 0$ and $p > 1$, there is $C = C(p, R)$ such that
              \begin{equation}
                  \label{eq:modif_stab_M}
                  \forall (\mu, \nu) \in \mathcal P_{\varrho,R}\bigl(\real^d\bigr) \times \mathcal P_{\varrho,R}\bigl(\real^d\bigr), \qquad
                  \myabs{\wm(\mu) - \wm(\nu)}
                  \le C \, \wasserstein_p(\mu, \nu), \qquad
                  \varrho \coloneq \frac{(s+1)p}{p-1}
              \end{equation}
        \item
              for all $R > 0$ and $p > 1$, there is $C = C(p, R)$ such that
              \begin{equation}
                  \label{eq:modif_stab_C}
                  \forall (\mu, \nu) \in \mathcal P_{\varrho,R}\bigl(\real^d\bigr) \times \mathcal P_{\varrho,R}\bigl(\real^d\bigr), \qquad
                  \myabs{\wc(\mu) - \wc(\nu)}
                  \le C \, \wasserstein_p(\mu, \nu), \qquad
                  \varrho \coloneq \frac{(s+2)p}{p-1}
              \end{equation}
    \end{itemize}
\end{lemma}
\begin{remark}
    For $p = s+2$, the bound~\eqref{eq:modif_stab_M} coincides with~\eqref{eq:stab_wmean_simple},
    and for $p = s+3$, the bound~\eqref{eq:modif_stab_C} coincides with~\eqref{eq:stab_wcov_simple}.
\end{remark}

\paragraph{Acknowledgments}
We are grateful to Yunan Yang, Arnaud Guyader and Louis-Pierre Chaintron for useful discussions,
to Hui Huang and Hicham Kouhkouh for pointing out an issue with Assumption \cref{assump:f-at-infinity},
and to the anonymous reviewers for very useful comments and suggestions.
We also thank the Lorentz Center (Leiden, the Netherlands) for hosting the workshop ``Purpose-driven particle systems'' in March 2023,
which enabled discussions about this work.
FH is supported by start-up funds at the California Institute of Technology and by NSF CAREER Award 2340762. FH was also
supported by the Deutsche Forschungsgemeinschaft (DFG, German Research Foundation) via
project 390685813 - GZ 2047/1 - HCM.
UV is partially supported by the European Research Council (ERC) under the EU Horizon 2020 programme (grant agreement No 810367),
and by the Agence Nationale de la Recherche under grants ANR-21-CE40-0006 (SINEQ) and ANR-23-CE40-0027 (IPSO).

\printbibliography
\end{document}